\documentclass[11pt,a4paper]{article}

\usepackage{psfrag}
\usepackage{verbatim}
\usepackage{amsmath,amssymb,cite}
\usepackage{latexsym}
\usepackage{graphicx}
\usepackage{lscape}
\usepackage{afterpage}
\usepackage[color=green!40]{todonotes}
\usepackage[normalem]{ulem}

\definecolor{darkolivegreen}{rgb}{0.43, 0.52, 0.28}

\usepackage{hyperref}
\hypersetup{
    colorlinks=true,
    linkcolor=blue,
    filecolor=magenta,      
    urlcolor=cyan,
    citecolor=red,
}
\DeclareMathOperator*{\argmin}{argmin}
\DeclareMathOperator*{\sgn}{sgn}

\newcommand{\ds}{\displaystyle}
\newcommand{\nexto}{\kern -0.54em}
\newcommand{\dR}{{\rm {I\ \nexto R}}}

\newcommand{\dZ}{{\cal Z \kern -0.7em Z}}
\newcommand{\dC}{{\rm\hbox{C \kern-0.8em\raise0.2ex\hbox{\vrule
height5.4pt width0.7pt}}}}
\newcommand{\dQ}{{\rm\hbox{Q \kern-0.85em\raise0.25ex\hbox{\vrule
height5.4pt width0.7pt}}}}
\newcommand{\proofbox}{\hspace{\fill}{$\Box$}}
\newtheorem{lemma}{Lemma}
\newtheorem{theorem}[lemma]{Theorem}
\newtheorem{corollary}[lemma]{Corollary}
\newtheorem{proposition}[lemma]{Proposition}

\newtheorem{remark}[lemma]{Remark}
\newtheorem{assumption}[lemma]{Assumption}

\newenvironment{proof}{Proof.}{\proofbox}

\DeclareMathOperator{\sinc}{sinc}

\begin{document}

\author{
C. Yal{\c c}{\i}n Kaya\footnote{School of Mathematical Sciences, Adelaide University, Adelaide SA 5005, Australia. \\ E-mail: yalcin.kaya@adelaide.edu.au\,.}
\and
Lyle Noakes\footnote{Department of Mathematics and Statistics, The University of Western Australia, Nedlands WA 6009, Australia. Email: Lyle.Noakes@uwa.edu.au\,.}
\and
Philip Schrader\footnote{School of Mathematics, Statistics, Chemistry and Physics, Murdoch University, Murdoch WA 6150, Australia. Email: Phil.Schrader@murdoch.edu.au\,.}
}

\title{\vspace{0mm}\bf Curves of Minimax Curvature}

\date{June 20, 2026}

\maketitle

\vspace*{-5mm}
\begin{abstract} 
We consider the problem of finding curves of minimum pointwise-maximum curvature, i.e., curves of minimax curvature, among planar curves of fixed length with prescribed endpoints and tangents at the endpoints. 
We reformulate the problem in terms of optimal control and use the maximum principle, as well as some geometrical arguments, to produce a classification of the types of solutions. Using the classification,  we devise a numerical method which reduces the infinite-dimensional optimization problem to a finite-dimensional problem with just six variables. The solution types, together with some further observations on optimality,  are illustrated via numerical examples.
\end{abstract}

\begin{verse} {\bf Key words.} {\sf Minimax curvature, Markov--Dubins path, Optimal control, Singular control, Bang--bang control.}
\end{verse}

\begin{verse} {\bf AMS subject classifications.} {\sf Primary 49J15, 49K15\ \ Secondary 65K10, 90C30}
\end{verse}

\pagestyle{myheadings}
\thispagestyle{plain}
\markboth{\sf\scriptsize C.~Y.~Kaya, L.~Noakes, and P.~Schrader}{\sf\scriptsize Curves of Minimax Curvature\ \ by\ C.~Y.~Kaya, L.~Noakes, and P.~Schrader}

\section{Introduction}

We are interested in finding a curve $z:[0,t_f]\longrightarrow\dR^2$ which minimizes the almost everywhere~(a.e.) pointwise maximum curvature of $z(t)$, i.e. the $L^\infty$-norm of the curvature, among all curves having fixed length; prescribed endpoints $z_0$ and $z_f$ at $0$ and $t_f$ respectively; and prescribed velocities $v_0$ and $v_f$ at the endpoints. Since the curvature is independent of parametrization we will consider only curves which are parametrized with respect to arc length so that the curvature is $\Vert \ddot z(t) \Vert$, and the parameter $t_f$ is the length of the whole curve (and therefore $t_f$ is fixed). The problem can then be posed as
\[
\mbox{(P)}\left\{\begin{array}{rl}
\ds\min_{z(\cdot)} &\ \ds\max_{t\in[0,t_f]}\|\ddot {z}(t)\| \\[4mm]
\mbox{s.t.} &\ z(0) = z_0\,,\ z(t_f) = z_f\,,\\[2mm]
  &\ \dot{z}(0) = v_0\,,\ \dot{z}(t_f) = v_f\,,\\[2mm]
   &\ \|\dot{z}(t)\| = 1\,,\mbox{ for a.e. }  t\in[0,t_f]\,,
\end{array}\right.
\]
where $\dot{z} = dz/dt$, $\ddot{z} = d^2z/dt^2$, $\|\cdot\|$ is the Euclidean norm, and $v_0$ and $v_f$ are given such that $\|v_0\| = \|v_f\| = 1$.  

If the $L^\infty$-norm is replaced by the squared $L^2$-norm, of $\ddot {z}$, namely $\int_0^{t_f} \|\ddot {z}(t)\|^2\,dt$, then  critical curves for Problem~(P) become the celebrated (fixed-length) {\em Euler's elastica}.  The $L^p$-norm generalizations of the Euler elastica problem for $p\in[1,\infty)$, where $\int_0^{t_f} \|\ddot {z}(t)\|^2\,dt$ is replaced by $\int_0^{t_f} \|\ddot {z}(t)\|^p\,dt$, have been referred to as the {\em $p$-elastica} problem in the literature---see~\cite{Huang2004}.  Similarly, potential minimizers of Problem~(P) have been called the {\em $\infty$-elastica} in~\cite{moser2022}. Here we refer to (P) as the problem of finding {\em curves of minimum pointwise-maximum curvature}, or simply {\em curves of minimax curvature}, because it is more descriptive and we are ultimately focused on finding global minimizers.

Another problem related to (P) is the Markov--Dubins problem, in which one is to minimize this time the (free) length $t_f$ of the curve of interest subject to a fixed bound $a$ imposed on the curvature, while the remaining constraints appearing in (P) are kept in place.  This problem was first proposed and some instances studied by Andrey Markov in 1889~\cite{Markov1889} (also see~\cite{KreNud1977}) but fully solved by Lester Dubins only in 1957~\cite{Dubins1957}.  Dubins' result characterizes minimum-length solutions in terms of concatenations of  straight lines and circular arcs: Let a straight line segment be denoted by an $S$ and a circular arc of the maximum allowed curvature $a$ (or, turning radius $1/a$) by a $C$, then the shortest curve is a concatenation of type $CCC$, or of type $CSC$, or a subset thereof. 

The Markov--Dubins path can be viewed as the shortest path of a car modelled as a point mass which goes only forwards at unit speed under the constraints described above. In 1990, Reeds and Shepp \cite{ReeShe1990} considered a car which goes not only forwards but also backwards, as an extension of the Markov–Dubins problem.  In 1991, Boissonnat, Cérézo and Leblond \cite{BoiCerLeb1991} (also see \cite{BoiCerLeb1994}), and Sussmann and Tang \cite{SusTan1991}, independently employed optimal control theory, and additional techniques, to derive the same results as those obtained by Reeds and Shepp.  They also recovered Dubins' result as a special case of the Reeds and Shepp model.  There have since been many extensions and applications of the Markov--Dubins problem---see, for example, \cite{BuiBoiSouLau1994, SouLau1996, Sussmann1995, Sussmann1997, Kaya2019}.  More recently, optimal control theory has also been employed in \cite{Kaya2017, Kaya2019} to derive the existing, as well as some new, results and to devise numerical methods to compute the Markov--Dubins paths---the reader is referred to these two articles for a more in-depth literature survey on the Markov--Dubins problem.

The first published work on Problem (P) appears to be that of Moser \cite{moser2022} where it is studied for curves in $\mathbb R^n$, followed by the extension to Riemannian manifolds in \cite{gallagher2023elastica}. In \cite{moser2022} $\infty$-elastica are defined as minimizers of the sum of maximum curvature and a penalty term which is a kind of $L^2$-distance between curves. In order to obtain differential equations which characterize $\infty$-elastica, Moser studies the Euler--Lagrange equations for the sum of $p$-elastic energy with the same penalty term, in the limit $p\to \infty$. From the differential equations Moser obtains a classification of $\infty$-elastica as concatenations of circular arcs and straight lines, similar to the solutions of the Markov--Dubins problem, except that $\infty$-elastica can contain more than three pieces, and closed loops are allowed. If we denote by $O$ a full circle with the maximum curvature, and $C,S$ as before, the conclusion is that $\infty$-elastica can be of type $CSOSOC$, perhaps with more or less loops and straight line segments, or of type $CC\ldots C$ with any number of $C$ segments and such that the interior segments have equal length and alternating orientation. 

The present work takes a fundamentally different approach to that of Moser: we show that by using the same technique applied in \cite{Kaya2017} and \cite{KayNoa2013}, Problem (P) can be reformulated as an optimal control problem without any approximation by $p$-elastic energy. As usual, the maximum principle then provides us with necessary conditions for optimality in the form of differential equations for the state and costate variables, and constraints on the control. A careful analysis of these conditions -- in particular of the phase portrait of the switching function, which is very similar to the analysis of the Markov--Dubins problem in \cite{Kaya2017, Kaya2019} -- allows us to draw similar conclusions to Moser's about which types of curves  satisfy the necessary condition.

In the present paper, we are not only interested in classifying critical curves, we would also like to know which of the critical curves are {\em optimal} as well as devise an efficient numerical method to compute these curves. To this end, the necessary conditions obtained from the maximum principle are combined with a series of geometric arguments, for example Lemma \ref{CCCC} shows that a curve of type $CCCC$ cannot be optimal because  it can be deformed into a feasible curve with the same maximum curvature, but which does not satisfy the necessary conditions. These observations lead to our main result, a refined classification in Theorem~\ref{classification}:  any solution to Problem~(P) is of type $CCC$, or $CSC$, or $COC$, or $SOS$, or a subset thereof, with all circular arcs having the same maximum possible radius (or minimum possible maximum curvature). So we have the new result that, just like in the Markov--Dubins problem, solutions of problem (P) cannot consist of more than three segments. However, unlike in the Markov--Dubins problem, under specific conditions (Propositions \ref{SCS} and \ref{COC}), they can contain closed circular loops.

Since Problem (P) is related to the Markov--Dubins problem by interchange of constraint and cost, it is perhaps not surprising that the solutions are so similar. Indeed, a relationship between solutions of the two problems was established in \cite[Proposition 5]{moser2022}: a solution of the Markov--Dubins problem is a solution to Problem (P). In Theorem \ref{theo:MD&P} we give an alternative proof of this fact in terms of the optimal control formulation, and then demonstrate that the converse does not hold (Remark \ref{rem:MD&P}). 

We devise a numerical method, which, using the classification of the solution curves stated in Theorem~\ref{classification} and switching time optimization techniques, reduces the infinite-dimensional optimization problem~(P) to a finite-dimensional optimization problem~(Ps) with just six variables.  We carry out extensive numerical experiments.  We especially focus on solutions containing a loop to exemplify the specific results involving loops, as well as gain insights about those instances for which no specific results are available.

It is well known that some optimal control problems can produce solutions with {\em chattering arcs}, i.e., optimal control with infinitely many switchings, as in
Fuller’s phenomenon~\cite{Borisov2000}.  We show through Proposition~\ref{prop:GLC} and Remark~\ref{rem:chatter1} that Problem~(P) will not generate chattering arcs.  The authors of this paper have recently studied in~\cite{KayNoaSch2025} another problem related to (P), in that instead of minimizing the $L^\infty$-norm of the curvature, they considered minimizing the $L^\infty$-norm of the arc-length derivative of the curvature, resulting in the so-called {\em curves of minimax spirality}.  In Remark~\ref{rem:chatter2}, we discuss further that the minimax spirality problem can give rise to chattering arcs.  Numerical evidence to chattering arcs is presented in~\cite{KayNoaSch2025}.

In~\cite{KayNoaSch2025}, it is shown that a curve of minimax spirality is in general a concatenation of Euler spirals, unless the oriented endpoints are placed on a circle or along a line in geometrical agreement.  However, the maximum number of such (spiral) arcs needed for a solution cannot be prescribed, as opposed to Problem~(P) where a selection of just three circular arcs and straight lines is needed in total.  So, obtaining a numerical solution for the curves of minimax curvature is much easier than for the curves of minimax spirality.

The paper is organized as follows.  In Section~\ref{sec:reformulations}, we reformulate Problem~(P) as an optimal control problem. In Section~\ref{sec:max_principle}, we write down the maximum principle for the optimal control problem.  In Section~\ref{sec:classification}, we classify the types of solution for Problem~(P) and provide the main theoretical result as Theorem~\ref{classification}.  In Section~\ref{num_meth}, we develop a numerical method to solve Problem~(P), and provide examples of applications and numerical solutions using standard optimization software to illustrate the theoretical results.  Finally, in Section~\ref{sec:conclusion}, concluding remarks and possible directions for future work are provided.

\section{Reformulations}
\label{sec:reformulations}

Obviously, if the distance separating $z_0$ and $z_f$ is more than the fixed length $t_f$ then Problem~(P) has no solution.  In fact, with  $t_f = \|z_0-z_f\|$, a feasible solution may still not exist.  Consider the points $z_0 = (0,0)$ and $z_f = (1,0)$.  With $t_f = 1$, unless $v_0 = v_f = (1,0)$, there exists no feasible curve between $z_0$ and $z_f$. Therefore we make the following assumption.

\begin{assumption}\label{ass1}
$t_f > \|z_0-z_f\|$.
\end{assumption}

Problem~(P) is nonsmooth because of the $\max$ operator appearing in its objective functional.  On the other hand, Problem~(P) can be transformed into a variational problem for a smooth functional by using a standard technique from nonlinear programming, in the same way it was also done in \cite{KayNoa2013}:
\[
\mbox{(P1)}\left\{\begin{array}{rl}
\ds\min_{a, z(\cdot)} &\ a \\[2mm]
\mbox{s.t.} &\ z(0) = z_0\,,\ z(t_f) = z_f\,,\\[2mm]
  &\ \dot{z}(0) = v_0\,,\ \dot{z}(t_f) = v_f\,,\\[2mm]
   &\ \|\ddot {z}(t)\|\le a\,,\ \ \ \|\dot{z}(t)\| = 1\,,\mbox{ for a.e. }  t\in[0,t_f]\,,
\end{array}\right.
\]
where the bound $a \ge 0$ is a new optimization variable of the problem.

\begin{remark}  \label{rem:observation} \rm
One has the solution $a = 0$ if and only if the curve joining $z_0$ and $z_f$ is a straight line of length $t_f$, contradicting Assumption \ref{ass1}. So $a>0$.
\proofbox
\end{remark}

Problem~(P) can equivalently be cast as an optimal control problem as in \cite{Kaya2017} as follows. Let\linebreak $z(t) := (x(t), y(t))\in\dR^2$, with $\dot{x}(t) := \cos\theta(t)$ and $\dot{y}(t) := \sin\theta(t)$, where $\theta(t)$ is the angle the velocity vector $\dot{z}(t)$ of the curve $z(t)$ makes with the horizontal.  These definitions verify that $\|\dot{z}(t)\| = 1$.  Moreover,
\[
\|\ddot{z}\|^2 = \ddot{x}^2+\ddot{y}^2 = \dot{\theta}^2\,.
\]
Therefore, $|\dot{\theta}(t)|$ is nothing but the curvature.  In fact, $\dot{\theta}(t)$ itself, which can be positive or negative, is referred to as the {\em signed curvature}.  For example, consider a vehicle travelling along a circular path.  If $\dot{\theta}(t) > 0$ then the vehicle travels in the counter-clockwise direction, i.e., it {\em turns left}, and if $\dot{\theta}(t) < 0$ then the vehicle travels in the clockwise direction, i.e., it {\em turns right}.  If $\dot{\theta}(t) = 0$ then the vehicle travels along a straight line.

Suppose that the directions at the points $z_0$ and $z_f$ are denoted by the angles $\theta_0$ and $\theta_f$, respectively.  The curvature minimizing problem~(P1), or equivalently Problem~(P), can then be re-written as a ({\em parametric}) {\em optimal control problem}, where the objective functional is the maximum curvature $a$ (the {\em parameter}), and $x$, $y$ and $\theta$ are the {\em state variables} and $u$ is the {\em control variable}\,:
\[
\mbox{(Pc)}\left\{\begin{array}{rll}
\ds\min_{a, u(\cdot)} &\ \ds a & \\[2mm]
\mbox{s.t.} &\ \dot{x}(t) = \cos\theta(t)\,, & x(0) = x_0\,,\ 
              x(t_f) = x_f\,, \\[2mm] 
  &\ \dot{y}(t) = \sin\theta(t)\,, & y(0) = y_0\,,\ 
              y(t_f) = y_f\,,\\[2mm] 
  &\ \dot{\theta}(t) = u(t)\,, & \theta(0) = \theta_0\,,\ 
              \theta(t_f) = \theta_f\,,\\[2mm]
  & & |u(t)|\le a\,,\mbox{ for a.e. }  t\in[0,t_f]\,.
\end{array}\right.
\]

Note that with the unit speed constraint $\|\dot{z}(t)\| = 1$ in Problem~(P) the solution curve $z(t)$ is parametrized by arc length.  Therefore the terminal time $t_f$ is nothing but the length of the curve $z(t)$.  When the maximum allowed curvature $a$ is fixed, the length $t_f$ is allowed to vary, and the objective functional $a$ in Problem~(Pc) is replaced by $t_f$, Problem~(Pc) becomes the celebrated Markov--Dubins problem~\cite{Dubins1957,KreNud1977,Markov1889}, where one looks for the shortest curve between the oriented points $(x_0,y_0,\theta_0)$ and $(x_f,y_f,\theta_f)$, with a prescribed bound $a$ on its curvature:
\[
\mbox{(MD)}\left\{\begin{array}{rll}
\ds\min_{t_f, u(\cdot)} &\ \ds t_f & \\[2mm]
\mbox{s.t.} & \mbox{Constraints in Problem~(Pc) with $a$ fixed}.
\end{array}\right.
\]
Let a straight line segment be denoted by an $S$ and a circular arc segment of curvature $a$ (or, turning radius $1/a$) by a $C$.  Dubins' result is stated in terms of concatenations of type $S$ and type $C$ curve segments in the following theorem.
\begin{theorem}[Dubins' Theorem~\cite{Dubins1957}]  \label{Dubins}
A solution to Problem~{\em (MD)} exists, and any such solution, that is to say, any $C^1$ and piecewise-$C^2$ curve of minimum length in the plane between two prescribed endpoints, where the slopes of the curve at these endpoints as well as a bound on the pointwise curvature are also prescribed, is of type $CSC$, or of type $CCC$, or a subset thereof.
\end{theorem}

In the following theorem, we establish a relationship between Problems~(MD) and (Pc).  In fact, the proof of the theorem utilizes Theorem~\ref{Dubins} as well as the following simple observation: Problem~(MD) minimizes the length $t_f$ while the bound $a$ on the curvature is fixed, and Problem~(Pc) minimizes $a$ while $t_f$ is kept fixed.

\begin{theorem}[Solutions of (MD) and (Pc)]  \label{theo:MD&P}
Any solution $t_f^*$ to Problem {\em (MD)} with $a = a^* >0$ fixed is a solution $a^*$ to Problem {\em (Pc)} with $t_f = t_f^*$ fixed.
\end{theorem}

\begin{proof}
Suppose that, with $a = a^*$ fixed, the solution of Problem~(MD) yields $t_f^*$.  By Theorem~\ref{Dubins}, either $|\dot{\theta}^*(t)| = a^*$ or $\dot{\theta}^*(t) = 0$, a.e.\ $t \in [0,t_f^*]$.  In other words, with the optimal length $t_f^*$, the optimal angular velocity $\dot{\theta}^*$ has to satisfy
\begin{equation}  \label{opt_thetadot}
|\dot{\theta}^*(t)| \notin (0,a^*)\,,\ \ \mbox{a.e.}\ \ t \in [0,t_f^*].  
\end{equation}    
Now suppose that, with $t_f = t_f^*$ fixed, solution of Problem~(Pc) yields $\widehat{a}$.  Obviously, $0 < \widehat{a} \le a^*$.  We can obtain a contradiction for the case when $\widehat{a} < a^*$ as follows. 
Since this solution of Problem~(Pc) satisfies the constraints of Problem~(MD) with $a=a^*$, and has the optimal length $t_f^*$ for Problem~(MD), it must also be a solution of~(MD). Then it must satisfy \eqref{opt_thetadot}.
However, since it is a solution of Problem~(Pc) it also satisfies $0 \le |\dot{\theta}(t)| \le \widehat{a}$ for a.e. $t \in [0,t_f^*]$, and so we cannot have $\widehat a <a^*$.
Therefore, with $t_f = t_f^*$ fixed, the solution of Problem~(Pc) yields $\widehat{a} = a^*$.  So, if the pair $(t_f^*, a^*)$ solves Problem~(MD), then it also solves Problem~(Pc).  
\end{proof}

\begin{remark} \rm  \label{rem:MD&P}
The converse of Theorem~\ref{theo:MD&P} is not true:  Consider the oriented points $(x_0,y_0,\theta_0) = (0,0,0)$ and $(x_f,y_f,\theta_f) = (1,0,0)$.  Fix $t_f = \widehat{t}_f = 2 > \|x_f - x_0\| = 1$.  Then clearly the solution $\widehat{a}$ to Problem~(Pc) is bounded away from 0, as $\widehat{a} = 0$ is the solution where $x_0$ and $x_f$ are joined by a straight line with $t_f = 1 < 2$.  On the other hand, with the curvature bound $a = \widehat{a}$ fixed, Problem~(MD) yields the solution $t_f^* = 1$, and $\theta^*(t) = 0$ for all $t\in[0,t_f^*]$ (the straight line joining $x_0$ and $x_f$).  So one gets $t_f^* < \widehat{t}_f$, furnishing a counterexample for the converse of Theorem~\ref{theo:MD&P}.
Various other counterexamples are suggested by the numerical experiments in Section \ref{ex:SCS} and Section \ref{ex:COC}, the results of which are illustrated in Figure~\ref{fig:SCS_ex} and  Figure~\ref{fig:COC_ex} respectively. 
\proofbox
\end{remark}

\section{Maximum Principle}
\label{sec:max_principle}

In this section, we study the optimality conditions for a re-formulation of Problem~(Pc) in a classical form. Let us redefine the control variable $u$ such that $\alpha(t)\,v(t) := \dot\theta(t)$, where $\alpha(t) := a > 0$ and so $\dot{\alpha}(t) := 0$, for a.e. $t\in[0,t_f]$.  Problem~(Pc), or equivalently Problem~(P), can then be re-written as the following {\em optimal control problem}:
\[
\mbox{(OC)}\left\{\begin{array}{rll}
\ds\min_{v(\cdot)} &\ \ds \int_0^{t_f} \alpha(t)\,dt  & \\[4mm]
\mbox{s.t.} &\ \dot{x}(t) = \cos\theta(t)\,, & x(0) = x_0\,,\ 
              x(t_f) = x_f\,, \\[2mm] 
  &\ \dot{y}(t) = \sin\theta(t)\,, & y(0) = y_0\,,\ 
              y(t_f) = y_f\,,\\[2mm] 
  &\ \dot{\theta}(t) = \alpha(t)\,v(t)\,, & \theta(0) = \theta_0\,,\ 
              \theta(t_f) = \theta_f\,,\ |v(t)|\le 1\,,\mbox{ for a.e. }  t\in[0,t_f]\,,\\[2mm]
  &\ \dot{\alpha}(t) = 0\,, & \mbox{for a.e. }  t\in[0,t_f]\,.
\end{array}\right.
\]
By Theorem~\ref{theo:MD&P}, a solution of Problem~(MD) is also a solution of Problem~(Pc), or equivalently Problem~(OC). However, the fact that the converse is not true (see Remark~\ref{rem:MD&P}) prompts us to address the existence of solutions of Problem~(OC), or equivalently Problem~(Pc) or (P).

\begin{proposition}[Existence of Minimizers] \label{existence}
    Given initial and final points $z_0,z_f$ and directions $v_0,v_f$ as in Problem (P), there exists a minimizer $z\in W^{2,\infty}((0,t_f),\mathbb R^2)$ which is a solution to Problem~(P). 
\end{proposition}
Conceivably, one can prove Proposition~\ref{existence} by using the existence results from the optimal control literature, by showing that a list of properties are satisfied: the controllability of the dynamical system, a growth condition on the dynamics that is affine in the control, and the compactness and convexity of the control set.  For various forms of optimal control existence results, see, for example, \cite{Filippov1962}, \cite[Theorem~23.10]{Clarke2013}, \cite[Chapter~IV, Theorem~3]{MacStr1982}, \cite[Theorem~2.9]{Trelat2024}.  In this paper, we provide a more self-contained proof of Proposition \ref{existence}.
Since the proof is somewhat long and does not contribute directly to our main result, we defer it until Section \ref{subsec:proof}.

Define the {\em Hamiltonian function} for Problem~(OC) as
\begin{equation}  \label{Hamiltonian}
H(x,y,\theta, \alpha, \lambda_0,\lambda_1,\lambda_2,\lambda_3,\lambda_4,v) := \lambda_0\,\alpha
+ \lambda_1\,\cos\theta + \lambda_2\,\sin\theta + \lambda_3\,\alpha\,v + \lambda_4 \cdot 0\,,
\end{equation}
where $\lambda_0 \ge 0$ is a scalar (multiplier) parameter and $\lambda_i: [0,t_f]\to\dR$, $i = 1,2,3,4$, are the {\em adjoint} (or {\em costate}) {\em variables}.  Let
\[
H[t] := H(x(t),y(t),\theta(t),\alpha(t),\lambda_0,\lambda_1(t),\lambda_2(t),\lambda_3(t),\lambda_4(t),v(t))\,. 
\]
The adjoint variables are required to satisfy
\begin{eqnarray}
&& \dot{\lambda}_1(t) := -H_x[t] = 0\,, \label{adjoint1} \\[1mm]
&& \dot{\lambda}_2(t) := -H_y[t] = 0\,, \label{adjoint2} \\[1mm]
&& \dot{\lambda}_3(t) := -H_\theta[t] = \lambda_1(t)\,\sin\theta(t) - \lambda_2(t)\,\cos\theta(t)\,, \label{adjoint3} \\[1mm]
&& \dot{\lambda}_4(t) := -H_\alpha[t] = -\lambda_0 - \lambda_3(t)\,v(t)\,,\ \ \lambda_4(0) = 0\,,\ 
 \lambda_4(t_f) = 0\,, \label{adjoint4} 
\end{eqnarray}
where $H_x = \partial H / \partial x$, etc.  By these definitions, the state and adjoint variables verify a Hamiltonian system in that, in addition to \eqref{adjoint1}--\eqref{adjoint4}, one has $\dot{x}(t) = H_{\lambda_1}[t]$, $\dot{y}(t) = H_{\lambda_2}[t]$, $\dot{\theta}(t) = H_{\lambda_3}[t]$ and $\dot{\alpha}(t) = H_{\lambda_4}[t]$. Since $H$ does not depend on $t$ explicitly,
\begin{equation}  \label{H_const}
H[t] = h\,,  
\end{equation}
where $h$ is some real constant.

The maximum principle \cite[Theorem 1]{PonBolGamMis1962} for Problem~(OC) can simply be stated as follows.  Suppose that $x,y,\theta,\alpha\in W^{1,\infty}(0,t_f;\dR)$ and $v\in L^\infty(0,t_f;\dR)$ solve Problem~(OC).  Then
there exist a number $\lambda_0\ge0$ and functions $\lambda_i\in W^{1,\infty}(0,t_f;\dR)$, $i=1,2,3,4$, such that $\lambda(t ):= (\lambda_0,\lambda_1(t), \lambda_2(t), \lambda_3(t), \lambda_4(t)) \neq \bf0$, for every $t\in[0,t_f]$, and, in addition to the state differential equations and other constraints given in Problem~(OC) and the adjoint
differential equations \eqref{adjoint1}--\eqref{adjoint2}, \eqref{adjoint4} and \eqref{adjoint3a}, the following condition holds:
\begin{equation}  \label{control} 
u(t)\in\argmin_{|w|\le 1} H(x(t),y(t),\theta(t),\alpha(t),\lambda_0,\lambda_1(t),\lambda_2(t),\lambda_3(t),\lambda_4(t),w)\,.
\end{equation}
Using the definition in \eqref{Hamiltonian}, \eqref{control} can be concisely written as
\begin{equation}  \label{control2}
v(t)\in\argmin_{|w|\le 1}\ a\,\lambda_3(t)\,w\,,
\end{equation}
which, since $a>0$, yields the optimal control as
\begin{equation}  \label{control3}
v(t) = \left\{\begin{array}{ll}
\ \ 1\,, & \mbox{if}\ \lambda_3(t) < 0\,, \\[3mm]
-1\,, & \mbox{if}\ \lambda_3(t) > 0\,, \\[3mm]
\mbox{undetermined}\,, & \mbox{if}\ \lambda_3(t) = 0\,.
\end{array}\right.
\end{equation} 

By \eqref{adjoint1},\eqref{adjoint2} we have $\lambda_1(t) = \overline{\lambda}_1$ and $\lambda_2(t) = \overline{\lambda}_2$ for all $t\in[0,t_f]$, where $\overline{\lambda}_1$ and $\overline{\lambda}_2$ are constants.

Consider first the case where $\bar \lambda _1=\bar \lambda _2=0$. Then $\lambda _3$ is constant. If $\lambda _3=0$ then $\lambda _0=0$ because otherwise $\lambda _4$ would be strictly monotonic, contradicting $\lambda _4(0)=\lambda _4(t_f)=0$. Then $\lambda _4$ would also be identically $0$, namely $\lambda _0=\lambda _1=\lambda _2=\lambda _3=\lambda_4=0$, contradicting the conditions for the maximum principle. So $\lambda _3\not=0$ and, by the maximum principle, $v(t)$ is either identically $-1$ or identically $1$, for a.e. $t\in [0,t_f]$. The solution is a single arc, with no switching. 

Having dealt with $\bar \lambda _1=\bar \lambda _2=0$, suppose from now on that the complex number $\bar \lambda _1+{\bf i}\bar \lambda _2$ is nonzero, with modulus $\rho >0$ and argument $\phi \in [0,2\pi )$.

Then \eqref{Hamiltonian} and \eqref{adjoint3} can respectively be re-written as
\begin{equation}  \label{Hamiltonian2}
H[t] = a\,\lambda_0 + \rho\,\cos(\theta(t) - \phi) + a\,\lambda_3(t)\,v(t)\,,
\end{equation}
where we have also used $\alpha(t) = a$, for all $t\in[0,t_f]$, and 
\begin{equation}  \label{adjoint3a}
\dot{\lambda}_3(t) = \rho\,\sin(\theta(t) - \phi)\,.
\end{equation}
Furthermore, \eqref{Hamiltonian2} and \eqref{H_const} give
\begin{equation}  \label{H_const2}
a\,\lambda_3(t)\,v(t) + \rho\,\cos(\theta(t) - \phi) + a\,\lambda_0 =: h\,.
\end{equation}

The control $v(t)$ to be chosen for the case when $\lambda_3(t)=0$ for a.e.\ $t\in[\zeta_1,\zeta_2]\subset[0,t_f]$ is referred to as {\em singular control}, because \eqref{control2} does not yield any further information.  On the other hand, when $\lambda_3(t)\neq 0$ for a.e.\ $t\in[0,t_f]$ the control $v(t)$ is said to be {\em nonsingular}.  It should be noted that, if $\lambda_3(\tau)=0$ only at an isolated point $\tau$, the optimal control at this isolated point can be chosen as $v(\tau) = -1$ or $v(\tau) = 1$, conveniently.  Therefore, if the control $v(t)$ is nonsingular, it will take on either the value $-1$ or $1$, the bounds on the control variable.  In this case, the control $v(t)$ is referred to as {\em bang--bang}.  Since the sign of $\lambda_3(t)$ determines the value of the optimal control $v(t)$, $\lambda_3$ is referred to as the {\em switching function}.

\begin{lemma}[Singularity and Straight Line Segments] \label{singular}
Suppose that the optimal control $v(t)$ for Problem~{\em (OC)} is singular over an interval $[\zeta_1,\zeta_2]\subset[0,t_f]$.  Then $\theta(t)$ is constant for all $t\in[\zeta_1,\zeta_2]$.  Moreover $v(t) = 0$ for all $t\in[\zeta_1,\zeta_2]$.
\end{lemma}
\begin{proof}
Suppose that the optimal control $v(t)$ is singular, i.e., $a\,\lambda_3(t)=0$, for a.e.\ $t\in[\zeta_1,\zeta_2]\subset[0,t_f]$.  Because $a>0$, $\dot\lambda_3(t)=0$ for a.e.\ $t\in[\zeta_1,\zeta_2]\subset[0,t_f]$. In other words, from \eqref{adjoint3a}, $\sin(\theta(t) - \phi) = 0$, which implies that $\theta(t)$ is constant, i.e., $\dot{\theta}(t) = u(t) = 0$, for all $t\in[\zeta_1,\zeta_2]$.
\end{proof}

Using Lemma~\ref{singular} we can rewrite the optimal control in \eqref{control3} as
\begin{equation} \label{u}
v(t) = -\sgn(\lambda_3(t))\,, \mbox{ a.e. } t\in[0,t_f]\,. 
\end{equation}
When $\lambda_3(t) = 0$ over an interval of time, $\theta(t)$ is constant, so the solution curve in that interval is a straight line.  When $\lambda_3(t) \neq 0$, i.e., $\lambda_3(t) \neq 0$, over an interval of time the solution curve in that interval is a circular arc segment with radius $1/a$.

By Lemma~\ref{singular}, we have established that the only straight-line solution of Problem~(OC) is a singular control solution.

The problems that yield a solution with $\lambda_0 = 0$ are referred to as {\em abnormal}, so are their solutions, in  optimal control theory, for which the conditions obtained from the maximum principle are independent of the objective functional $\int_0^{t_f} \alpha(t)\,dt$ in Problem~(OC) and therefore not sufficiently informative.  The problems that result in $\lambda_0 > 0$, and their solutions, are referred to as {\em normal}.  Lemma~\ref{normality} below states that we can rule out abnormality for Problem~(OC).

\begin{lemma}[Normality of Solutions] \label{normality}
Problem~{\em (OC)} is normal, i.e., one can set $\lambda_0 := 1$.
\end{lemma}
\begin{proof}
For contradiction purposes, suppose that $\lambda_0 = 0$ and recall that $a>0$.  Then, from~\eqref{adjoint4} and \eqref{u}, $\dot{\lambda}_4(t) = |\lambda_3(t)| \ge 0$, for every $t\in[0,t_f]$, which, along with the boundary conditions $\lambda_4(0) = \lambda_4(t_f) = 0$, implies that $\lambda_4(t) = \lambda_3(t) = 0$ for every $t\in[0,t_f]$, giving rise to a totally singular solution.  This in turn implies that $v(t) = 0$ and so $\dot\theta(t) = 0$, for every $t\in[0,t_f]$, and thus $a=0$, which is a contradiction.  Therefore $\lambda_0 > 0$.  Then, without loss of generality, one can set $\lambda_0 = 1$.
\end{proof}

\begin{remark}  \label{lambda0} \rm
We recall that as proved in~\cite{Kaya2017} a solution to a similar optimal control reformulation of the Markov--Dubins problem~(MD) can be abnormal, while we prove in Lemma~\ref{normality} that there are no abnormal solutions to Problem~(OC).
\endproof
\end{remark}

\begin{lemma}[Singularity and \boldmath{$\rho$}] \label{lem:rho} 
Suppose that the optimal control $v(t)$ for Problem~{\em (OC)} is singular over an interval $[\zeta_1,\zeta_2]\subset[0,t_f]$.  Then $\rho = |a - h| > 0$, with $h$ as defined in~\eqref{H_const2}.
\end{lemma}
\begin{proof}
Suppose the optimal control $u(t)$ is singular over $[\zeta_1,\zeta_2]\subset[0,t_f]$.  Then $\lambda_3(t) = \dot{\lambda}_3(t) = 0$ for a.e.\ $t\in[\zeta_1,\zeta_2]$. From \eqref{adjoint3a}, $\sin(\theta(t) - \phi) = 0$, which implies that $\cos(\theta(t) - \phi) = 1$ or $-1$.  Then, substituting $\lambda_3(t)=0$, $\cos(\theta(t) - \phi) = \pm 1$, and (from Lemma~\ref{normality}) $\lambda_0 = 1$, into \eqref{H_const2}, one gets $\rho = \mp(a - h)$, which is positive by the definition of $\rho$ given just before~\eqref{Hamiltonian2}.
\end{proof}

In the rest of the paper, we will not show at times the dependence of variables on $t$ for clarity of presentation.

\begin{proposition}[Phase Portrait for Switchings]  \label{prop:portrait}
The adjoint variable $\lambda_3$ for Problem~{\em (OC)} solves the differential equation
\begin{equation} \label{lambda3_DE} 
\dot{\lambda}_3^2(t) + \left(a\,|\lambda_3(t)| - a + h\right)^2 = \rho^2\,,
\end{equation}
the {\em phase portrait} of which, in the {\em phase variables} $\lambda_3$ and $\dot{\lambda}_3$, can be given as pieces, or concatenations of pieces, of concentric ellipses centred at $(-|a - h|/a, 0)$ for $v(t) = 1$ and $(|a - h|/a, 0)$ for $v(t) = -1$, as shown in Figure~\ref{phase}.  One has the following case-by-case conclusions.
\begin{itemize}
\item[(a)] $\rho > |a - h|$:  In this case, the trajectories are concatenations of (pieces of) ellipses, examples of which are shown by (dark blue) solid curves in Figure~\ref{phase}.  
\begin{itemize}
\item[(i)] The (solid) ellipses in the phase portrait are concatenated at the switching points given by $(0,\sqrt{\rho^2 - (a-h)^2})$ and $(0,-\sqrt{\rho^2 - (a-h)^2})$, where the value of the bang--bang control $v(t)$ switches from $1$ to $-1$ or from $-1$ to $1$, respectively.
\item[(ii)] The concatenated ellipses of the phase portrait cross the $\lambda_3$-axis at $(-(\rho + |a-h|)/a, 0)$ and $((\rho + |a-h|)/a, 0)$.
\item[(iii)] If the bang--bang control has two switchings, the second arc must have a length $\xi$ strictly greater than $\pi/a$, in other words, the smallest possible curvature $a$ is $\pi/\xi$.
\end{itemize}
\item[(b)] $\rho = |a - h|$:  The trajectory contains a singular arc, represented by the origin $(0,0)$ of the phase plane with $v(t) = 0$. In this case, any {\em partially singular}, i.e., a {\em bang--singular}, trajectory is contained by the two unique (red) dashed elliptic curves in Figure~\ref{phase}.  The expressions for the $\lambda_3$ intercepts (when $\lambda_3 \neq 0$) are the same as in part~(a)(ii) above.
\item[(c)] $0 < \rho < |a - h|$: This case is represented by (isolated) elliptic trajectories, examples of which are shown with (black) dotted curves in Figure~\ref{phase}.  Trajectories cross the $\lambda_3$-axis at $(-(|a-h| + \rho)/a, 0)$ and $(-(|a-h| - \rho)/a, 0)$ when $\lambda_3 < 0$, and $((|a-h| + \rho)/a, 0)$ and $((|a-h| - \rho)/a, 0)$ when $\lambda_3 >0$.  However, the trajectories no longer intercept the $\dot{\lambda}_3$-axis; therefore, they represent bang--bang control with no switchings, i.e., either $v(t) = 1$ for all $t\in[0,t_f]$ or $v(t) = -1$ for all $t\in[0,t_f]$.
\end{itemize}
\end{proposition}
\begin{figure}
\vspace*{-10mm}
\begin{center}
\psfrag{L}{$\lambda_3$}
\psfrag{Ld}{$\dot{\lambda}_3$}
\psfrag{b}{\footnotesize$-\ds\frac{|a - h|}{a}$}
\psfrag{c}{\footnotesize\hspace*{-0mm} $\ds\frac{|a - h|}{a}$}
\psfrag{u1}{\footnotesize $v(t) = 1$}
\psfrag{u2}{\footnotesize $v(t) = -1$}
\psfrag{rho1}{\footnotesize $0\le\rho < |a-h|$}
\psfrag{rho2}{\footnotesize $\rho = |a-h| \neq 0$}
\psfrag{rho3}{\footnotesize $\rho > |a-h|$}
\[\includegraphics[width=140mm]{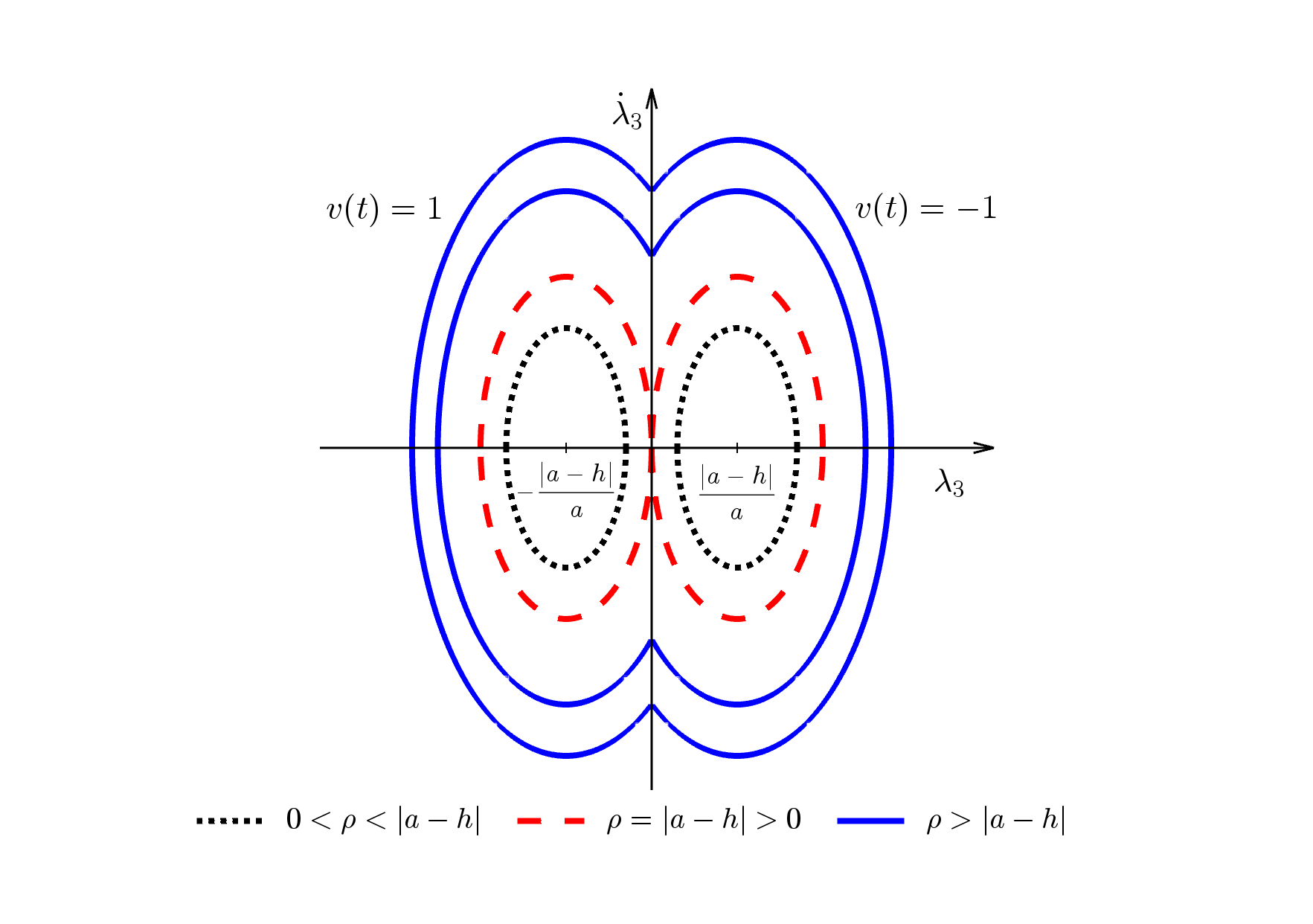}\]
\end{center}
\vspace*{-10mm}
\caption{\small\sf Phase portrait of \eqref{lambda3_DE} for the switching function $\lambda_3$.}
\label{phase}
\end{figure}
\begin{proof}
  From \eqref{adjoint3a},
\begin{equation} \label{eqnA} 
\dot{\lambda}_3^2(t) = \rho^2\,\sin^2(\theta(t) - \phi) = \rho^2 -
\rho^2\,\cos^2(\theta(t) - \phi)\,.
\end{equation}
Using $v(t) = -\sgn(\lambda_3(t))$ and $\lambda_0 = 1$ in \eqref{H_const2}, one gets
\[
\rho\,\cos(\theta(t) - \phi) = a\,|\lambda_3(t)| - a + h\,.
\]
Substituting this into the right-hand side of \eqref{eqnA} and rearranging give \eqref{lambda3_DE}.  Note that the differential equation \eqref{lambda3_DE} is given in terms of the {\em phase variables} $\lambda_3$ and $\dot{\lambda}_3$, and can be put into the form
\begin{equation} \label{lambda3_DE_ellipse} 
\left(\lambda_3 + \frac{|a - h|}{a}\,v\right)^2 + \frac{\dot\lambda_3^2}{a^2} = \frac{\rho^2}{a^2}\,,
\end{equation}
where we have used $v = -\sgn(\lambda_3)$ from~\eqref{u}.  Equation~\eqref{lambda3_DE_ellipse} clearly tells us that the trajectories in the {\em phase plane} for $\lambda_3$ (the $\lambda_3$$\dot{\lambda}_3$-plane) will be pieces or concatenations of pieces of concentric ellipses centred at $(-|a - h|/a, 0)$ for $v(t) = 1$ and $(|a - h|/a, 0)$ for $v(t) = -1$, as shown in Figure~\ref{phase}.  

In parts~(a)(i) and (a)(ii), the $\dot{\lambda}_3$ and $\lambda_3$ intercepts can be found after substituting $\lambda_3 = 0$ and $\dot{\lambda}_3 = 0$ into~\eqref{lambda3_DE_ellipse}, respectively, by also setting $v(t) = 1$ (when $\lambda_3(t) < 0$) and $v(t) = -1$ (when $\lambda_3(t) > 0$) in each case.  Part~(a)(iii) is straightforward to deduce after inspection of Figure~\ref{phase}, with the $\lambda_3$ intercepts found in the same way as in part (a)(ii).

Part~(b) is furnished by Lemma~\ref{lem:rho}.

In part~(c), substitutions similar to those made in parts~(a)(i) and (a)(ii) result in the expressions for the $\lambda_3$ intercepts, when $v(t) = 1$ and when $v(t) = -1$.
\end{proof}

\begin{remark}[Concatenations] \label{rem:concatenations} \rm 
The diagram in Figure~\ref{phase} tells us that an optimal path will generally be a concatenation of straight lines (i.e., singular arcs, where $v(t) = 0$) and circular arcs (i.e., nonsingular arcs, where $v(t) = 1$ or $-1$).  However, the diagram does not tell us how many switchings an optimal control must have.
\endproof
\end{remark}

\begin{lemma} \label{nonsingular}
Suppose that optimal control $v(t)$ for Problem~{\em (OC)} is nonsingular over an interval $[\zeta_3,\zeta_4]\subset[0,t_f]$.  Then 
\begin{equation} \label{lambda3} 
|\lambda_3(t)| = \frac{1}{a}\left[\rho\,\cos(\theta(t) - \phi) + a - h\right]\,,\ \ \mbox{ for a.e.\ } t\in[\zeta_3,\zeta_4]\subset[0,t_f]\,. 
\end{equation}
\end{lemma}
\begin{proof}
Substitution of $v(t) = -\sgn(\lambda_3(t))$ and $\lambda_0 = 1$ into \eqref{H_const2} and re-arranging yield the required expression. 
\end{proof}

Recall that when the switching function $\lambda_3(t)$ vanishes, or equivalently when $H_v[t] = 0$, for $t\in[\zeta_1,\zeta_2]\subset[0,t_f]$, we identified the singular control as $v(t) = 0$ over the same interval (see Lemma~\ref{singular}).  However, owing to $\lambda_3(t)$ vanishing, \eqref{control2} is uninformative as a necessary condition of optimality in this interval.  Instead, a higher-order necessary condition of optimality for the singular control candidate, called the {\em generalized Legendre--Clebsch condition} (GLC), can be checked (see, e.g. \cite{BelJac1975, SchLed2012}):
\begin{equation}  \label{GLC}
(-1)^k \frac{\partial}{\partial v}\, \frac{d^{2k}}{dt^{2k}} H_v[t] \ge 0\,,
\end{equation}
for all $t\in[\zeta_1,\zeta_2]$, where $2k$ is the number of times $H_v$ is differentiated until the control $v$ appears for the first time, $k$ being referred to as the {\em order of the singular control} or the {\em order of singularity}.  The inequality in~\eqref{GLC} is also known as the Kelley condition, dating back to the 1960s when it was first conceived.  The GLC can also be expressed neatly in terms of Lie brackets~\cite{SchLed2012} or Poisson (or Jacobi) brackets\cite{Melikyan2006}, which have advantages for calculations and interpretations in a differential geometric setting.

\begin{proposition}[Order of Singularity and GLC]  \label{prop:GLC}
The singular control $v(t) = 0$ is of order one (i.e., $k =1$) and satisfies the generalized Legendre--Clebsch condition.
\end{proposition}
\begin{proof}
Recall that singularity arises when $H_v[t] = a\,\lambda_3(t) = 0$, or $\lambda_3(t) = 0$, over $[\zeta_1,\zeta_2]\subset[0,t_f]$.  The successive derivatives of $H_v[t]$ along a singular arc give $dH_v[t]/dt = a\,\dot{\lambda}_3(t) = a\,\rho\,\sin(\theta(t)-\phi) = 0$ and so $\sin(\theta(t)-\phi) = 0$ as $a$ and $\rho$ are nonzero; $d^2H_v[t]/dt^2 = a\,\rho\,\dot{\theta}(t)\,\cos(\theta(t)-\phi) = a^2\,\rho\,v(t)\,\cos(\theta(t)-\phi) = 0$, which implies $k = 1$, i.e., the order of singularity is one, and that $v(t) = 0$, since $\cos(\theta(t)-\phi) = 1$ or $-1$ as $\sin(\theta(t)-\phi) = 0$.  Therefore, in view of~\eqref{GLC}, we need to show that
\[
\frac{\partial}{\partial v}\, \frac{d^{2}}{dt^{2}} H_v[t] = a^2\,\rho\,\cos(\theta(t)-\phi) \le 0\,,
\]
for all $t\in[\zeta_1,\zeta_2]$.  Since $a^2\,\rho > 0$ and $\cos(\theta(t)-\phi) = 1$ or $-1$, satisfying GLC boils down to showing that $\cos(\theta(t)-\phi) = -1$ along a singular arc.

By Proposition~\ref{prop:portrait}(a)(ii)--(b), on a bang-trajectory (i.e., when $v(t) = 1$ or $-1$) where $\dot{\lambda}_3(t) = 0$ for some $t\in[0,t_f]$, we have $|\lambda_3(t)| = (\rho + |a - h|) / a$.  Then, by Lemma~\ref{nonsingular},
\begin{equation}  \label{eq:lambda3_intercept}
|\lambda_3(t)| = \frac{1}{a}\left[\rho\,\cos(\theta(t) - \phi) + a - h\right] = \frac{1}{a}(\rho + |a - h|)\,.
\end{equation}
Consider a bang--singular--bang trajectory and the phase portrait in Figure~\ref{phase} (in particular, consider the (red) dashed elliptic curves there).  Then, by Lemma~\ref{lem:rho}, $\rho = |a - h|$.  Now, consider a (bang) part of this trajectory such that $\dot{\lambda}_3(t) = 0$ but $|\lambda_3(t)| \neq 0$.  
Then, \eqref{eq:lambda3_intercept} becomes
\[
2\,|a - h| = |a - h|\,\cos(\theta(t) - \phi) + a - h\,,
\]
which implies that $\cos(\theta(t) - \phi) = 1$ and $h < a$.  Next, suppose that $\dot{\lambda}_3(t) = 0$ and $|\lambda_3(t)| = 0$.  Then, from~\eqref{eq:lambda3_intercept},
\[
0 = |a - h|\,\cos(\theta(t) - \phi) + a - h\,,
\]
which in turn implies that $\cos(\theta(t) - \phi) = -1$ since $a - h > 0$, completing the proof.
\end{proof}

\begin{remark}[Singularity and Chattering Arcs]   \label{rem:chatter1}  \rm
By the proof of Proposition~\ref{prop:GLC} above, we note that not only the order of singularity is one, but GLC is also strictly satisfied, i.e., $k=1$ and in~\eqref{GLC} strict inequality holds.  Moreover, the singular control $v(t) = 0\in(-1,1)$.  Then, by~\cite[Prop.~2.8.4]{SchLed2012}, the bang--singular and singular--bang concatenations of the arcs are allowed; in other words, switching from a bang arc to a singular arc or from a singular arc to a bang arc is allowed to occur, in line with the maximum principle.  Consequently, chattering arcs, i.e., (optimal) arcs with infinitely many switchings, cannot occur.
\endproof
\end{remark}

\begin{remark}[Chattering Arcs in a Related Problem]   \label{rem:chatter2}  \rm
The {\em minimax spirality problem}, introduced and studied in~\cite{KayNoaSch2025} by the authors of the current paper, is related to Problem~(P) in that it is concerned with minimizing the $L^\infty$-norm of the arc-length derivative of the curvature instead of the $L^\infty$-norm of the curvature.  The minimax spirality problem can be cast as an optimal control problem (without the path constraints on the curvature that are imposed in~\cite{KayNoaSch2025}) as follows.
\[
\mbox{(MS)}\left\{\begin{array}{rll}
\ds\min_{v(\cdot)} &\ \ds \int_0^{t_f} \beta(t)\,dt  & \\[4mm]
\mbox{s.t.} &\ \dot{x}(t) = \cos\theta(t)\,, & x(0) = x_0\,,\ 
              x(t_f) = x_f\,, \\[2mm] 
  &\ \dot{y}(t) = \sin\theta(t)\,, & y(0) = y_0\,,\ 
              y(t_f) = y_f\,,\\[2mm] 
  &\ \dot{\theta}(t) = \kappa(t)\,, & \theta(0) = \theta_0\,,\ 
              \theta(t_f) = \theta_f\,,\\[2mm]
  &\ \dot{\kappa}(t) = \beta(t)\,v(t)\,, & \kappa(0) = \kappa_0\,,\ \kappa(t_f) = \kappa_f\,, \\[2mm]
  &\ \dot{\beta}(t) = 0\,, & |v(t)|\le 1\,, \ \mbox{for a.e. }  t\in[0,t_f]\,,
\end{array}\right.
\]
where $\kappa(t)$ is the curvature and $\beta(t)$, which stands for an unknown (constant) bound $b > 0$ on $|\dot{\kappa}(t)|$ for a.e. $t\in[0,t_f]$, is to be minimized.  Furthermore, $v(t)$ is the control variable.  In fact, it can be shown that the (candidate) singular control here is $v(t) = 0$~\cite[Lemma~3.1]{KayNoaSch2025}.  The result in~\cite[Lemma~3.5]{KayNoaSch2025} states that if there is a singular arc but no chattering in the solution curve of Problem~(MS), then the singularity is total, i.e., the singular arc is the whole curve.  Here, we provide an observation which was not explicitly made in~\cite{KayNoaSch2025}:  By the contrapositive of~\cite[Lemma~3.5]{KayNoaSch2025}, if the singularity is partial, i.e., the singular arc is to be concatenated with a bang-arc (along which $v(t) = 1$ or $-1$), then there is a chattering arc in the solution curve of Problem~(MS).  Subsequently, the numerical experiments in~\cite[Fig.~6]{KayNoaSch2025} also provide some evidence for the presence of chattering arcs.  Problem~(MS) furnishes an example that chattering arcs are not so uncommon with variational curves.

For an alternative approach, we note that the (candidate) singular control $v(t) = 0$ lies in the interior of the interval $[-1,1]$ and the order of the singularity for Problem~(MS) can easily be checked to be two ($k = 2$).  These are two of the three hypotheses listed in~\cite[Prop.~2.8.5]{SchLed2012} for chattering arcs to eventuate.  The third hypothesis in~\cite[Prop.~2.8.5]{SchLed2012} is that GLC, i.e., \eqref{GLC}, with $k=2$ and strict inequality, is satisfied; however, this assumption does not seem straightforward to check for Problem~(MS).
\endproof
\end{remark}

\section{Classification of Optimal Paths}
\label{sec:classification}

A trajectory satisfying the necessary conditions can be described as being of type $C, S, CSC,\ldots$ etc. In \cite[Lemma~6]{Kaya2017} the analog of Figure \ref{phase} is used to conclude that an optimal path for (MD) contains a straight line segment $S$ then it must be of type $CSC, CS, SC,$ or $S$. Part of the argument is that a trajectory containing e.g.\ 
$SCS$ must traverse a full circle and is therefore not optimal because the objective in the Markov--Dubins problem is to minimize length. 

For Problem~(OC) we can also conclude from Figure \ref{phase} and \eqref{adjoint3a} that $SCS$ must traverse a full circle, but it is no longer obvious that such a path cannot be optimal. In fact we will show that in some cases it \textit{is} optimal ( Proposition \ref{SCS}, Proposition \ref{COC}) while in others it is not (Corollary \ref{OSC}).  It will be helpful to introduce the notation $O$ to represent a full circle once covered, and reserve $C$ for circular arcs which are not closed.

\begin{lemma} \label{subtrajectory}
    If $\mathbf{x}^*(t):=(x^*(t),y^*(t),\theta^*(t))$ with $\alpha^*(t)=a^*$ is an optimal path for Problem (OC) then any sub-path of $\mathbf{x}^*$ is also optimal for the inherited constraints. That is, for any  $[t_1,t_2]\subset [0,t_f]$  the restriction $\mathbf{x}^*| [t_1,t_2]$ is optimal for problem (OC) with the constraints replaced by $x(t_i)=x^*(t_i)$, $y(t_i)=y^*(t_i)$, $\theta(t_i)=\theta^*(t_i)$ for $ i=1,2$, with  $ t\in [t_1,t_2]$.
\end{lemma}
\begin{proof}
    If the sub-path is a line then it is optimal. Assume it is not a line, and suppose it is not optimal. Then there exists a path satisfying the inherited constraints and having $0<a<a^*$, meaning it contains circular arcs with lower curvature than those in $\mathbf{x}^*$. We can insert it between $\mathbf{x}^*([0,t_1])$ and $\mathbf{x}^*([t_2,t_f])$ to obtain a new path $\tilde{\mathbf{x}}(t)$. Now either $\mathbf{x}^*$ and $\tilde{\mathbf{x}}$ have the same maximum curvature, or the maximum curvature of $\tilde{\mathbf{x}}$ is smaller. The latter contradicts our assumption that $\mathbf{x}^*$ is optimal. If they have the same maximum curvature, then $\tilde{\mathbf{x}}$ contains at least one circular arc with curvature $a^*$ and at least one circular arc with curvature $a$, and so it does not satisfy the necessary conditions for optimality (specifically \eqref{control3}). It then follows that $a^*$ cannot be a minimum, again contradicting our assumption on $\mathbf x$. 
\end{proof}

\begin{lemma}\label{fullcircle}
    Any path of type $XCY$, where $X,Y$ are $C,O$ or $S$ but not both $C$, is not optimal.
\end{lemma}
\begin{proof}
    Since at least one of $X$ and $Y$ is $S$ or $O$, in order for the path to be optimal it must have switching function on the (red) dashed ellipses in Figure~\ref{phase}. But since the middle $C$ in $XCY$ does not traverse a full loop it is impossible for the switching function to traverse a full ellipse and switch to $S$ or to  $C$ or $O$ with the opposite orientation. (Note that we are assuming that, for example, in $CCS$ the two arcs have opposite orientation, otherwise this path would be written as $CS$ or $OCS$). In the case of $CCO$ (similarly $OCC$) this is perhaps not immediately obvious because the middle $C$ and the $O$ may have the same orientation, and then the switching function would not need to return to the origin of the phase portrait in Figure~\ref{phase}.  However, we can reflect the loop about the endpoint with no change in maximum curvature, and the resulting path again is not compatible with the necessary conditions on the switching function. 
\end{proof}

\begin{corollary} \label{OSC}
    Any path of type $SOSC$ or any permutation thereof, or type $COCC$ or any permutation thereof,  cannot be optimal.
\end{corollary}

\begin{proof}
    Note that the loop can be shifted anywhere along the path without changing the maximum curvature. For example a path of type $SOSC$ has the same maximum curvature as the path of type $SCO$ obtained by shifting the loop to the end. The latter is not optimal by Lemma \ref{fullcircle}, and therefore neither is the former. Similarly, the other permutations of $SOSC$ and permutations of $COCC$ are not optimal. 
\end{proof}

\begin{lemma}\label{CCCC}
	A $CCCC$ trajectory is not a solution to (Pc).
\end{lemma}

\emph{Sketch of the proof}. Following the observations in Proposition \ref{prop:portrait}, a $CCCC$ trajectory must be of the type pictured in Figure \ref{fig:CCSC}, with the second and third circular arcs having the same length, which is greater than $\pi/a$. The idea behind this proof is very simple: beginning with the $CCCC$ trajectory, we roll the circle $c_1$ around the circle $c_0$ while allowing the third arc to be `pulled-tight' and allowing the fourth arc to relax to the circle centred at $\tilde c_2$, to obtain the dashed trajectory in Figure~\ref{fig:CCSC}. This new trajectory does not satisfy the necessary conditions for optimality because the second circular arc does not traverse a full circle (see Proposition \ref{prop:portrait}). However, it has the same maximum curvature as the $CCCC$ trajectory we started with, and therefore this maximum curvature cannot be minimal.

The reason the proof becomes somewhat technical below is that it does not seem to be easy to solve explicitly for the rolling angle $\beta_0$ which ensures that the dashed curve satisfies the length constraint. Instead we will prove (taking care with a couple of slightly different cases) that the length depends continuously on $\beta_0$, and that there exist values of $\beta_0$ such that the length is greater than and less than the length of the original $CCCC$ trajectory. 

\begin{proof}
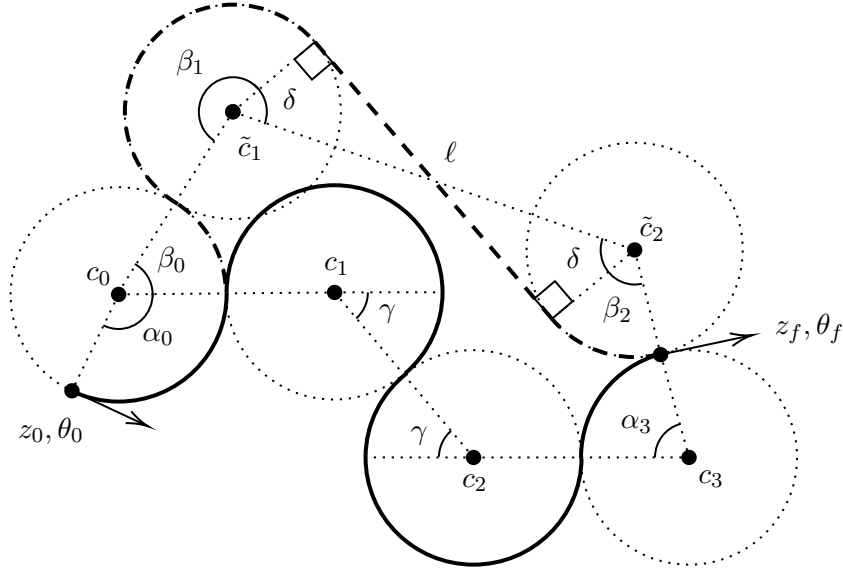
\begin{figure}[ht]
\centering


\tikzset{every picture/.style={line width=.75pt}} 

\begin{tikzpicture}[x=0.75pt,y=0.75pt,yscale=-1,xscale=1]

\draw  [dash pattern={on 0.84pt off 2.51pt}] (183.35,155.75) .. controls (183.35,125.96) and (207.5,101.81) .. (237.29,101.81) .. controls (267.08,101.81) and (291.23,125.96) .. (291.23,155.75) .. controls (291.23,185.54) and (267.08,209.69) .. (237.29,209.69) .. controls (207.5,209.69) and (183.35,185.54) .. (183.35,155.75) -- cycle ;
\draw  [draw opacity=0][line width=1.5]  (291.18,154.75) .. controls (291.19,155.08) and (291.19,155.42) .. (291.19,155.75) .. controls (291.19,185.52) and (267.06,209.65) .. (237.29,209.65) .. controls (228.76,209.65) and (220.69,207.66) .. (213.51,204.13) -- (237.29,155.75) -- cycle ; \draw  [line width=1.5]  (291.18,154.75) .. controls (291.19,155.08) and (291.19,155.42) .. (291.19,155.75) .. controls (291.19,185.52) and (267.06,209.65) .. (237.29,209.65) .. controls (228.76,209.65) and (220.69,207.66) .. (213.51,204.13) ;  
\draw  [dash pattern={on 0.84pt off 2.51pt}] (290.81,154.68) .. controls (290.81,124.9) and (314.96,100.75) .. (344.75,100.75) .. controls (374.54,100.75) and (398.69,124.9) .. (398.69,154.68) .. controls (398.69,184.47) and (374.54,208.62) .. (344.75,208.62) .. controls (314.96,208.62) and (290.81,184.47) .. (290.81,154.68) -- cycle ;
\draw  [dash pattern={on 0.84pt off 2.51pt}] (360.19,237.73) .. controls (360.19,207.94) and (384.34,183.79) .. (414.13,183.79) .. controls (443.92,183.79) and (468.07,207.94) .. (468.07,237.73) .. controls (468.07,267.52) and (443.92,291.67) .. (414.13,291.67) .. controls (384.34,291.67) and (360.19,267.52) .. (360.19,237.73) -- cycle ;
\draw  [dash pattern={on 0.84pt off 2.51pt}] (468.07,237.73) .. controls (468.07,207.94) and (492.22,183.79) .. (522.01,183.79) .. controls (551.8,183.79) and (575.94,207.94) .. (575.94,237.73) .. controls (575.94,267.52) and (551.8,291.67) .. (522.01,291.67) .. controls (492.22,291.67) and (468.07,267.52) .. (468.07,237.73) -- cycle ;
\draw    (214,204.23) -- (250.3,221.31) ;
\draw [shift={(252.11,222.16)}, rotate = 205.2] [color={rgb, 255:red, 0; green, 0; blue, 0 }  ][line width=0.75]    (10.93,-3.29) .. controls (6.95,-1.4) and (3.31,-0.3) .. (0,0) .. controls (3.31,0.3) and (6.95,1.4) .. (10.93,3.29)   ;
\draw [shift={(214,204.23)}, rotate = 25.2] [color={rgb, 255:red, 0; green, 0; blue, 0 }  ][fill={rgb, 255:red, 0; green, 0; blue, 0 }  ][line width=0.75]      (0, 0) circle [x radius= 3.35, y radius= 3.35]   ;
\draw    (507.83,186.03) -- (551.74,176.63) ;
\draw [shift={(553.7,176.21)}, rotate = 167.92] [color={rgb, 255:red, 0; green, 0; blue, 0 }  ][line width=0.75]    (10.93,-3.29) .. controls (6.95,-1.4) and (3.31,-0.3) .. (0,0) .. controls (3.31,0.3) and (6.95,1.4) .. (10.93,3.29)   ;
\draw [shift={(507.83,186.03)}, rotate = 347.92] [color={rgb, 255:red, 0; green, 0; blue, 0 }  ][fill={rgb, 255:red, 0; green, 0; blue, 0 }  ][line width=0.75]      (0, 0) circle [x radius= 3.35, y radius= 3.35]   ;
\draw  [draw opacity=0][line width=1.5]  (291.18,154.75) .. controls (291.18,154.74) and (291.18,154.73) .. (291.18,154.72) .. controls (291.18,124.96) and (315.31,100.83) .. (345.08,100.83) .. controls (374.84,100.83) and (398.98,124.96) .. (398.98,154.72) .. controls (398.98,170.99) and (391.77,185.58) .. (380.37,195.46) -- (345.08,154.72) -- cycle ; \draw  [line width=1.5]  (291.18,154.75) .. controls (291.18,154.74) and (291.18,154.73) .. (291.18,154.72) .. controls (291.18,124.96) and (315.31,100.83) .. (345.08,100.83) .. controls (374.84,100.83) and (398.98,124.96) .. (398.98,154.72) .. controls (398.98,170.99) and (391.77,185.58) .. (380.37,195.46) ;  
\draw  [draw opacity=0][line width=1.5]  (468.36,238.3) .. controls (468.07,267.82) and (444.05,291.67) .. (414.46,291.67) .. controls (384.69,291.67) and (360.56,267.54) .. (360.56,237.77) .. controls (360.56,220.54) and (368.65,205.2) .. (381.23,195.33) -- (414.46,237.77) -- cycle ; \draw  [line width=1.5]  (468.36,238.3) .. controls (468.07,267.82) and (444.05,291.67) .. (414.46,291.67) .. controls (384.69,291.67) and (360.56,267.54) .. (360.56,237.77) .. controls (360.56,220.54) and (368.65,205.2) .. (381.23,195.33) ;  
\draw  [draw opacity=0][line width=1.5]  (468.11,238.25) .. controls (468.11,238.07) and (468.11,237.9) .. (468.11,237.73) .. controls (468.11,212.87) and (484.94,191.93) .. (507.84,185.71) -- (522.01,237.73) -- cycle ; \draw  [line width=1.5]  (468.11,238.25) .. controls (468.11,238.07) and (468.11,237.9) .. (468.11,237.73) .. controls (468.11,212.87) and (484.94,191.93) .. (507.84,185.71) ;  
\draw  [dash pattern={on 0.84pt off 2.51pt}]  (237.29,155.75) -- (214,204.23) ;
\draw  [dash pattern={on 0.84pt off 2.51pt}]  (237.29,155.75) -- (345.12,154.75) ;
\draw [shift={(345.12,154.75)}, rotate = 359.47] [color={rgb, 255:red, 0; green, 0; blue, 0 }  ][fill={rgb, 255:red, 0; green, 0; blue, 0 }  ][line width=0.75]      (0, 0) circle [x radius= 3.35, y radius= 3.35]   ;
\draw [shift={(237.29,155.75)}, rotate = 359.47] [color={rgb, 255:red, 0; green, 0; blue, 0 }  ][fill={rgb, 255:red, 0; green, 0; blue, 0 }  ][line width=0.75]      (0, 0) circle [x radius= 3.35, y radius= 3.35]   ;
\draw  [dash pattern={on 0.84pt off 2.51pt}]  (345.08,154.72) -- (398.69,154.68) ;
\draw  [dash pattern={on 0.84pt off 2.51pt}]  (345.12,154.75) -- (414.46,237.77) ;
\draw  [dash pattern={on 0.84pt off 2.51pt}]  (414.46,237.77) -- (522.01,237.73) ;
\draw [shift={(522.01,237.73)}, rotate = 359.98] [color={rgb, 255:red, 0; green, 0; blue, 0 }  ][fill={rgb, 255:red, 0; green, 0; blue, 0 }  ][line width=0.75]      (0, 0) circle [x radius= 3.35, y radius= 3.35]   ;
\draw [shift={(414.46,237.77)}, rotate = 359.98] [color={rgb, 255:red, 0; green, 0; blue, 0 }  ][fill={rgb, 255:red, 0; green, 0; blue, 0 }  ][line width=0.75]      (0, 0) circle [x radius= 3.35, y radius= 3.35]   ;
\draw  [dash pattern={on 0.84pt off 2.51pt}]  (507.84,185.71) -- (522.01,237.73) ;
\draw  [draw opacity=0] (254.24,154.94) .. controls (254.25,155.21) and (254.26,155.48) .. (254.26,155.75) .. controls (254.26,165.35) and (246.66,173.14) .. (237.29,173.14) .. controls (234.54,173.14) and (231.95,172.47) .. (229.66,171.28) -- (237.29,155.75) -- cycle ; \draw   (254.24,154.94) .. controls (254.25,155.21) and (254.26,155.48) .. (254.26,155.75) .. controls (254.26,165.35) and (246.66,173.14) .. (237.29,173.14) .. controls (234.54,173.14) and (231.95,172.47) .. (229.66,171.28) ;  
\draw  [draw opacity=0] (361.71,154.74) .. controls (361.7,159.95) and (359.45,164.62) .. (355.89,167.8) -- (344.75,154.68) -- cycle ; \draw   (361.71,154.74) .. controls (361.7,159.95) and (359.45,164.62) .. (355.89,167.8) ;  
\draw  [draw opacity=0] (397.17,237.38) .. controls (397.28,232.17) and (399.61,227.54) .. (403.22,224.42) -- (414.13,237.73) -- cycle ; \draw   (397.17,237.38) .. controls (397.28,232.17) and (399.61,227.54) .. (403.22,224.42) ;  
\draw  [draw opacity=0] (505.04,237.7) .. controls (505.05,229.65) and (510.41,222.88) .. (517.66,220.91) -- (522.01,237.73) -- cycle ; \draw   (505.04,237.7) .. controls (505.05,229.65) and (510.41,222.88) .. (517.66,220.91) ;  
\draw    (294.76,258.41) ;
\draw  [dash pattern={on 0.84pt off 2.51pt}]  (360.19,237.73) -- (414.46,237.77) ;
\draw  [draw opacity=0][dash pattern={on 5.63pt off 4.5pt}][line width=1.5]  (267.15,110.5) .. controls (281.32,120.12) and (290.66,136.31) .. (290.81,154.68) -- (236.91,155.13) -- cycle ; \draw  [color={rgb, 255:red, 0; green, 0; blue, 0 }  ,draw opacity=1 ][dash pattern={on 5.63pt off 4.5pt}][line width=1.5]  (267.15,110.5) .. controls (281.32,120.12) and (290.66,136.31) .. (290.81,154.68) ;  
\draw  [dash pattern={on 0.84pt off 2.51pt}] (240.37,63.89) .. controls (240.37,34.1) and (264.52,9.95) .. (294.31,9.95) .. controls (324.1,9.95) and (348.24,34.1) .. (348.24,63.89) .. controls (348.24,93.68) and (324.1,117.83) .. (294.31,117.83) .. controls (264.52,117.83) and (240.37,93.68) .. (240.37,63.89) -- cycle ;
\draw  [dash pattern={on 0.84pt off 2.51pt}] (440.94,133.38) .. controls (440.94,103.59) and (465.09,79.44) .. (494.88,79.44) .. controls (524.66,79.44) and (548.81,103.59) .. (548.81,133.38) .. controls (548.81,163.17) and (524.66,187.32) .. (494.88,187.32) .. controls (465.09,187.32) and (440.94,163.17) .. (440.94,133.38) -- cycle ;
\draw  [draw opacity=0][dash pattern={on 5.63pt off 4.5pt}][line width=1.5]  (266.84,110.27) .. controls (251.02,100.88) and (240.41,83.63) .. (240.41,63.89) .. controls (240.41,34.12) and (264.54,9.99) .. (294.31,9.99) .. controls (310.54,9.99) and (325.09,17.16) .. (334.97,28.51) -- (294.31,63.89) -- cycle ; \draw  [dash pattern={on 5.63pt off 4.5pt}][line width=1.5]  (266.84,110.27) .. controls (251.02,100.88) and (240.41,83.63) .. (240.41,63.89) .. controls (240.41,34.12) and (264.54,9.99) .. (294.31,9.99) .. controls (310.54,9.99) and (325.09,17.16) .. (334.97,28.51) ;  
\draw  [draw opacity=0][dash pattern={on 5.63pt off 4.5pt}][line width=1.5]  (507.84,185.71) .. controls (503.68,186.74) and (499.34,187.28) .. (494.88,187.28) .. controls (478.31,187.28) and (463.49,179.8) .. (453.6,168.04) -- (494.88,133.38) -- cycle ; \draw  [dash pattern={on 5.63pt off 4.5pt}][line width=1.5]  (507.84,185.71) .. controls (503.68,186.74) and (499.34,187.28) .. (494.88,187.28) .. controls (478.31,187.28) and (463.49,179.8) .. (453.6,168.04) ;  
\draw  [dash pattern={on 0.84pt off 2.51pt}]  (236.91,155.13) -- (294.31,63.89) ;
\draw [shift={(294.31,63.89)}, rotate = 302.17] [color={rgb, 255:red, 0; green, 0; blue, 0 }  ][fill={rgb, 255:red, 0; green, 0; blue, 0 }  ][line width=0.75]      (0, 0) circle [x radius= 3.35, y radius= 3.35]   ;
\draw [line width=1.5]  [dash pattern={on 5.63pt off 4.5pt}]  (335.29,28.74) -- (453.6,168.04) ;
\draw  [dash pattern={on 0.84pt off 2.51pt}]  (453.6,168.04) -- (494.88,133.38) ;
\draw  [dash pattern={on 0.84pt off 2.51pt}]  (494.88,133.38) -- (507.84,185.71) ;
\draw  [draw opacity=0] (245.37,140.46) .. controls (250.66,143.4) and (254.26,149.15) .. (254.26,155.75) .. controls (254.26,156.1) and (254.25,156.45) .. (254.23,156.8) -- (237.29,155.75) -- cycle ; \draw   (245.37,140.46) .. controls (250.66,143.4) and (254.26,149.15) .. (254.26,155.75) .. controls (254.26,156.1) and (254.25,156.45) .. (254.23,156.8) ;  
\draw  [draw opacity=0] (285.08,78.48) .. controls (280.42,75.38) and (277.34,70.01) .. (277.34,63.89) .. controls (277.34,54.28) and (284.94,46.5) .. (294.31,46.5) .. controls (299.42,46.5) and (304.01,48.82) .. (307.12,52.49) -- (294.31,63.89) -- cycle ; \draw   (285.08,78.48) .. controls (280.42,75.38) and (277.34,70.01) .. (277.34,63.89) .. controls (277.34,54.28) and (284.94,46.5) .. (294.31,46.5) .. controls (299.42,46.5) and (304.01,48.82) .. (307.12,52.49) ;  
\draw  [draw opacity=0] (498.46,150.38) .. controls (497.21,150.66) and (495.91,150.79) .. (494.57,150.77) .. controls (485.21,150.61) and (477.75,142.69) .. (477.91,133.09) .. controls (477.95,131.14) and (478.29,129.27) .. (478.9,127.52) -- (494.88,133.38) -- cycle ; \draw   (498.46,150.38) .. controls (497.21,150.66) and (495.91,150.79) .. (494.57,150.77) .. controls (485.21,150.61) and (477.75,142.69) .. (477.91,133.09) .. controls (477.95,131.14) and (478.29,129.27) .. (478.9,127.52) ;  
\draw    (534.64,177.71) ;
\draw  [dash pattern={on 0.84pt off 2.51pt}]  (294.31,63.89) -- (335.29,28.74) ;
\draw  [dash pattern={on 0.84pt off 2.51pt}]  (294.31,63.89) -- (494.88,133.38) ;
\draw [shift={(494.88,133.38)}, rotate = 19.11] [color={rgb, 255:red, 0; green, 0; blue, 0 }  ][fill={rgb, 255:red, 0; green, 0; blue, 0 }  ][line width=0.75]      (0, 0) circle [x radius= 3.35, y radius= 3.35]   ;
\draw [shift={(294.31,63.89)}, rotate = 19.11] [color={rgb, 255:red, 0; green, 0; blue, 0 }  ][fill={rgb, 255:red, 0; green, 0; blue, 0 }  ][line width=0.75]      (0, 0) circle [x radius= 3.35, y radius= 3.35]   ;
\draw   (325.06,37.47) -- (335.29,28.74) -- (344.02,38.97) -- (333.79,47.7) -- cycle ;
\draw   (444.87,157.81) -- (455.1,149.08) -- (463.83,159.31) -- (453.6,168.04) -- cycle ;
\draw  [draw opacity=0] (304.69,50.13) .. controls (308.69,53.31) and (311.27,58.29) .. (311.27,63.89) .. controls (311.27,66.1) and (310.87,68.21) .. (310.14,70.16) -- (294.31,63.89) -- cycle ; \draw   (304.69,50.13) .. controls (308.69,53.31) and (311.27,58.29) .. (311.27,63.89) .. controls (311.27,66.1) and (310.87,68.21) .. (310.14,70.16) ;  

\draw (185.72,216.36) node [anchor=north west][inner sep=0.75pt]    {$z_{0} ,\theta _{0}$};
\draw (563.43,164.18) node [anchor=north west][inner sep=0.75pt]    {$z_{f} ,\theta _{f}$};
\draw (247.71,171.35) node [anchor=north west][inner sep=0.75pt]    {$\alpha _{0}$};
\draw (365.91,159.43) node [anchor=north west][inner sep=0.75pt]    {$\gamma $};
\draw (381.64,220.48) node [anchor=north west][inner sep=0.75pt]    {$\gamma $};
\draw (486.47,215.34) node [anchor=north west][inner sep=0.75pt]    {$\alpha _{3}$};
\draw (338.39,135.87) node [anchor=north west][inner sep=0.75pt]    {$c_{1}$};
\draw (407.31,244.82) node [anchor=north west][inner sep=0.75pt]    {$c_{2}$};
\draw (525.28,242.58) node [anchor=north west][inner sep=0.75pt]    {$c_{3}$};
\draw (219.57,140.87) node [anchor=north west][inner sep=0.75pt]    {$c_{0}$};
\draw (295.33,74.77) node [anchor=north west][inner sep=0.75pt]    {$\tilde{c}_{1}$};
\draw (495.83,115.25) node [anchor=north west][inner sep=0.75pt]    {$\tilde{c}_{2}$};
\draw (264.47,31.23) node [anchor=north west][inner sep=0.75pt]    {$\beta _{1}$};
\draw (255.1,130.51) node [anchor=north west][inner sep=0.75pt]    {$\beta _{0}$};
\draw (477.57,155.56) node [anchor=north west][inner sep=0.75pt]    {$\beta _{2}$};
\draw (398.45,77.61) node [anchor=north west][inner sep=0.75pt]    {$\ell $};
\draw (459.33,130.78) node [anchor=north west][inner sep=0.75pt]    {$\delta $};
\draw (317.83,51.04) node [anchor=north west][inner sep=0.75pt]    {$\delta $};

\end{tikzpicture}


\caption{Diagram for the proof of Lemma \ref{CCCC}, case 1: trajectory of type $CCCC$ (solid) which satisfies the necessary conditions from the maximum principle, and another of type $CCSC$ (dashed) with the same maximum curvature which does not. Notice that in order to satisfy the necessary conditions the second and third arcs in the $CCCC$ trajectory must subtend the same angle: this follows from \eqref{adjoint3a} and the symmetry about the vertical axis in the phase portrait for $\lambda_3$ (Figure \ref{phase}).}
\label{fig:CCSC}
\end{figure}

Referring to Figure \ref{fig:CCSC}, we consider the points $c_i$ to be in the complex plane with $c_0=0$, and assume that the radius of each of the circles is 1.
 
 {\bf Case 1a}: $\alpha_3<\pi/2$ and $\gamma<\pi/2$. We reflect $c_3$ across $\theta_f$ and rotate $c_1$ about $c_0$ until we obtain the $CCSC$ trajectory pictured in Figure $\ref{fig:CCSC}$. Supposing there exists $\beta_0$ such that the pictured trajectory satisfies our length constraint, then it satisfies all the constraints of Problem~(Pc). As explained in the sketch above, this second trajectory has the same maximum curvature as the original $CCCC$ but cannot be optimal, and therefore the $CCCC$ trajectory is not optimal.
 
It remains to prove that there exists $\beta_0$ such that the $CCSC$ trajectory satisfies the length constraint. Referring to Figure \ref{fig:CCSC} we have
\begin{align*}
    \tilde c_1&= 2e^{i\beta_0} \\
    \tilde c_2&=z_f+e^{i(\pi/2 +\theta_f)}\\
    L&=\alpha_0+\beta_0+\beta_1+\ell +\beta_2 \\
    \beta_0-\beta_1+\beta_2&=\theta_f-\pi/2
\end{align*}
where $L$ is the total length. If we let 
\[
w=\tilde c_2- \tilde c_1 =z_f+e^{i(\pi/2 +\theta_f)}-2e^{i\beta_0}
\]
then 
\begin{align} \nonumber
    \ell&=\sqrt{|w|^2-2}\\ \nonumber
    \delta&=\cos^{-1}(2/|w|)\\ \nonumber
    \beta_1&=\delta-\arg(w)-\pi-\beta_0=\cos^{-1}(2/|w|)-\arg(w)-\pi-\beta_0\\ \nonumber
    \beta_2&= \theta_f-\pi/2+\beta_1-\beta_0=\theta_f-3\pi/2+\cos^{-1}(2/|w|)-\arg(w)-2\beta_0\\  
    L&=\alpha_0+2\beta_1+\ell+\theta_f-\pi/2= \alpha_0+2\cos^{-1}(2/|w|)-2\arg(w)-2\pi-2\beta_0+\ell+\theta_f-\pi/2 \label{newlength}
\end{align}
Since $w$ is determined by the boundary conditions and $\beta_0$, it follows that each of $\ell,\delta, \beta_1, \beta_2$ are determined by the the boundary conditions and $\beta_0$ and in fact $L$ depends continuously on $\beta_0$. If we can prove that there are values of $\beta_0$ for which $L$ is above and below the length constraint, then by the intermediate value theorem a trajectory satisfying the length constraint exists. 

Note that we can make $L$ as large as we like by increasing $\beta_0$, and wrapping the trajectory around the circle centred at $c_0$ if necessary. To prove that $L<t_f$ can be achieved we consider two cases. First we assume that as in Figure \ref{fig:CCSC} the reflected circle centred at $\tilde c_2$ does not intersect the circle centred at $c_1$, and so we can take $\beta_0=0$ and observe that the new $CCSC$ trajectory, where the $S$ is a common tangent to $c_1$ and $\tilde c_2$, is shorter. In the case where the reflected circle does intersect the circle centred at $c_1$, we consider the $CCC$ trajectory pictured in Figure \ref{fig:CCC}. Here $\ell=0$ and $\beta_0$ is at its minimum.

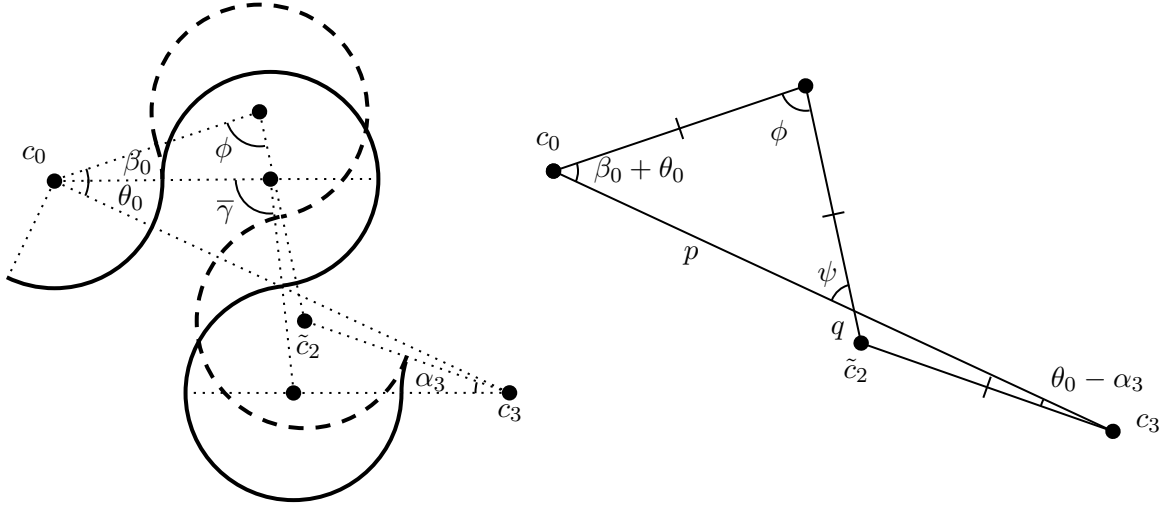
\begin{figure}
        \centering
        

\tikzset{every picture/.style={line width=0.75pt}} 

\begin{tikzpicture}[x=0.75pt,y=0.75pt,yscale=-1,xscale=1]

\draw    (670.66,245.41) -- (544.96,201.01) ;
\draw [shift={(607.81,223.21)}, rotate = 19.45] [color={rgb, 255:red, 0; green, 0; blue, 0 }  ][line width=0.75]    (0,5.59) -- (0,-5.59)   ;
\draw [shift={(670.66,245.41)}, rotate = 199.45] [color={rgb, 255:red, 0; green, 0; blue, 0 }  ][fill={rgb, 255:red, 0; green, 0; blue, 0 }  ][line width=0.75]      (0, 0) circle [x radius= 3.35, y radius= 3.35]   ;
\draw  [draw opacity=0] (634.65,232.64) .. controls (635.02,231.54) and (635.43,230.46) .. (635.89,229.41) -- (670.66,245.41) -- cycle ; \draw   (634.65,232.64) .. controls (635.02,231.54) and (635.43,230.46) .. (635.89,229.41) ;  
\draw    (391.41,114.49) -- (517.25,71.64) ;
\draw [shift={(517.25,71.64)}, rotate = 341.19] [color={rgb, 255:red, 0; green, 0; blue, 0 }  ][fill={rgb, 255:red, 0; green, 0; blue, 0 }  ][line width=0.75]      (0, 0) circle [x radius= 3.35, y radius= 3.35]   ;
\draw [shift={(454.33,93.06)}, rotate = 161.19] [color={rgb, 255:red, 0; green, 0; blue, 0 }  ][line width=0.75]    (0,5.59) -- (0,-5.59)   ;
\draw [shift={(391.41,114.49)}, rotate = 341.19] [color={rgb, 255:red, 0; green, 0; blue, 0 }  ][fill={rgb, 255:red, 0; green, 0; blue, 0 }  ][line width=0.75]      (0, 0) circle [x radius= 3.35, y radius= 3.35]   ;
\draw    (391.41,114.49) -- (670.66,245.41) ;
\draw [shift={(391.41,114.49)}, rotate = 25.12] [color={rgb, 255:red, 0; green, 0; blue, 0 }  ][fill={rgb, 255:red, 0; green, 0; blue, 0 }  ][line width=0.75]      (0, 0) circle [x radius= 3.35, y radius= 3.35]   ;
\draw    (517.25,71.64) -- (544.96,201.01) ;
\draw [shift={(544.96,201.01)}, rotate = 77.91] [color={rgb, 255:red, 0; green, 0; blue, 0 }  ][fill={rgb, 255:red, 0; green, 0; blue, 0 }  ][line width=0.75]      (0, 0) circle [x radius= 3.35, y radius= 3.35]   ;
\draw [shift={(531.11,136.32)}, rotate = 257.91] [color={rgb, 255:red, 0; green, 0; blue, 0 }  ][line width=0.75]    (0,5.59) -- (0,-5.59)   ;
\draw  [draw opacity=0] (519.23,84) .. controls (518.58,84.11) and (517.92,84.16) .. (517.25,84.16) .. controls (511.88,84.16) and (507.31,80.61) .. (505.68,75.67) -- (517.25,71.64) -- cycle ; \draw   (519.23,84) .. controls (518.58,84.11) and (517.92,84.16) .. (517.25,84.16) .. controls (511.88,84.16) and (507.31,80.61) .. (505.68,75.67) ;  
\draw  [draw opacity=0] (403.14,110.94) .. controls (403.46,112.07) and (403.63,113.26) .. (403.63,114.49) .. controls (403.63,116.41) and (403.21,118.24) .. (402.45,119.86) -- (391.41,114.49) -- cycle ; \draw   (403.14,110.94) .. controls (403.46,112.07) and (403.63,113.26) .. (403.63,114.49) .. controls (403.63,116.41) and (403.21,118.24) .. (402.45,119.86) ;  
\draw  [draw opacity=0] (530.34,179.43) .. controls (531.8,175.6) and (535.05,172.68) .. (539.03,171.75) -- (541.73,183.97) -- cycle ; \draw   (530.34,179.43) .. controls (531.8,175.6) and (535.05,172.68) .. (539.03,171.75) ;  
\draw  [draw opacity=0][line width=1.5]  (196.18,118.53) .. controls (196.19,118.86) and (196.19,119.19) .. (196.19,119.53) .. controls (196.19,149.29) and (172.06,173.42) .. (142.29,173.42) .. controls (133.76,173.42) and (125.69,171.44) .. (118.51,167.91) -- (142.29,119.53) -- cycle ; \draw  [line width=1.5]  (196.18,118.53) .. controls (196.19,118.86) and (196.19,119.19) .. (196.19,119.53) .. controls (196.19,149.29) and (172.06,173.42) .. (142.29,173.42) .. controls (133.76,173.42) and (125.69,171.44) .. (118.51,167.91) ;  
\draw  [draw opacity=0][line width=1.5]  (196.18,118.53) .. controls (196.18,118.52) and (196.18,118.51) .. (196.18,118.5) .. controls (196.18,88.73) and (220.31,64.6) .. (250.08,64.6) .. controls (279.84,64.6) and (303.98,88.73) .. (303.98,118.5) .. controls (303.98,146.21) and (283.07,169.03) .. (256.18,172.06) -- (250.08,118.5) -- cycle ; \draw  [line width=1.5]  (196.18,118.53) .. controls (196.18,118.52) and (196.18,118.51) .. (196.18,118.5) .. controls (196.18,88.73) and (220.31,64.6) .. (250.08,64.6) .. controls (279.84,64.6) and (303.98,88.73) .. (303.98,118.5) .. controls (303.98,146.21) and (283.07,169.03) .. (256.18,172.06) ;  
\draw  [draw opacity=0][line width=1.5]  (315.51,226.03) .. controls (315.51,226.05) and (315.51,226.08) .. (315.51,226.1) .. controls (315.51,255.87) and (291.38,280) .. (261.61,280) .. controls (231.85,280) and (207.72,255.87) .. (207.72,226.1) .. controls (207.72,198.01) and (229.2,174.94) .. (256.64,172.43) -- (261.61,226.1) -- cycle ; \draw  [line width=1.5]  (315.51,226.03) .. controls (315.51,226.05) and (315.51,226.08) .. (315.51,226.1) .. controls (315.51,255.87) and (291.38,280) .. (261.61,280) .. controls (231.85,280) and (207.72,255.87) .. (207.72,226.1) .. controls (207.72,198.01) and (229.2,174.94) .. (256.64,172.43) ;  
\draw  [draw opacity=0][line width=1.5]  (315.53,226.06) .. controls (315.53,225.89) and (315.53,225.72) .. (315.53,225.55) .. controls (315.53,220.29) and (316.28,215.2) .. (317.69,210.4) -- (369.43,225.55) -- cycle ; \draw  [line width=1.5]  (315.53,226.06) .. controls (315.53,225.89) and (315.53,225.72) .. (315.53,225.55) .. controls (315.53,220.29) and (316.28,215.2) .. (317.69,210.4) ;  
\draw  [dash pattern={on 0.84pt off 2.51pt}]  (142.29,119.53) -- (119,168) ;
\draw  [dash pattern={on 0.84pt off 2.51pt}]  (142.29,119.53) -- (250.12,118.53) ;
\draw [shift={(250.12,118.53)}, rotate = 359.47] [color={rgb, 255:red, 0; green, 0; blue, 0 }  ][fill={rgb, 255:red, 0; green, 0; blue, 0 }  ][line width=0.75]      (0, 0) circle [x radius= 3.35, y radius= 3.35]   ;
\draw [shift={(142.29,119.53)}, rotate = 359.47] [color={rgb, 255:red, 0; green, 0; blue, 0 }  ][fill={rgb, 255:red, 0; green, 0; blue, 0 }  ][line width=0.75]      (0, 0) circle [x radius= 3.35, y radius= 3.35]   ;
\draw  [dash pattern={on 0.84pt off 2.51pt}]  (250.08,118.5) -- (303.69,118.46) ;
\draw  [dash pattern={on 0.84pt off 2.51pt}]  (250.12,118.53) -- (261.74,226.05) ;
\draw  [dash pattern={on 0.84pt off 2.51pt}]  (261.61,226.1) -- (369.45,226.03) ;
\draw [shift={(369.45,226.03)}, rotate = 359.96] [color={rgb, 255:red, 0; green, 0; blue, 0 }  ][fill={rgb, 255:red, 0; green, 0; blue, 0 }  ][line width=0.75]      (0, 0) circle [x radius= 3.35, y radius= 3.35]   ;
\draw [shift={(261.61,226.1)}, rotate = 359.96] [color={rgb, 255:red, 0; green, 0; blue, 0 }  ][fill={rgb, 255:red, 0; green, 0; blue, 0 }  ][line width=0.75]      (0, 0) circle [x radius= 3.35, y radius= 3.35]   ;
\draw  [dash pattern={on 0.84pt off 2.51pt}]  (369.45,226.03) -- (267.2,189.91) ;
\draw  [draw opacity=0] (159.24,118.72) .. controls (159.25,118.99) and (159.26,119.26) .. (159.26,119.53) .. controls (159.26,122.15) and (158.69,124.63) .. (157.68,126.86) -- (142.29,119.53) -- cycle ; \draw   (159.24,118.72) .. controls (159.25,118.99) and (159.26,119.26) .. (159.26,119.53) .. controls (159.26,122.15) and (158.69,124.63) .. (157.68,126.86) ;  
\draw  [draw opacity=0] (251.48,135.76) .. controls (250.91,135.82) and (250.33,135.85) .. (249.75,135.85) .. controls (240.56,135.85) and (233.08,128.36) .. (232.79,119.01) -- (249.75,118.46) -- cycle ; \draw   (251.48,135.76) .. controls (250.91,135.82) and (250.33,135.85) .. (249.75,135.85) .. controls (240.56,135.85) and (233.08,128.36) .. (232.79,119.01) ;  
\draw  [draw opacity=0] (352.47,226.22) .. controls (352.47,226) and (352.46,225.77) .. (352.46,225.55) .. controls (352.46,223.88) and (352.69,222.28) .. (353.11,220.75) -- (369.43,225.55) -- cycle ; \draw   (352.47,226.22) .. controls (352.47,226) and (352.46,225.77) .. (352.46,225.55) .. controls (352.46,223.88) and (352.69,222.28) .. (353.11,220.75) ;  
\draw    (199.76,222.19) ;
\draw  [dash pattern={on 0.84pt off 2.51pt}]  (207.35,226.06) -- (261.61,226.1) ;
\draw  [draw opacity=0][dash pattern={on 5.63pt off 4.5pt}][line width=1.5]  (192.82,100.2) .. controls (194.95,105.91) and (196.13,112.09) .. (196.18,118.53) -- (142.28,118.97) -- cycle ; \draw  [dash pattern={on 5.63pt off 4.5pt}][line width=1.5]  (192.82,100.2) .. controls (194.95,105.91) and (196.13,112.09) .. (196.18,118.53) ;  
\draw  [draw opacity=0][dash pattern={on 5.63pt off 4.5pt}][line width=1.5]  (193.19,100.73) .. controls (191.61,95.66) and (190.76,90.26) .. (190.76,84.67) .. controls (190.76,54.9) and (214.89,30.77) .. (244.65,30.77) .. controls (274.42,30.77) and (298.55,54.9) .. (298.55,84.67) .. controls (298.55,110.73) and (280.05,132.48) .. (255.46,137.48) -- (244.65,84.67) -- cycle ; \draw  [dash pattern={on 5.63pt off 4.5pt}][line width=1.5]  (193.19,100.73) .. controls (191.61,95.66) and (190.76,90.26) .. (190.76,84.67) .. controls (190.76,54.9) and (214.89,30.77) .. (244.65,30.77) .. controls (274.42,30.77) and (298.55,54.9) .. (298.55,84.67) .. controls (298.55,110.73) and (280.05,132.48) .. (255.46,137.48) ;  
\draw  [draw opacity=0][dash pattern={on 5.63pt off 4.5pt}][line width=1.5]  (318.23,207.3) .. controls (310.99,228.53) and (290.88,243.81) .. (267.2,243.81) .. controls (237.43,243.81) and (213.3,219.68) .. (213.3,189.91) .. controls (213.3,163.85) and (231.8,142.11) .. (256.38,137.1) -- (267.2,189.91) -- cycle ; \draw  [dash pattern={on 5.63pt off 4.5pt}][line width=1.5]  (318.23,207.3) .. controls (310.99,228.53) and (290.88,243.81) .. (267.2,243.81) .. controls (237.43,243.81) and (213.3,219.68) .. (213.3,189.91) .. controls (213.3,163.85) and (231.8,142.11) .. (256.38,137.1) ;  
\draw  [dash pattern={on 0.84pt off 2.51pt}]  (142.28,118.97) -- (244.65,84.67) ;
\draw [shift={(244.65,84.67)}, rotate = 341.47] [color={rgb, 255:red, 0; green, 0; blue, 0 }  ][fill={rgb, 255:red, 0; green, 0; blue, 0 }  ][line width=0.75]      (0, 0) circle [x radius= 3.35, y radius= 3.35]   ;
\draw  [dash pattern={on 0.84pt off 2.51pt}]  (142.29,119.53) -- (369.45,226.03) ;
\draw  [draw opacity=0] (158.16,113.37) .. controls (158.87,115.28) and (159.26,117.36) .. (159.26,119.53) .. controls (159.26,119.88) and (159.25,120.23) .. (159.23,120.57) -- (142.29,119.53) -- cycle ; \draw   (158.16,113.37) .. controls (158.87,115.28) and (159.26,117.36) .. (159.26,119.53) .. controls (159.26,119.88) and (159.25,120.23) .. (159.23,120.57) ;  
\draw  [dash pattern={on 0.84pt off 2.51pt}]  (244.65,84.67) -- (267.2,189.91) ;
\draw [shift={(267.2,189.91)}, rotate = 77.91] [color={rgb, 255:red, 0; green, 0; blue, 0 }  ][fill={rgb, 255:red, 0; green, 0; blue, 0 }  ][line width=0.75]      (0, 0) circle [x radius= 3.35, y radius= 3.35]   ;
\draw  [draw opacity=0] (247.4,101.83) .. controls (246.51,101.98) and (245.59,102.06) .. (244.65,102.06) .. controls (237.19,102.06) and (230.86,97.12) .. (228.59,90.26) -- (244.65,84.67) -- cycle ; \draw   (247.4,101.83) .. controls (246.51,101.98) and (245.59,102.06) .. (244.65,102.06) .. controls (237.19,102.06) and (230.86,97.12) .. (228.59,90.26) ;  
\draw    (428.06,230.7) ;

\draw (638.17,211.26) node [anchor=north west][inner sep=0.75pt]    {$\theta _{0} -\alpha _{3}$};
\draw (680.62,235.6) node [anchor=north west][inner sep=0.75pt]    {$c_{3}$};
\draw (381.14,93.77) node [anchor=north west][inner sep=0.75pt]    {$c_{0}$};
\draw (410.01,106.07) node [anchor=north west][inner sep=0.75pt]    {$\beta _{0} +\theta _{0}$};
\draw (498.09,87.09) node [anchor=north west][inner sep=0.75pt]    {$\phi $};
\draw (534.66,207.14) node [anchor=north west][inner sep=0.75pt]    {$\tilde{c}_{2}$};
\draw (521.72,158.38) node [anchor=north west][inner sep=0.75pt]    {$\psi $};
\draw (454.95,150.86) node [anchor=north west][inner sep=0.75pt]    {$p$};
\draw (528.4,188.59) node [anchor=north west][inner sep=0.75pt]    {$q$};
\draw (221.91,127.21) node [anchor=north west][inner sep=0.75pt]    {$\overline{\gamma }$};
\draw (320.9,215.0) node [anchor=north west][inner sep=0.75pt]    {$\alpha _{3}$};
\draw (361.96,230.44) node [anchor=north west][inner sep=0.75pt]    {$c_{3}$};
\draw (124.57,100.65) node [anchor=north west][inner sep=0.75pt]    {$c_{0}$};
\draw (172.47,120.01) node [anchor=north west][inner sep=0.75pt]    {$\theta _{0}$};
\draw (175.1,102.29) node [anchor=north west][inner sep=0.75pt]    {$\beta _{0}$};
\draw (220.83,95.82) node [anchor=north west][inner sep=0.75pt]    {$\phi $};
\draw (261.2,195.31) node [anchor=north west][inner sep=0.75pt]    {$\tilde{c}_{2}$};

\end{tikzpicture}


        \caption{ Left: trajectory of type $CCCC$ (solid) which satisfies the necessary conditions for optimality, and another of type $CCC$ (dashed) which we prove is shorter, and can therefore be extended to a trajectory of type $CCSC$ (as in Figure \ref{fig:CCSC}) with the correct length. Right: enlargement of the internal triangles from the $CCC$ trajectory. }\label{fig:CCC}
    \end{figure}

The length of the original trajectory is $L_1=\alpha_0+2\pi+2\gamma+\alpha_3 $, and from \eqref{newlength} the length of the $CCC$ trajectory is $L_2 = \alpha_0+2\beta_1+\theta_f-\pi/2$. Hence, using $\alpha_3=\frac{\pi}{2}-\theta_f$
\[
L_1-L_2=2\pi +2\gamma+2\alpha_3-2\beta_1
\]
and our aim is to prove that this quantity is positive. With the angles labelled as in Figure \ref{fig:CCC}, noting that $\gamma =\pi-\bar{\gamma} $ and $\beta_1=2\pi-\phi$, we need to prove that $\bar \gamma < \phi +\alpha_3$. For this, defining $s:=|c_3-c_0|$ note that 
\[
|\tilde c_2- c_0|=8-8\cos\phi = 4+s^2-2s\cos(\theta_0-\alpha_3)
\]
and that this is still true if $\alpha_3>\theta_0$. Supposing we allow $\alpha_3$ to vary, then we determine that 
\[
\frac{d\phi}{d\beta_0}=-\frac{s\sin(\theta_0-\alpha_3)}{4 \sin\phi}\frac{d\alpha_3}{d\beta_0}
\]
Note that as $\beta_0$ increases so too does $\alpha_3$, and so $\frac{d\phi}{d\beta_0}$ is negative for $\alpha_3<\theta_0$ and positive for $\alpha_3>\theta_0$. It follows that
\[
\frac{d}{d\beta_0}(\phi+\alpha_3)=\frac{d\phi}{d\beta_0}\left (1-\frac{s\sin(\theta_0-\alpha_3)}{4 \sin\phi}\right ).
\]
Since $\phi+\alpha_3=\bar \gamma$ at $\beta_0=0$, if we can show that the above quantity is always positive then we have the required result: $\bar \gamma<\phi+\alpha_3$. When $\alpha_3>\theta_0$ this is immediate: all the quantities on the right hand side are positive. For $\alpha_3<\theta_0$ we refer to the right hand side of Figure \ref{fig:CCC} and observe that
\[
\frac{\sin(\theta_0-\alpha_3)}{q}=\frac{\sin\psi}{2}=\frac{\sin\phi}{p}
\]
and therefore $\frac{\sin(\theta_0-\alpha_3)}{\sin(\phi)}=\frac{q}{p}$. Now $\frac{d}{d\beta_0}(\phi+\alpha_3)$ will be positive if $4p>qs$. Comparing with Figure \ref{fig:CCC}, we know $q<1$ and since $\bar \gamma>\pi/2$ (by assumption) we have that $p>s/4$,  so we are done with this case. 

{\bf Case 1b:} $\alpha_3\geq \pi/2$ and $\gamma<\pi/2$. In this case we cannot reflect $c_3$ about $\theta_f$. We set $\tilde c_2$ to $c_3+2i$ instead (i.e. vertically above $c_3$), and then construct a trajectory of type $CCSCC$. This is not very different to the previous case.

\begin{figure}[h]
    \begin{center}


\tikzset{every picture/.style={line width=0.75pt}} 

\begin{tikzpicture}[x=0.75pt,y=0.75pt,yscale=-1,xscale=1]

\draw  [dash pattern={on 0.84pt off 2.51pt}] (234.35,155.08) .. controls (234.35,125.29) and (258.5,101.15) .. (288.29,101.15) .. controls (318.08,101.15) and (342.23,125.29) .. (342.23,155.08) .. controls (342.23,184.87) and (318.08,209.02) .. (288.29,209.02) .. controls (258.5,209.02) and (234.35,184.87) .. (234.35,155.08) -- cycle ;
\draw  [draw opacity=0][line width=1.5]  (342.18,154.09) .. controls (342.19,154.42) and (342.19,154.75) .. (342.19,155.08) .. controls (342.19,184.85) and (318.06,208.98) .. (288.29,208.98) .. controls (279.76,208.98) and (271.69,207) .. (264.51,203.47) -- (288.29,155.08) -- cycle ; \draw  [line width=1.5]  (342.18,154.09) .. controls (342.19,154.42) and (342.19,154.75) .. (342.19,155.08) .. controls (342.19,184.85) and (318.06,208.98) .. (288.29,208.98) .. controls (279.76,208.98) and (271.69,207) .. (264.51,203.47) ;  
\draw  [dash pattern={on 0.84pt off 2.51pt}] (341.81,154.02) .. controls (341.81,124.23) and (365.96,100.08) .. (395.75,100.08) .. controls (425.54,100.08) and (449.69,124.23) .. (449.69,154.02) .. controls (449.69,183.81) and (425.54,207.96) .. (395.75,207.96) .. controls (365.96,207.96) and (341.81,183.81) .. (341.81,154.02) -- cycle ;
\draw  [dash pattern={on 0.84pt off 2.51pt}] (309.19,257.06) .. controls (309.19,227.27) and (333.34,203.12) .. (363.13,203.12) .. controls (392.92,203.12) and (417.07,227.27) .. (417.07,257.06) .. controls (417.07,286.85) and (392.92,311) .. (363.13,311) .. controls (333.34,311) and (309.19,286.85) .. (309.19,257.06) -- cycle ;
\draw  [dash pattern={on 0.84pt off 2.51pt}] (416.74,257.02) .. controls (416.74,227.23) and (440.89,203.08) .. (470.68,203.08) .. controls (500.47,203.08) and (524.62,227.23) .. (524.62,257.02) .. controls (524.62,286.81) and (500.47,310.96) .. (470.68,310.96) .. controls (440.89,310.96) and (416.74,286.81) .. (416.74,257.02) -- cycle ;
\draw    (265,203.56) -- (301.3,220.64) ;
\draw [shift={(303.11,221.49)}, rotate = 205.2] [color={rgb, 255:red, 0; green, 0; blue, 0 }  ][line width=0.75]    (10.93,-3.29) .. controls (6.95,-1.4) and (3.31,-0.3) .. (0,0) .. controls (3.31,0.3) and (6.95,1.4) .. (10.93,3.29)   ;
\draw [shift={(265,203.56)}, rotate = 25.2] [color={rgb, 255:red, 0; green, 0; blue, 0 }  ][fill={rgb, 255:red, 0; green, 0; blue, 0 }  ][line width=0.75]      (0, 0) circle [x radius= 3.35, y radius= 3.35]   ;
\draw    (456.51,205.01) -- (500.42,195.61) ;
\draw [shift={(502.38,195.19)}, rotate = 167.92] [color={rgb, 255:red, 0; green, 0; blue, 0 }  ][line width=0.75]    (10.93,-3.29) .. controls (6.95,-1.4) and (3.31,-0.3) .. (0,0) .. controls (3.31,0.3) and (6.95,1.4) .. (10.93,3.29)   ;
\draw [shift={(456.51,205.01)}, rotate = 347.92] [color={rgb, 255:red, 0; green, 0; blue, 0 }  ][fill={rgb, 255:red, 0; green, 0; blue, 0 }  ][line width=0.75]      (0, 0) circle [x radius= 3.35, y radius= 3.35]   ;
\draw  [draw opacity=0][line width=1.5]  (342.18,154.09) .. controls (342.18,154.08) and (342.18,154.07) .. (342.18,154.06) .. controls (342.18,124.29) and (366.31,100.16) .. (396.08,100.16) .. controls (425.84,100.16) and (449.98,124.29) .. (449.98,154.06) .. controls (449.98,183.83) and (425.84,207.96) .. (396.08,207.96) .. controls (390.26,207.96) and (384.66,207.03) .. (379.41,205.33) -- (396.08,154.06) -- cycle ; \draw  [line width=1.5]  (342.18,154.09) .. controls (342.18,154.08) and (342.18,154.07) .. (342.18,154.06) .. controls (342.18,124.29) and (366.31,100.16) .. (396.08,100.16) .. controls (425.84,100.16) and (449.98,124.29) .. (449.98,154.06) .. controls (449.98,183.83) and (425.84,207.96) .. (396.08,207.96) .. controls (390.26,207.96) and (384.66,207.03) .. (379.41,205.33) ;  
\draw  [draw opacity=0][line width=1.5]  (417.03,257.6) .. controls (416.74,287.12) and (392.72,310.96) .. (363.13,310.96) .. controls (333.36,310.96) and (309.23,286.83) .. (309.23,257.06) .. controls (309.23,227.29) and (333.36,203.16) .. (363.13,203.16) .. controls (369.33,203.16) and (375.28,204.21) .. (380.82,206.13) -- (363.13,257.06) -- cycle ; \draw  [line width=1.5]  (417.03,257.6) .. controls (416.74,287.12) and (392.72,310.96) .. (363.13,310.96) .. controls (333.36,310.96) and (309.23,286.83) .. (309.23,257.06) .. controls (309.23,227.29) and (333.36,203.16) .. (363.13,203.16) .. controls (369.33,203.16) and (375.28,204.21) .. (380.82,206.13) ;  
\draw  [draw opacity=0][line width=1.5]  (416.78,257.54) .. controls (416.78,257.37) and (416.78,257.19) .. (416.78,257.02) .. controls (416.78,232.16) and (433.61,211.23) .. (456.51,205.01) -- (470.68,257.02) -- cycle ; \draw  [line width=1.5]  (416.78,257.54) .. controls (416.78,257.37) and (416.78,257.19) .. (416.78,257.02) .. controls (416.78,232.16) and (433.61,211.23) .. (456.51,205.01) ;  
\draw  [dash pattern={on 0.84pt off 2.51pt}]  (288.29,155.08) -- (265,203.56) ;
\draw  [dash pattern={on 0.84pt off 2.51pt}]  (288.29,155.08) -- (396.12,154.09) ;
\draw [shift={(396.12,154.09)}, rotate = 359.47] [color={rgb, 255:red, 0; green, 0; blue, 0 }  ][fill={rgb, 255:red, 0; green, 0; blue, 0 }  ][line width=0.75]      (0, 0) circle [x radius= 3.35, y radius= 3.35]   ;
\draw [shift={(288.29,155.08)}, rotate = 359.47] [color={rgb, 255:red, 0; green, 0; blue, 0 }  ][fill={rgb, 255:red, 0; green, 0; blue, 0 }  ][line width=0.75]      (0, 0) circle [x radius= 3.35, y radius= 3.35]   ;
\draw  [dash pattern={on 0.84pt off 2.51pt}]  (396.08,154.06) -- (449.69,154.02) ;
\draw  [dash pattern={on 0.84pt off 2.51pt}]  (396.12,154.09) -- (363.13,257.06) ;
\draw  [dash pattern={on 0.84pt off 2.51pt}]  (363.13,257.06) -- (470.68,257.02) ;
\draw [shift={(470.68,257.02)}, rotate = 359.98] [color={rgb, 255:red, 0; green, 0; blue, 0 }  ][fill={rgb, 255:red, 0; green, 0; blue, 0 }  ][line width=0.75]      (0, 0) circle [x radius= 3.35, y radius= 3.35]   ;
\draw [shift={(363.13,257.06)}, rotate = 359.98] [color={rgb, 255:red, 0; green, 0; blue, 0 }  ][fill={rgb, 255:red, 0; green, 0; blue, 0 }  ][line width=0.75]      (0, 0) circle [x radius= 3.35, y radius= 3.35]   ;
\draw  [dash pattern={on 0.84pt off 2.51pt}]  (456.51,205.01) -- (470.68,257.02) ;
\draw  [draw opacity=0] (305.24,154.28) .. controls (305.25,154.54) and (305.26,154.81) .. (305.26,155.08) .. controls (305.26,164.69) and (297.66,172.47) .. (288.29,172.47) .. controls (285.54,172.47) and (282.95,171.8) .. (280.66,170.62) -- (288.29,155.08) -- cycle ; \draw   (305.24,154.28) .. controls (305.25,154.54) and (305.26,154.81) .. (305.26,155.08) .. controls (305.26,164.69) and (297.66,172.47) .. (288.29,172.47) .. controls (285.54,172.47) and (282.95,171.8) .. (280.66,170.62) ;  
\draw  [draw opacity=0] (412.71,154.02) .. controls (412.71,163.62) and (405.12,171.41) .. (395.75,171.41) .. controls (393.98,171.41) and (392.28,171.13) .. (390.67,170.62) -- (395.75,154.02) -- cycle ; \draw    (412.71,154.02) .. controls (412.71,163.62) and (405.12,171.41) .. (395.75,171.41) .. controls (393.98,171.41) and (392.28,171.13) .. (390.67,170.62) ;  
\draw  [draw opacity=0] (346.5,256.81) .. controls (346.5,256.81) and (346.5,256.81) .. (346.5,256.81) .. controls (346.66,247.21) and (354.39,239.55) .. (363.76,239.71) .. controls (365.48,239.74) and (367.13,240.03) .. (368.68,240.55) -- (363.46,257.1) -- cycle ; \draw   (346.5,256.81) .. controls (346.5,256.81) and (346.5,256.81) .. (346.5,256.81) .. controls (346.66,247.21) and (354.39,239.55) .. (363.76,239.71) .. controls (365.48,239.74) and (367.13,240.03) .. (368.68,240.55) ;  
\draw  [draw opacity=0] (453.71,256.99) .. controls (453.72,248.94) and (459.08,242.17) .. (466.34,240.21) -- (470.68,257.02) -- cycle ; \draw   (453.71,256.99) .. controls (453.72,248.94) and (459.08,242.17) .. (466.34,240.21) ;  
\draw    (345.76,257.75) ;
\draw  [dash pattern={on 0.84pt off 2.51pt}]  (309.19,257.06) -- (363.46,257.1) ;
\draw  [draw opacity=0][dash pattern={on 5.63pt off 4.5pt}][line width=1.5]  (318.15,109.83) .. controls (332.32,119.45) and (341.66,135.64) .. (341.81,154.02) -- (287.91,154.46) -- cycle ; \draw  [color={rgb, 255:red, 0; green, 0; blue, 0 }  ,draw opacity=1 ][dash pattern={on 5.63pt off 4.5pt}][line width=1.5]  (318.15,109.83) .. controls (332.32,119.45) and (341.66,135.64) .. (341.81,154.02) ;  
\draw  [dash pattern={on 0.84pt off 2.51pt}] (291.37,63.22) .. controls (291.37,33.43) and (315.52,9.28) .. (345.31,9.28) .. controls (375.1,9.28) and (399.24,33.43) .. (399.24,63.22) .. controls (399.24,93.01) and (375.1,117.16) .. (345.31,117.16) .. controls (315.52,117.16) and (291.37,93.01) .. (291.37,63.22) -- cycle ;
\draw  [draw opacity=0][dash pattern={on 5.63pt off 4.5pt}][line width=1.5]  (317.84,109.61) .. controls (302.02,100.22) and (291.41,82.96) .. (291.41,63.22) .. controls (291.41,33.46) and (315.54,9.32) .. (345.31,9.32) .. controls (363.84,9.32) and (380.18,18.68) .. (389.88,32.92) -- (345.31,63.22) -- cycle ; \draw  [dash pattern={on 5.63pt off 4.5pt}][line width=1.5]  (317.84,109.61) .. controls (302.02,100.22) and (291.41,82.96) .. (291.41,63.22) .. controls (291.41,33.46) and (315.54,9.32) .. (345.31,9.32) .. controls (363.84,9.32) and (380.18,18.68) .. (389.88,32.92) ;  
\draw  [draw opacity=0][dash pattern={on 5.63pt off 4.5pt}][line width=1.5]  (515.91,227.7) .. controls (521.39,236.14) and (524.58,246.21) .. (524.58,257.02) .. controls (524.58,286.79) and (500.44,310.92) .. (470.68,310.92) .. controls (440.91,310.92) and (416.78,286.79) .. (416.78,257.02) .. controls (416.78,257.01) and (416.78,257.01) .. (416.78,257) -- (470.68,257.02) -- cycle ; \draw  [dash pattern={on 5.63pt off 4.5pt}][line width=1.5]  (515.91,227.7) .. controls (521.39,236.14) and (524.58,246.21) .. (524.58,257.02) .. controls (524.58,286.79) and (500.44,310.92) .. (470.68,310.92) .. controls (440.91,310.92) and (416.78,286.79) .. (416.78,257.02) .. controls (416.78,257.01) and (416.78,257.01) .. (416.78,257) ;  
\draw  [dash pattern={on 0.84pt off 2.51pt}]  (287.91,154.46) -- (345.31,63.22) ;
\draw [shift={(345.31,63.22)}, rotate = 302.17] [color={rgb, 255:red, 0; green, 0; blue, 0 }  ][fill={rgb, 255:red, 0; green, 0; blue, 0 }  ][line width=0.75]      (0, 0) circle [x radius= 3.35, y radius= 3.35]   ;
\draw [line width=1.5]  [dash pattern={on 5.63pt off 4.5pt}]  (389.88,32.92) -- (515.91,227.7) ;
\draw  [dash pattern={on 0.84pt off 2.51pt}]  (470.68,257.02) -- (515.91,227.7) ;
\draw  [draw opacity=0] (296.37,139.79) .. controls (301.66,142.73) and (305.26,148.48) .. (305.26,155.08) .. controls (305.26,155.44) and (305.25,155.78) .. (305.23,156.13) -- (288.29,155.08) -- cycle ; \draw   (296.37,139.79) .. controls (301.66,142.73) and (305.26,148.48) .. (305.26,155.08) .. controls (305.26,155.44) and (305.25,155.78) .. (305.23,156.13) ;  
\draw  [draw opacity=0] (336.08,77.82) .. controls (331.42,74.72) and (328.34,69.34) .. (328.34,63.22) .. controls (328.34,53.62) and (335.94,45.83) .. (345.31,45.83) .. controls (350.42,45.83) and (355.01,48.16) .. (358.12,51.83) -- (345.31,63.22) -- cycle ; \draw   (336.08,77.82) .. controls (331.42,74.72) and (328.34,69.34) .. (328.34,63.22) .. controls (328.34,53.62) and (335.94,45.83) .. (345.31,45.83) .. controls (350.42,45.83) and (355.01,48.16) .. (358.12,51.83) ;  
\draw  [draw opacity=0] (485.02,247.71) .. controls (486.73,250.48) and (487.7,253.79) .. (487.64,257.31) .. controls (487.47,266.92) and (479.75,274.57) .. (470.38,274.41) .. controls (461.01,274.25) and (453.55,266.33) .. (453.71,256.73) .. controls (453.75,254.78) and (454.09,252.9) .. (454.7,251.16) -- (470.68,257.02) -- cycle ; \draw   (485.02,247.71) .. controls (486.73,250.48) and (487.7,253.79) .. (487.64,257.31) .. controls (487.47,266.92) and (479.75,274.57) .. (470.38,274.41) .. controls (461.01,274.25) and (453.55,266.33) .. (453.71,256.73) .. controls (453.75,254.78) and (454.09,252.9) .. (454.7,251.16) ;  
\draw  [dash pattern={on 0.84pt off 2.51pt}]  (345.31,63.22) -- (389.88,32.92) ;
\draw   (378.96,40.77) -- (389.88,32.92) -- (397.74,43.84) -- (386.81,51.69) -- cycle ;
\draw   (497.32,223.69) -- (508.62,216.4) -- (515.91,227.7) -- (504.61,234.99) -- cycle ;
\draw  [draw opacity=0] (355.69,49.46) .. controls (357.26,50.72) and (358.62,52.24) .. (359.68,53.98) -- (345.31,63.22) -- cycle ; \draw   (355.69,49.46) .. controls (357.26,50.72) and (358.62,52.24) .. (359.68,53.98) ;  

\draw (236.72,215.7) node [anchor=north west][inner sep=0.75pt]    {$x_{0} ,\theta _{0}$};
\draw (510.43,183.52) node [anchor=north west][inner sep=0.75pt]    {$x_{f} ,\theta _{f}$};
\draw (298.71,170.68) node [anchor=north west][inner sep=0.75pt]    {$\alpha _{0}$};
\draw (409.91,168.77) node [anchor=north west][inner sep=0.75pt]    {$\gamma $};
\draw (341.64,225.81) node [anchor=north west][inner sep=0.75pt]    {$\gamma $};
\draw (434.47,224.67) node [anchor=north west][inner sep=0.75pt]    {$\alpha _{3}$};
\draw (389.39,129.21) node [anchor=north west][inner sep=0.75pt]    {$c_{1}$};
\draw (347.76,261.15) node [anchor=north west][inner sep=0.75pt]    {$c_{2}$};
\draw (471.28,222.91) node [anchor=north west][inner sep=0.75pt]    {$c_{3}$};
\draw (270.57,129.21) node [anchor=north west][inner sep=0.75pt]    {$c_{0}$};
\draw (347.31,66.62) node [anchor=north west][inner sep=0.75pt]    {$\tilde{c}_{1}$};
\draw (315.47,30.57) node [anchor=north west][inner sep=0.75pt]    {$\beta _{1}$};
\draw (312.1,125.84) node [anchor=north west][inner sep=0.75pt]    {$\beta _{0}$};
\draw (470.57,277.9) node [anchor=north west][inner sep=0.75pt]    {$\beta _{2}$};
\draw (461.45,114.94) node [anchor=north west][inner sep=0.75pt]    {$\ell $};

\end{tikzpicture}
        
    \end{center}
    
    \caption{Diagram for Case 2 ($\gamma\geq \frac{\pi}{2}$) in the proof of Lemma \ref{CCCC}: a trajectory of type $CCCC$ (solid) which satisfies the necessary conditions from the maximum principle, and another of type $CCSC$ (dashed) with the same maximum curvature which does not.}
    \label{fig:case2}
\end{figure}
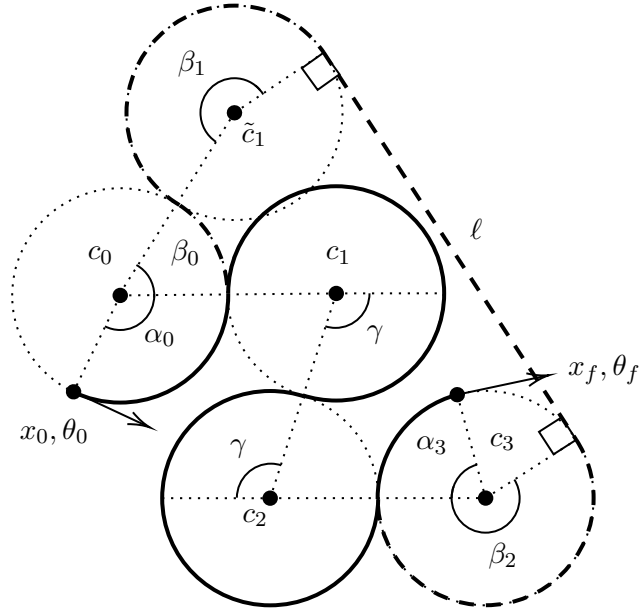

{\bf Case 2:} $\gamma\geq \pi/2$. Referring to Figure \ref{fig:case2}, let $\tilde{w}:=c_3-\tilde c_1$ and then we have
\begin{align*}
    \tilde c_1&=2 e^{i\beta_0} \\
    c_3&=z_f+e^{i(\theta_f-\pi/2)}\\
    l&=|\tilde{w}| \\
    L_2&=\alpha_0+\beta_0+\beta_1+\ell+\beta_2+\alpha_3\\
    &=\alpha_0+2\beta_0+2\pi+|\tilde{w}|+\alpha_3
\end{align*}
Arguing as in case 1, we have that $L_2$ depends continuously on $\beta_0$ and it is enough to show that $L_2<L_1$ when $\beta_0=0$, where $L_1=\alpha_0+2\pi+2\gamma+\alpha_3$ is the length of the original trajectory. Indeed, at $\beta_0$ we have $L_1-L_2=2\gamma-|\tilde{w}|$. Noting that 
$|\tilde{w}|^2=8+8\cos\gamma$
and recalling that we have assumed $\gamma>\pi/2$ we have that $|\tilde{w}|^2\leq 8<(2\gamma)^2$ and therefore $L_1>L_2$ at $\beta_0=0$.
\end{proof}

\begin{theorem}\label{classification}
    Any solution to problem (P) is of type $CXC$ or sub-path 
    thereof, where $X$ can be $S$, $C$, or $O$; or of type $SOS$ or sub-path thereof. The radii of all the $C$ and $O$ arcs in a solution are the same.
\end{theorem}
\begin{proof} 
First note that any path containing more than one loop cannot be optimal. The loops can be brought together without changing the maximum curvature, but a sub-path of type $OO$ is not optimal - it can be expanded into a single loop with lower curvature. 
Consider all paths of type $SXYZ$, where $X,Y,Z$ can each be $C,O, S$ or void, 
\begin{itemize}
    \item $X=C$: $SC$ can be optimal, $SCY$ cannot, by Lemma \ref{fullcircle}.
    \item $X=O$: $SO$ can be optimal and so can $SOS$, but $SOO$ cannot by the argument above, and $SOC$ is excluded by Lemma \ref{OSC}. Similarly, we rule out $SOSO$ and $SOSC$.
    \item A single straight line segment $S$ can be optimal (similarly $C$ and $O$).  
\end{itemize}
By Lemma \ref{subtrajectory} we have exhausted the possibilities for optimal paths beginning with $S$. Next consider paths of type $CXYZ$:
\begin{itemize}
    \item $X=C$: $CC$ and $CCC$ can be optimal, $CCO$ and $CCS$ are not optimal by Lemma \ref{fullcircle}. $CCCC$, $CCCO$ and $CCCS$ are excluded by Lemmas \ref{CCCC}, Corollary \ref{OSC} and Lemma \ref{fullcircle} respectively. 
    \item $X=O$: $CO$ and $COC$ can be optimal, $COS$ is ruled out by Corollary \ref{OSC} along with  $COCC$ and $COCS$.
    \item $X=S$: $CS$ and $CSC$ are possible, $CSO$ and $CSCO$ are not optimal by Corollary \ref{OSC}, and Lemma \ref{fullcircle} rules out $CSCS$ and $CSCC$.
\end{itemize}
Finally, for paths of type $OXYZ$: $OC$ can be optimal but $OCC$ and $OCS$ are not by Lemma \ref{fullcircle} and Corollary \ref{OSC} respectively, $OS$ can be optimal but $OSC$ is not by Lemma \ref{OSC}.

\end{proof}

The boundary conditions that allow for paths of types $SOS$ and $COC$ are quite specific, and so we are able to prove a little more about the conditions under which these paths may be optimal. 

\begin{proposition}\label{SCS}
Consider an $SOS$ path as per Figure \ref{fig:SCS}, and let $d=\Vert z(0)-z(t_f) \Vert$ be the distance between the initial and final points. If $\frac{d}{t_f}>b $ where $b$ is the solution of $\sinc^{-1}(b)(1-b)=\pi/2$, then the path is not optimal. 
\end{proposition}

\begin{proof}

\begin{figure}[ht]
\centering

\tikzset{every picture/.style={line width=0.75pt}} 

\begin{tikzpicture}[x=0.75pt,y=0.75pt,yscale=-1,xscale=1]

\draw [color={rgb, 255:red, 0; green, 0; blue, 0 }  ,draw opacity=1 ][line width=1.5]  [dash pattern={on 5.63pt off 4.5pt}]  (190.5,117) -- (491.5,117) ;
\draw  [color={rgb, 255:red, 0; green, 0; blue, 0 }  ,draw opacity=1 ][dash pattern={on 5.63pt off 4.5pt}][line width=1.5]  (260,72.25) .. controls (260,47.81) and (279.81,28) .. (304.25,28) .. controls (328.69,28) and (348.5,47.81) .. (348.5,72.25) .. controls (348.5,96.69) and (328.69,116.5) .. (304.25,116.5) .. controls (279.81,116.5) and (260,96.69) .. (260,72.25) -- cycle ;
\draw [color={rgb, 255:red, 0; green, 0; blue, 0 }  ,draw opacity=1 ] [dash pattern={on 0.84pt off 2.51pt}]  (304.25,72.25) -- (348.5,72.25) ;
\draw  [draw opacity=0][line width=1.5]  (415.77,168.46) .. controls (427.5,138.53) and (456.52,117.35) .. (490.46,117.35) .. controls (490.49,117.35) and (490.52,117.35) .. (490.55,117.35) -- (490.46,198.01) -- cycle ; \draw  [line width=1.5]  (415.77,168.46) .. controls (427.5,138.53) and (456.52,117.35) .. (490.46,117.35) .. controls (490.49,117.35) and (490.52,117.35) .. (490.55,117.35) ;  
\draw  [draw opacity=0][line width=1.5]  (190.5,117) .. controls (190.58,117) and (190.66,117) .. (190.74,117) .. controls (225.73,117) and (255.5,139.43) .. (266.47,170.71) -- (190.74,197.33) -- cycle ; \draw  [line width=1.5]  (190.5,117) .. controls (190.58,117) and (190.66,117) .. (190.74,117) .. controls (225.73,117) and (255.5,139.43) .. (266.47,170.71) ;  
\draw  [draw opacity=0][line width=1.5]  (415.77,168.46) .. controls (404.67,199.1) and (375.26,221) .. (340.73,221) .. controls (306.86,221) and (277.92,199.93) .. (266.35,170.22) -- (340.73,141.39) -- cycle ; \draw  [line width=1.5]  (415.77,168.46) .. controls (404.67,199.1) and (375.26,221) .. (340.73,221) .. controls (306.86,221) and (277.92,199.93) .. (266.35,170.22) ;  
\draw  [dash pattern={on 0.84pt off 2.51pt}]  (190.74,197.33) -- (190.5,117) ;
\draw [shift={(190.5,117)}, rotate = 269.83] [color={rgb, 255:red, 0; green, 0; blue, 0 }  ][fill={rgb, 255:red, 0; green, 0; blue, 0 }  ][line width=0.75]      (0, 0) circle [x radius= 3.35, y radius= 3.35]   ;
\draw  [dash pattern={on 0.84pt off 2.51pt}]  (190.74,197.33) -- (340.73,141.39) ;
\draw  [dash pattern={on 0.84pt off 2.51pt}]  (490.46,198.01) -- (490.22,117.68) ;
\draw [shift={(490.22,117.68)}, rotate = 269.83] [color={rgb, 255:red, 0; green, 0; blue, 0 }  ][fill={rgb, 255:red, 0; green, 0; blue, 0 }  ][line width=0.75]      (0, 0) circle [x radius= 3.35, y radius= 3.35]   ;
\draw  [dash pattern={on 0.84pt off 2.51pt}]  (340.73,141.39) -- (490.46,198.01) ;
\draw  [draw opacity=0] (190.68,177.8) .. controls (190.7,177.8) and (190.72,177.8) .. (190.74,177.8) .. controls (199.36,177.8) and (206.7,183.2) .. (209.45,190.75) -- (190.74,197.33) -- cycle ; \draw   (190.68,177.8) .. controls (190.7,177.8) and (190.72,177.8) .. (190.74,177.8) .. controls (199.36,177.8) and (206.7,183.2) .. (209.45,190.75) ;  
\draw  [draw opacity=0] (472.2,190.3) .. controls (475.21,183.42) and (482.12,178.59) .. (490.18,178.48) -- (490.46,198.01) -- cycle ; \draw   (472.2,190.3) .. controls (475.21,183.42) and (482.12,178.59) .. (490.18,178.48) ;  
\draw  [draw opacity=0] (359.51,147.77) .. controls (356.83,155.43) and (349.43,160.93) .. (340.73,160.93) .. controls (332.25,160.93) and (325.01,155.7) .. (322.16,148.35) -- (340.73,141.39) -- cycle ; \draw   (359.51,147.77) .. controls (356.83,155.43) and (349.43,160.93) .. (340.73,160.93) .. controls (332.25,160.93) and (325.01,155.7) .. (322.16,148.35) ;  

\draw (183,94.4) node [anchor=north west][inner sep=0.75pt]    {$z_{0}$};
\draw (479,94.4) node [anchor=north west][inner sep=0.75pt]    {$z_{f}$};
\draw (320,51.4) node [anchor=north west][inner sep=0.75pt]    {$R$};
\draw (199,161.4) node [anchor=north west][inner sep=0.75pt]    {$\alpha $};
\draw (469,163.4) node [anchor=north west][inner sep=0.75pt]    {$\alpha $};
\draw (330,165.4) node [anchor=north west][inner sep=0.75pt]    {$2\alpha $};
\draw (171,144.4) node [anchor=north west][inner sep=0.75pt]    {$r$};

\end{tikzpicture}
\caption{Diagram for the proof of Lemma \ref{SCS}: trajectory of type $SCS$, and a trajectory of type $CCC$ with lower maximum curvature.}\label{fig:SCS}
\end{figure}
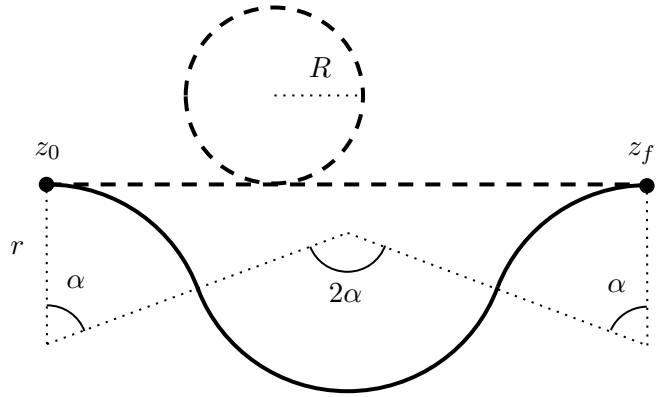
We show that with the assumption on $d/t_f$, the proposed trajectory of type $CCC$ in Figure \ref{fig:SCS} has $r>R$. Note that $d=4r\sin(\alpha)$ (still true if $\alpha$ is obtuse) and $t_f=4r\alpha$ gives $d/t_f=\sinc(\alpha)$. $d/t_f$ is at most $1$ and so $\sinc^{-1}(d/t_f)$ has a solution between $0$ and $\pi$ and  $\alpha=\sinc^{-1}(d/t_f)$. Moreover from $2\pi R+d=t_f$ we then have
\[ R=\frac{t_f-d}{2\pi}=\frac{4r(\alpha-\sin\alpha)}{2\pi} \]
and therefore $r>R$ while $2(\alpha-\sin\alpha)/\pi<1$. Rewriting this inequality in terms of $d/t_f$ gives
\[ \sinc^{-1}\left(\frac{d}{t_f}\right)\left( 1-\frac{d}{t_f}\right)<\frac{\pi}{2} \]
and on the domain $[0,1)$ the left hand side of the inequality is decreasing.

\end{proof}

\begin{remark}  \label{rem:SCS} \rm 
Note that the full circular arc, or the loop, in Figure~\ref{fig:SCS} can be placed anywhere on the straight line $S$ without changing the maximum curvature or the boundary conditions and so the above result also applies to paths of type $OS$ and $SO$.
\endproof
\end{remark}

Note that a path of type $COC$ cannot be optimal if the $C$'s have different orientations. If they do then the loop can be shifted to the terminal point to obtain a path of type $CCO$ with the same maximum curvature, but such a path is not optimal by Lemma \ref{fullcircle}.

\begin{proposition}\label{COC}
Consider a path of type $COC$ and of length $t_f = \beta/a + 2\,\pi$ as in Figure \ref{fig:COC} (solid). Such a path can be optimal only if $\beta \leq \frac{\pi}{2}$. 
\end{proposition}

\begin{proof}
 For comparison we consider the $CCC$ path also pictured in Figure \ref{fig:COC} (dashed). Assuming that the circular arcs all have radius 1, the length of the $COC$ path is $L_1=\beta+2\pi$, and the length of the $CCC$ is $L_2=2\alpha+2\pi-\beta=4\pi-3\beta$. Therefore if $\beta>\frac{\pi}{2}$ then  $L_2<L_1$ and then, adjusting the angles as necessary, the radii of each of the circular arcs in the $CCC$ curve can be increased (reducing the maximum curvature) until the curve has the correct length $\beta+2\pi$, and so the original $COC$ is not a minimizer.
\end{proof}
 
\begin{figure}
    \centering
\tikzset{every picture/.style={line width=0.75pt}} 

\begin{tikzpicture}[x=0.75pt,y=0.75pt,yscale=-1,xscale=1]

\draw  [draw opacity=0][line width=1.5]  (426.05,78.85) .. controls (416.53,94.06) and (399.62,104.17) .. (380.34,104.17) .. controls (362.88,104.17) and (347.36,95.87) .. (337.51,83) -- (380.34,50.27) -- cycle ; \draw  [line width=1.5]  (426.05,78.85) .. controls (416.53,94.06) and (399.62,104.17) .. (380.34,104.17) .. controls (362.88,104.17) and (347.36,95.87) .. (337.51,83) ;  
\draw  [dash pattern={on 0.84pt off 2.51pt}] (416.56,110) .. controls (416.56,80.21) and (440.71,56.06) .. (470.5,56.06) .. controls (500.29,56.06) and (524.44,80.21) .. (524.44,110) .. controls (524.44,139.79) and (500.29,163.94) .. (470.5,163.94) .. controls (440.71,163.94) and (416.56,139.79) .. (416.56,110) -- cycle ;
\draw  [dash pattern={on 0.84pt off 2.51pt}]  (380.34,50.27) -- (290.5,110) ;
\draw [shift={(335.42,80.14)}, rotate = 146.38] [color={rgb, 255:red, 0; green, 0; blue, 0 }  ][fill={rgb, 255:red, 0; green, 0; blue, 0 }  ][line width=0.75]      (0, 0) circle [x radius= 3.35, y radius= 3.35]   ;
\draw  [dash pattern={on 0.84pt off 2.51pt}]  (380.34,50.27) -- (470.5,110) ;
\draw [shift={(425.42,80.14)}, rotate = 33.52] [color={rgb, 255:red, 0; green, 0; blue, 0 }  ][fill={rgb, 255:red, 0; green, 0; blue, 0 }  ][line width=0.75]      (0, 0) circle [x radius= 3.35, y radius= 3.35]   ;
\draw [shift={(380.34,50.27)}, rotate = 33.52] [color={rgb, 255:red, 0; green, 0; blue, 0 }  ][fill={rgb, 255:red, 0; green, 0; blue, 0 }  ][line width=0.75]      (0, 0) circle [x radius= 3.35, y radius= 3.35]   ;
\draw  [draw opacity=0] (300.33,103.12) .. controls (301.63,105.08) and (302.4,107.45) .. (302.4,110) .. controls (302.4,112.45) and (301.69,114.74) .. (300.47,116.65) -- (290.5,110) -- cycle ; \draw   (300.33,103.12) .. controls (301.63,105.08) and (302.4,107.45) .. (302.4,110) .. controls (302.4,112.45) and (301.69,114.74) .. (300.47,116.65) ;  
\draw  [draw opacity=0][dash pattern={on 5.63pt off 4.5pt}][line width=1.5]  (425.95,139.65) .. controls (431.45,148.1) and (434.65,158.19) .. (434.65,169.02) .. controls (434.65,198.79) and (410.52,222.92) .. (380.75,222.92) .. controls (350.98,222.92) and (326.85,198.79) .. (326.85,169.02) .. controls (326.85,157.81) and (330.27,147.4) .. (336.12,138.78) -- (380.75,169.02) -- cycle ; \draw  [color={rgb, 255:red, 0; green, 0; blue, 0 }  ,draw opacity=1 ][dash pattern={on 5.63pt off 4.5pt}][line width=1.5]  (425.95,139.65) .. controls (431.45,148.1) and (434.65,158.19) .. (434.65,169.02) .. controls (434.65,198.79) and (410.52,222.92) .. (380.75,222.92) .. controls (350.98,222.92) and (326.85,198.79) .. (326.85,169.02) .. controls (326.85,157.81) and (330.27,147.4) .. (336.12,138.78) ;  
\draw  [line width=1.5]  (326.41,158.15) .. controls (326.41,128.36) and (350.55,104.21) .. (380.34,104.21) .. controls (410.13,104.21) and (434.28,128.36) .. (434.28,158.15) .. controls (434.28,187.94) and (410.13,212.09) .. (380.34,212.09) .. controls (350.55,212.09) and (326.41,187.94) .. (326.41,158.15) -- cycle ;
\draw  [draw opacity=0][dash pattern={on 5.63pt off 4.5pt}][line width=1.5]  (333.51,77.52) .. controls (340.35,86.55) and (344.4,97.8) .. (344.4,110) .. controls (344.4,121.35) and (340.89,131.89) .. (334.89,140.58) -- (290.5,110) -- cycle ; \draw  [dash pattern={on 5.63pt off 4.5pt}][line width=1.5]  (333.51,77.52) .. controls (340.35,86.55) and (344.4,97.8) .. (344.4,110) .. controls (344.4,121.35) and (340.89,131.89) .. (334.89,140.58) ;  
\draw  [draw opacity=0][dash pattern={on 5.63pt off 4.5pt}][line width=1.5]  (428.51,143.8) .. controls (421.06,134.55) and (416.6,122.8) .. (416.6,110) .. controls (416.6,98.44) and (420.24,87.73) .. (426.44,78.95) -- (470.5,110) -- cycle ; \draw  [dash pattern={on 5.63pt off 4.5pt}][line width=1.5]  (428.51,143.8) .. controls (421.06,134.55) and (416.6,122.8) .. (416.6,110) .. controls (416.6,98.44) and (420.24,87.73) .. (426.44,78.95) ;  
\draw  [draw opacity=0] (391.19,57.73) .. controls (388.85,61.31) and (384.86,63.66) .. (380.34,63.66) .. controls (376.13,63.66) and (372.39,61.62) .. (370,58.45) -- (380.34,50.27) -- cycle ; \draw   (391.19,57.73) .. controls (388.85,61.31) and (384.86,63.66) .. (380.34,63.66) .. controls (376.13,63.66) and (372.39,61.62) .. (370,58.45) ;  
\draw  [dash pattern={on 0.84pt off 2.51pt}] (236.56,110) .. controls (236.56,80.21) and (260.71,56.06) .. (290.5,56.06) .. controls (320.29,56.06) and (344.44,80.21) .. (344.44,110) .. controls (344.44,139.79) and (320.29,163.94) .. (290.5,163.94) .. controls (260.71,163.94) and (236.56,139.79) .. (236.56,110) -- cycle ;
\draw  [dash pattern={on 0.84pt off 2.51pt}]  (470.5,110) -- (380.75,169.02) ;
\draw  [dash pattern={on 0.84pt off 2.51pt}]  (290.5,110) -- (380.75,169.02) ;

\draw (329.72,60.7) node [anchor=north west][inner sep=0.75pt]    {$z_{0}$};
\draw (418.43,60) node [anchor=north west][inner sep=0.75pt]    {$z_{f}$};
\draw (306.71,105.68) node [anchor=north west][inner sep=0.75pt]    {$\alpha $};
\draw (374.3,66.69) node [anchor=north west][inner sep=0.75pt]    {$\beta $};
\end{tikzpicture}

    \caption{Diagram for Proposition \ref{COC}: A path of type $COC$ (solid) and shorter curve of type $CCC$ (dashed) with the same maximum curvature.}
    \label{fig:COC}
\end{figure}
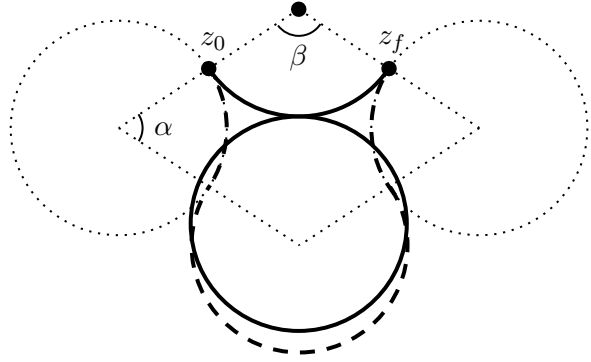

\begin{remark}  \label{rem:COC} \rm
Note that the full circular arc, or the loop, in Figure~\ref{fig:COC} can be placed anywhere on the circular arc $C$ between $z_0$ and $z_f$ without changing the maximum curvature or the boundary conditions and so the above result also applies to paths of type $OC$ and $CO$.
\endproof
\end{remark}

The numerical experiments in Section \ref{sec:numexp}, specifically Examples 1 and 2, suggest that the converses to Propositions \ref{SCS} and \ref{COC} are true. It may be possible to verify this by direct comparison with each of the potential minimizers allowed by Theorem \ref{classification} which satisfy the constraints (at least in the $SOS$ case there do not seem to be too many of these), but we will not attempt to do so here.

\subsection{Proof of Proposition~\ref{existence}} \label{subsec:proof}

\begin{proof}
    The proof is sketched in \cite{gallagher2023elastica} but for the convenience of the reader we fill out the details. For any $p\in [1,\infty)$, consider the $p$-elastic energy $\int_0^{t_f}\Vert \ddot z(t)\Vert^p dt$ for $z$ in the Sobolev space $W^{2,p}((0,t_f),\mathbb R^2)$ and satisfying the constraints of Problem (P). The existence of minimizers of the $p$-elastic energy subject to Dirichlet boundary conditions is proved using the direct method in \cite[Proposition 4.1]{MiuYos2022}, and the same proof works for the boundary conditions we have here. We may therefore consider a sequence $(z_p)$  where each $z_p$ is a minimizer of the $p$-elastic energy subject to the constraints of problem (P), and let $z_*$ be \emph{any} element of $W^{2,\infty}((0,t_f),\mathbb R^2)$ satisfying the constraints. Then for any $p\geq q$, using the minimality of $z_q$ and $z_p$ and the embedding $L^p\subset L^q$ we have
    \begin{equation}\label{zestimates}
{\Vert \ddot z_q\Vert}_{L^q} \leq {\Vert \ddot z_p\Vert}_{L^q} \leq c {\Vert \ddot z_p\Vert}_{L^p}\leq c{\Vert \ddot z_* \Vert}_{L^p}\leq {t_f}^{\tfrac{1}{q}}{\Vert \ddot z_* \Vert}_{L^\infty}
    \end{equation}
    where $c={t_f}^{\tfrac{1}{q}-\tfrac{1}{p}}$. Note also that for any $z\in W^{2,q}((0,t_f),\mathbb R^2)$ satisfying the constraints of (P), since $\Vert \dot z\Vert =1$ we have ${\Vert \dot z\Vert}_{L^q}={t_f}^{\frac{1}{q}}$, and if $q\geq 2$ then, using the fundamental theorem of calculus, Cauchy-Schwarz' inequality and Young's inequality with $\varepsilon$:
    \begin{align*}
        \Vert z(t)\Vert^2=\Vert z(0)\Vert^2+\int_0^{t}\frac{d}{d\tau}\Vert z(\tau)\Vert^2 d\tau=\Vert z(0)\Vert^2+\int_0^{t}\langle z(\tau),\dot z(\tau) \rangle d\tau \leq \Vert z(0)\Vert^2+\frac{1}{2\varepsilon}\Vert z\Vert^2_{L^2}+\frac{t\varepsilon}{2}.
    \end{align*}
    Rearranging and integrating over $(0,t_f)$ gives
    \begin{align*}
        \Vert z\Vert^2_{L^2}-\frac{t_f}{2\varepsilon}\Vert z\Vert^2_{L^2}
        &\leq t_f \Vert z(0)\Vert^2+t_f^2\varepsilon,
    \end{align*}
    and then setting $\varepsilon=t_f$ gives 
    \[
   \Vert z\Vert _{L^q} \leq c  \Vert z\Vert^2_{L^2}\leq 2c (t_f \Vert z(0)\Vert^2+t_f^3).
    \]
    for $q\geq 2$. Combining these lower order estimates with \eqref{zestimates}, we have that for each $2\leq q<\infty$ the subsequence $(z_p)_{p\geq q}$ is bounded in $W^{2,q}((0,t_f),\mathbb R^2)$, and therefore has a subsequence which converges weakly in $W^{2,q}((0,t_f),\mathbb R^2)$ (by the local weak compactness of reflexive Banach spaces, \cite[p.~126]{yosida}) and strongly in $C^1([0,t_f],\mathbb R^2)$ (by compactness of the Sobolev embedding, \cite[Theorem~6.3]{adams}). By induction we can construct successive subsequences $(z_{p_2(i)})(z_{p_3(i)}),\ldots (z_{p_k(i)}),\ldots $, such that $(z_{p_{k+1}(i)})$ is a subsequence of $(z_{p_k(i)})$ which converges weakly in $W^{2,k+1}$. Then the diagonal sequence $(z_{\tilde p(i)})$ where $\tilde p(i)=p_i(i)$ converges weakly in $W^{2,k}$ for every $k\geq 2$. Indeed, since weak convergence in $W^{2,k}$ implies weak convergence in $W^{2,q}$ for all $q\leq k$ by the  embedding $W^{2,q}\subset W^{2,k}$, we have that $z_{\tilde p(i)}$ converges weakly in $W^{2,q}$ (and strongly in $C^1$) for \emph{every} $q\geq 2$ to  $z_\infty \in \bigcap_{q\geq 2} W^{2,q}((0,t_f),\mathbb R^2)$. By the strong $C^1$ convergence, $z_\infty$ satisfies the constraints. Since $\Vert \ddot z_\infty \Vert_{L^q}$  is uniformly bounded for all $q$, we have $z_\infty \in W^{2,\infty}((0,t_f),\mathbb R^2)$ and $\lim_{p\to \infty}{\Vert \ddot z_\infty \Vert}_{L^p}={\Vert \ddot z_\infty \Vert}_{L^\infty}$ by \cite[Theorem 2.14]{adams}, and also $\lim_{p\to \infty}{\Vert \ddot z_p \Vert}_{L^p}={\Vert \ddot z_\infty \Vert}_{L^\infty}$ . Finally, $z_\infty$ is a minimizer of maximum curvature, because if not then there exists $z$ and $\varepsilon>0$ such that ${\Vert \ddot z\Vert}_{L^\infty}={\Vert \ddot z_\infty \Vert}_{L^\infty}-\varepsilon$, and then choosing $p$ sufficiently large:
    \[
{t_f}^{-\frac{1}{p}}{\Vert \ddot z_p \Vert}_{L^p}>{\Vert \ddot z_\infty \Vert}_{L^\infty}-\tfrac{\varepsilon}{2}>{\Vert \ddot z \Vert}_{L^\infty}\geq {t_f}^{-\frac{1}{p}}{\Vert \ddot z \Vert}_{L^p}
    \]
    which contradicts the assumption that $z_p$ is a minimizer of the $p$-elastic energy. 
\end{proof}

\section{A Numerical Method}
\label{num_meth}

In this section, we provide a numerical method utilizing {\em arc}, or {\em switching time}, {\em parametrization techniques} earlier used for optimal control problems exhibiting discontinuous controls in~\cite{KayNoa1996, KayNoa2003, KayLucSim2004, MauBueKimKay2005}.
Theorem~\ref{classification} asserts that a curve of minimax curvature can be of type (described by the strings) $CCC$, or $COC$, or $CSC$, or $SOS$, or a substring thereof, and nothing else. Then the problem of finding such a curve can be reduced to finding a correct concatenation/configuration of circular arcs ($C$), straight lines ($S$) and a loop ($O$) as listed, as well as finding the length of each arc involved.  Recall that a loop is nothing but a circular arc with length $2\,\pi/a$.  Therefore, for computational purposes, we will treat a loop as any other circular arc $C$ with an unknown length.  We also note that if there are more than one circular arc in a string then they must have the same curvature $a$, i.e., the same radius.

A circular arc $C$ can either be a left-turn or a right-turn arc, which we denote by $L$ and $R$, respectively.  Then a curve of type $CCC$ can be represented by the string $LRLR$.  Clearly, at least one of the arcs in $LRLR$ must be of zero length in order to represent $CCC$ or a substring thereof, and this should be determined by the numerical method to be proposed. A curve of type $CSC$ can similarly be represented by $LRSLR$, again with at least one of the arcs of type $L$ or $R$ being of zero length.  Note that if the length of the straight line (S) is zero, then $LRSLR$ effectively reduces to $LRLR$ representing $CCC$.  Therefore, the string $LRSLR$ is general enough to represent both of the types $CCC$ and $CSC$.  The type $SCS$ (or equivalently $SOS$), on the other hand, can be represented by the string $SLRS$.  Recall that, by Remark~\ref{rem:SCS}, any solution of type $SOS$ can be obtained from a solution of type $SO$ or $OS$, simply by placing the loop anywhere on the straight line segment between the endpoints.  Therefore either the string $SLR$ or the string $LRS$ will suffice to represent $SCS$.  As a result, one would cover all the solution types in Theorem~\ref{classification} by using the string $LRSLR$.

Let $L_{\xi_1}$ denote a left-turn arc with length $\xi_1$, $R_{\xi_2}$ a right-turn arc with length $\xi_2$ and $S_{\xi_3}$ a straight line with length $\xi_3$.  Similarly, let $L_{\xi_4}$ denote a left-turn arc with length $\xi_4$ and $R_{\xi_5}$ a right-turn arc with length $\xi_5$.  Then the types of solution curves described in Theorem~\ref{classification} can all be represented by the string
\[
L_{\xi_1}R_{\xi_2}S_{\xi_3}L_{\xi_4}R_{\xi_5}\,.
\]
For example, the type $RLR$ is given with $\xi_1=\xi_3=0$ and $\xi_2,\xi_4,\xi_5 > 0$. On the other hand, the type $LR$ is obtained either with $\xi_3=\xi_4=\xi_5=0$ and $\xi_1,\xi_2 > 0$ or with $\xi_1=\xi_2=\xi_3=0$ and $\xi_4,\xi_5 > 0$.

Let the initial time $t_0 := 0$ and the terminal time $t_5 := t_f$.  Also define the {\em switching times} $t_j$, $j = 1,\ldots 4$, such that
\begin{equation}  \label{arc_durations}
\xi_j := t_j - t_{j-1}\,,\quad\mbox{for } j = 1,\ldots,5\,.
\end{equation}
The lengths $\xi_j$ may also be referred to as {\em arc durations}.  Note that along an arc of type $L$ or $R$ the optimal control is bang--bang; in particular, $v(t) \equiv 1$ along an arc of type $L$ and $v(t) \equiv -1$ along an arc of type $R$.  Along an arc of type $S$, on the other hand, optimal control is singular; in particular, $v(t) \equiv 0$, by Lemma~\ref{singular}.  Since $v(t)$ is constant ($1$, $-1$ or 0) for $t_{j-1} \le t < t_j$ (along the $j$th arc), the ODEs in Problem~(OC) with $\alpha(t) = a$, or equivalently Problem~(Pc) with $u(t) = a\,v(t)$, can be solved as follows.
\begin{eqnarray}
&& \theta(t) = \theta(t_{j-1}) + a\,v(t)\,(t - t_{j-1})\,,\quad\mbox{ if } j = 1,\ldots,5\,, \label{theta_arc} \\[2mm]
&& x(t) = \left\{\begin{array}{ll}
x(t_{j-1}) + \big(\sin\theta(t) - \sin\theta(t_{j-1})\big) / \big(a\,v(t)\big)\,, & \mbox{ if } j = 1,2,4,5\,, \\[1mm]
x(t_{j-1}) + \cos\theta(t)\,(t - t_{j-1})\,, & \mbox{ if } j = 3\,,
\end{array}\right. \label{x_arc} \\[2mm]
&& y(t) = \left\{\begin{array}{ll}
y(t_{j-1}) - \big(\cos\theta(t) - \cos\theta(t_{j-1})\big) / \big(a\,v(t)\big)\,, & \mbox{ if } j = 1,2,4,5\,, \\[1mm]
y(t_{j-1}) + \sin\theta(t)\,(t - t_{j-1})\,, & \mbox{ if } j = 3\,, \label{y_arc}
\end{array}\right.
\end{eqnarray}
where
\begin{equation}  \label{arc_control}
v(t) = \left\{\begin{array}{rl}
1\,, & \mbox{ if } j = 1,4\,, \\[1mm]
-1\,, & \mbox{ if } j = 2,5\,, \\[1mm]
0\,, & \mbox{ if } j = 3\,,
\end{array}\right.
\end{equation}
for $t_{j-1} \le t < t_j$.  We have used the solution for $\theta(t)$ in \eqref{theta_arc} in obtaining the solutions for $x(t)$ and $y(t)$ in \eqref{x_arc}--\eqref{y_arc}.  Note that the control variable $v(t)$ is a piecewise constant function, which takes here the sequence of values $\{1, -1, 0, 1, -1\}$.  It should be clear from the context that the control variable $u(t)$ defined in Problem~(Pc) takes the sequence of values $\{a, -a, 0, a, -a\}$.  After evaluating the state variables in \eqref{theta_arc}--\eqref{y_arc} at the switching times and carrying out algebraic manipulations, one can equivalently rewrite Problem~(Pc), or Problem~(OC), as follows.
\[
\mbox{(Ps)}\left\{\begin{array}{rl}
\min &\ a
   \\[4mm]
\mbox{s.t.} &\ds\ x_0 - x_f + \frac{1}{a}\big(-\!\sin\theta_0 + 2\,\sin\theta_1 - 2\,\sin\theta_2 + 2\,\sin\theta_4 - \sin\theta_f \big) + \xi_3\,\cos\theta_2 = 0\,, \\[3mm] 
  &\ds\ y_0 - y_f + \frac{1}{a}\big(\cos\theta_0 - 2\,\cos\theta_1 + 2\,\cos\theta_2 - 2\,\cos\theta_4 + \cos\theta_f \big) + \xi_3\,\sin\theta_2 = 0\,, \\[3mm]
  &\ds\ \sin\theta_f = \sin\theta_5\,,\ \ \cos\theta_f = \cos\theta_5\,,  \\[2mm]
  & \ds t_f = \sum_{j=1}^5 \,\xi_j\,,\qquad \xi_j \ge 0\,,\ \mbox{ for } j = 1,\ldots,5\,,
\end{array}\right.
\]
where
\begin{equation}  \label{thetas}
\theta_1 = \theta_0 + a\,\xi_1\,,\qquad
\theta_2 = \theta_1 - a\,\xi_2\,,\qquad
\theta_4 = \theta_2 + a\,\xi_4\,,\qquad
\theta_5 = \theta_4 - a\,\xi_5\,.
\end{equation}
Substitution of $\theta_1$, $\theta_2$, $\theta_4$ and $\theta_5$ in \eqref{thetas} into Problem~(Ps) yields a finite-dimensional nonlinear optimization problem in just six variables, $\xi_j$,  $j = 1,\ldots,5$, and $a$.

\begin{remark}  \label{rem:angle}  \rm
With the constraints $\sin\theta_f = \sin\theta_5$ and $\cos\theta_f = \cos\theta_5$, we make sure that we satisfy the slope condition at the terminal point.  Otherwise, the constraint $\theta_f = \theta_5$ is stronger, and imposing it might result in missing some of the feasible solutions.
\endproof
\end{remark}

\begin{remark}[Optimization Software for (Ps)]  \label{rem:global}  \rm
Problem~(Ps) can be solved by standard optimization methods and software, for example, Algencan~\cite{Andreani2007,BirMar2014}, which implements augmented Lagrangian techniques, or Ipopt~\cite{WacBie2006}, which implements an interior point method, or SNOPT~\cite{GilMurSau2005}, which implements a sequential quadratic programming algorithm, or Knitro~\cite{Knitro}, which implements various interior point and active set algorithms to choose from.  For general nonconvex optimization problems like Problem~(Ps), what one can hope for, by using these software, is to get (at best) a locally optimal solution.  
\endproof
\end{remark}

\begin{remark}[Search for a Global Minimizer]  \label{rem:combinatorial}  \rm
We recall that the type of a solution curve can be captured by the five-letter-long string $L_{\xi_1}R_{\xi_2}S_{\xi_3}L_{\xi_4}R_{\xi_5}$, with at most three of the arc lengths being nonzero.  For instance, the type $RLR$ is defined by $\xi_1=\xi_3=0$ and $\xi_2,\xi_4,\xi_5 > 0$, and $LSR$ by $\xi_2=\xi_4=0$ and $\xi_1,\xi_3,\xi_5 > 0$.  This prompts one to restrict all possible solution types to strings with at most three letters and carry out a combinatorial search for an optimal solution.  This requires a more extensive coding for solving Problem~(Ps), and therefore we leave it outside the scope of the current paper.

Even a combinatorial approach as alluded above does not offer a certificate for a global minimizer, since a given (globally minimizing) type may result in more than just one critical solution, as will be demonstrated by Example~2 in the next section:  Figure~\ref{fig:COC_ex}(d) illustrates that type $RLR$ results in two critical solutions ($a \approx 0.503$ and $a \approx 0.798$), one of which ($a \approx 0.503$) is identified as the global minimizer, only after numerical experimentation.
\endproof
\end{remark}

\subsection{Numerical experiments}
\label{sec:numexp}

For computations numerically solving Problem~(Ps), or equivalently Problem~(P), we have employed the AMPL--Knitro computational suite: AMPL is an optimization modelling language~\cite{AMPL} and Knitro is a popular optimization software~\cite{Knitro} (Version 3.0.1 is used here).  In all runs, we have used the Knitro options {\tt `alg=0 feastol=1e-12 opttol=1e-12'} in AMPL; in other words, we have set both the feasibility and optimality tolerances at $10^{-12}$, and the particular algorithm to be employed was left to be chosen by Knitro itself.

In what follows, we present example instances of problems for which we solve Problem~(Ps) by using the AMPL--Knitro suite and display the solution curves.  The best numerical solution we have identified for an instance is depicted by a solid curve in the graphs.  The other solutions reported by the suite as ``locally optimal'' are depicted by dotted curves,  and referred to here as ``critical'', as they satisfy the necessary conditions of optimality furnished by the maximum principle.

\subsubsection{Example 1}
\label{ex:SCS}

Recall from Proposition~\ref{SCS} that $b$ is the solution of the equation $\sinc^{-1}(b)(1-b)=\pi/2$.  After rearranging this equation one gets $f(b) := b - \sinc\big(\pi/(2\,(1 - b))\big) = 0$.  A numerical solution to $f(b) = 0$ can be obtained as $b \approx 0.319966693534110$  (correct to 15 dp) in seven iterations by using the secant method starting with initial points $0.1$ and $0.2$.  The uniqueness of this solution can be established by observing that the graph of $f$ crosses the $b$-axis only once (not shown here). 

Figure~\ref{fig:SCS_ex} shows curves of minimax curvature between the oriented points $(x_0,y_0,\theta_0) = (0,0,0)$ and $(x_f,y_f,\theta_f) = (0,1,0)$, which is a similar situation to that in the diagram in Figure~\ref{fig:SCS}.  The orientation of the endpoints of the curves in Figure~\ref{fig:SCS_ex} are depicted by (black) arrows.  We note that, for the two given endpoints, $d = 1$.

\begin{figure}[t!]
\begin{minipage}{80mm}
\begin{center}
\psfrag{x}{\small $x$}  \psfrag{y}{\small $y$}
\includegraphics[width=90mm]{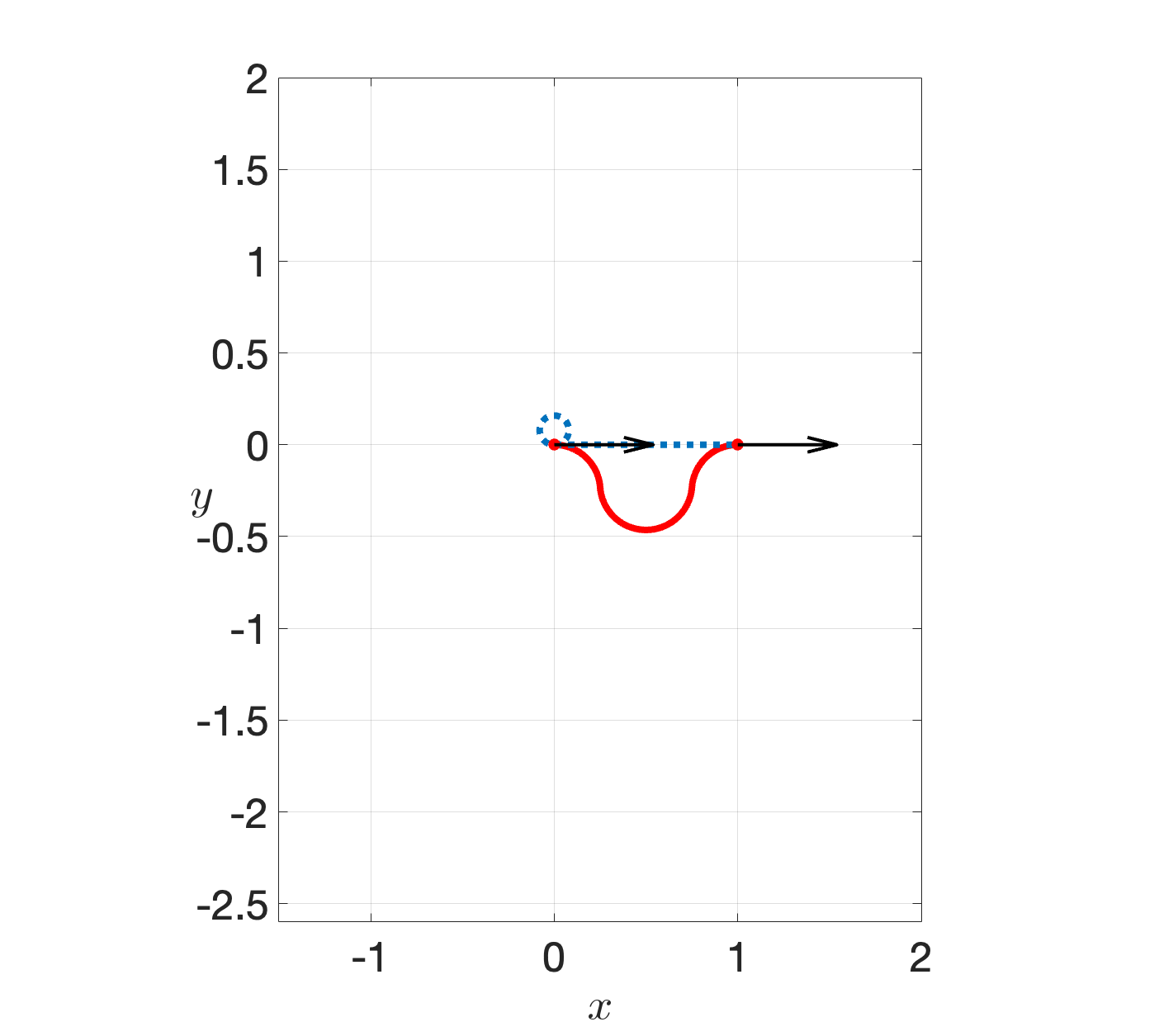} \\[0mm]
{\sf\small\hspace{7mm} (a) $t_f = 1.5$: Solid curve RLR (red) with $a \approx 3.989$; dotted curve LS (blue) with $a = 4\,\pi$.}
\end{center}
\end{minipage}
\begin{minipage}{80mm}
\begin{center}
\psfrag{x}{\small $x$}  \psfrag{y}{\small $y$}
\includegraphics[width=90mm]{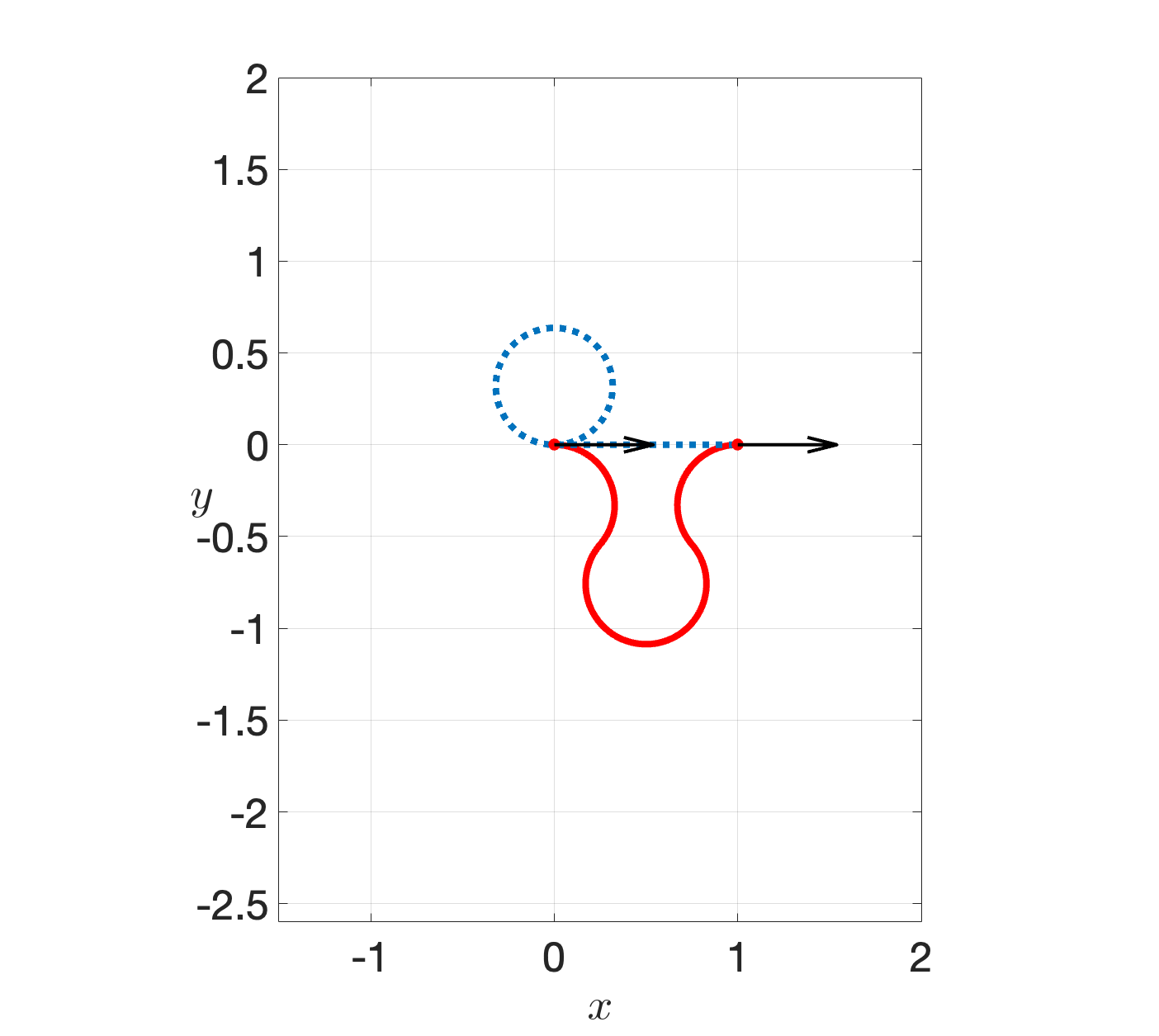} \\[0mm]
{\sf\small\hspace{7mm} (b) $t_f = 3$: Solid curve RLR (red) with $a \approx 3.038$; dotted curve LS (blue) with $a = \pi$.}
\end{center}
\end{minipage}
\\[0mm]
\begin{minipage}{80mm}
\begin{center}
\psfrag{x}{\small $x$}  \psfrag{y}{\small $y$}
\includegraphics[width=90mm]{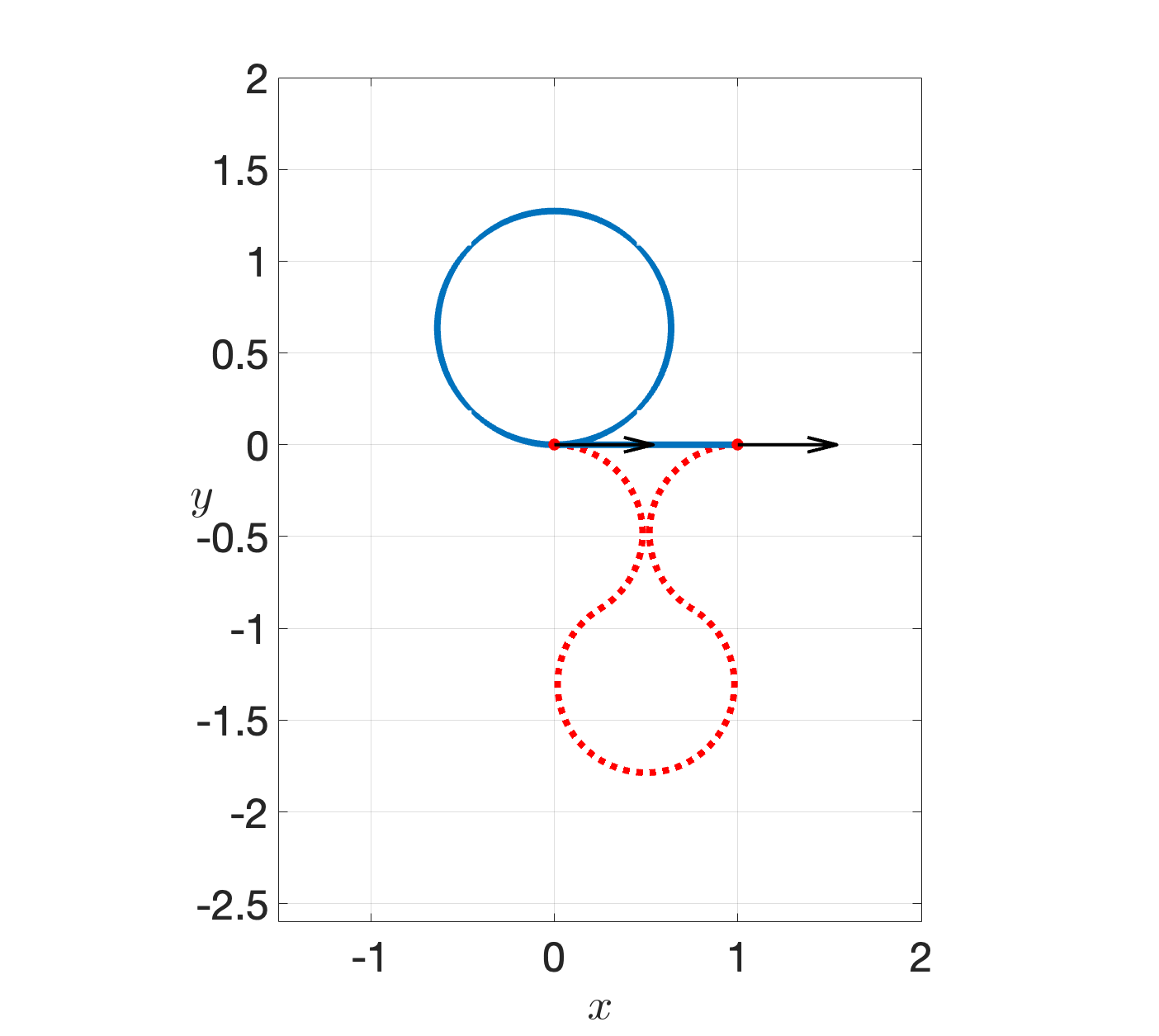} \\[0mm]
{\sf\small\hspace{7mm} (c) $t_f = 5$: Solid curve LS (blue) with $a = \pi/2$; dotted curve RLR (red) with $a \approx 2.077$.}
\end{center}
\end{minipage}
\begin{minipage}{80mm}
\begin{center}
\psfrag{x}{\small $x$}  \psfrag{y}{\small $y$}
\includegraphics[width=90mm]{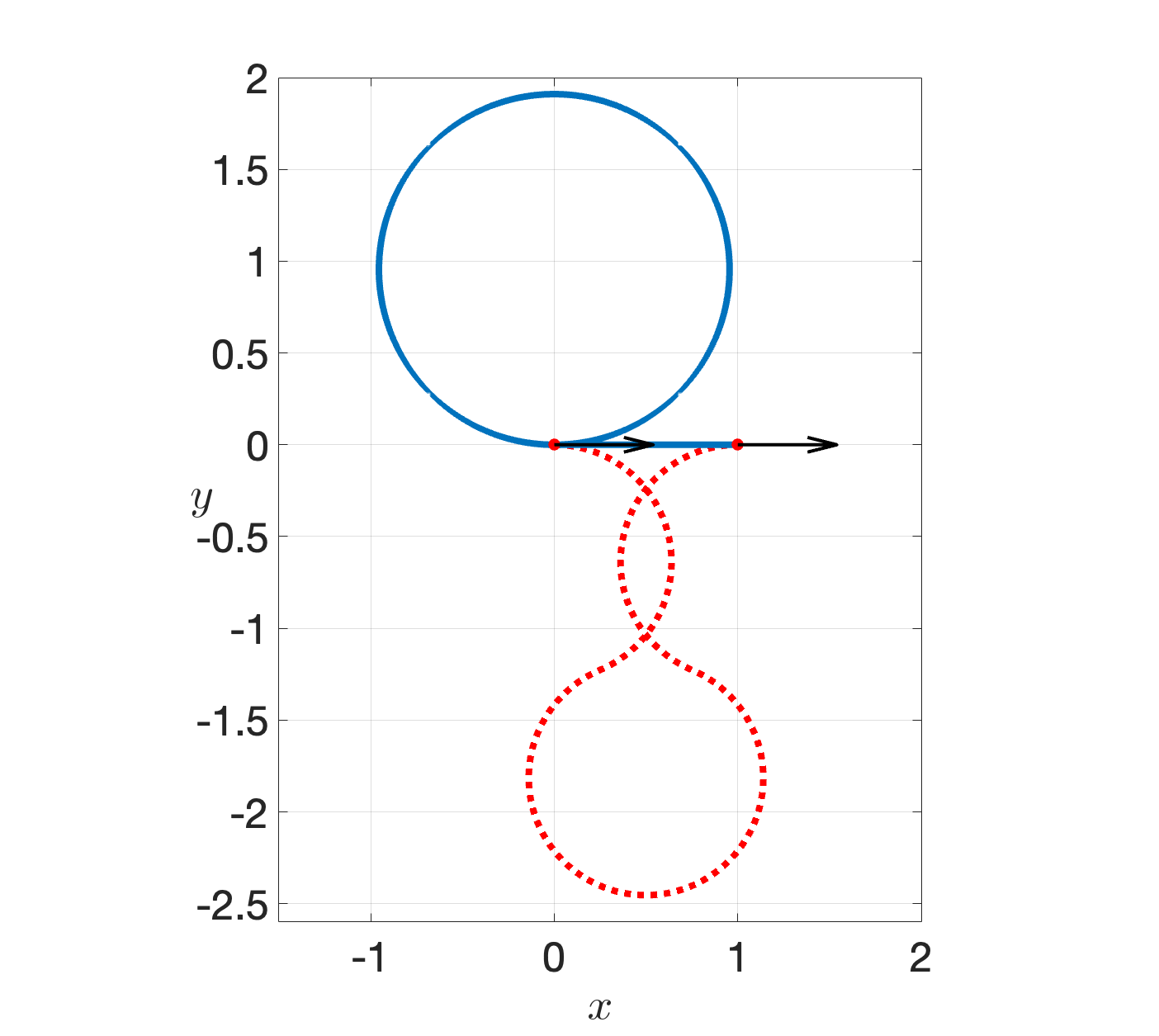} \\[0mm]
{\sf\small\hspace{7mm} (d) $t_f = 7$: Solid curve LS (blue) with $a = \pi/3$; dotted curve RLR (red) with $a \approx 1.565$.} 
\end{center}
\end{minipage}
\
\caption{\sf Example~1---Problem instances for Proposition~\ref{SCS} and its converse: (a)--(b) for $d / t_f > b$ ($t_f \le 3.12532$) and (c)--(d) for $d / t_f \le b$ ($t_f \ge 3.12533$)---Curves of minimax curvature between the oriented points $(0,0,0)$ and $(0,1,0)$, for various $t_f$.  The solid curves are optimal (identified by numerical experiments) while the dotted ones are only critical.}
\label{fig:SCS_ex}
\end{figure}

Proposition~\ref{SCS} asserts that the $SOS$ type of trajectory is not optimal (nor $OS$ or $SO$, by Remark~\ref{rem:SCS}) for $t_f < 1/b$, i.e., $t_f \le 3.12532$ (correct to 5 dp).  This is exemplified in Figures~\ref{fig:SCS_ex}(a) and \ref{fig:SCS_ex}(b) respectively for $t_f = 1.5$ and $t_f = 3$, where it is numerically verified that the CCC concatenation of curves is optimal:  For $t_f = 1.5$, the solution curve has $a = 3.9887508486$ (correct to 10 dp) and is of type $R_{\xi_2}L_{\xi_4}R_{\xi_5}$ (with $\xi_1 = \xi_3 = 0$), where the arc lengths $\xi_2 = 0.375$, $\xi_4 = 0.75$, and $\xi_5 = 0.375$.  For $t_f = 3$, $a = 3.0384835468$ (correct to 10 dp) and the solution is of the same type, with $\xi_2 = 0.75$, $\xi_4 = 1.5$, and $\xi_5 = 0.75$.  We point to the symmetry in the arc lengths in both cases, presumably because of the colinear configuration of the endpoints.

The converse of Proposition~\ref{SCS} for this example instance (which is not proved) is that the $SOS$ (or $OS$ or $SO$, by Remark~\ref{rem:SCS}) concatenation is optimal for $t_f \ge 1/b$, i.e., $t_f \ge 3.12533$ (correct to 5 dp).  This converse is exemplified in Figures~\ref{fig:SCS_ex}(c) and \ref{fig:SCS_ex}(d) respectively for $t_f = 5$ and $t_f = 7$.  For $t_f = 5$, the solution curve has $a = \pi/2$ and is of type $L_{\xi_1}S_{\xi_3}$ (with $\xi_2 = \xi_4 = \xi_5 = 0$), where $\xi_1 = 4$, and $\xi_3 = 1$.  For $t_f = 5$, $a = \pi/3$ and the solution is of the same type, with $\xi_1 = 6$, and $\xi_3 = 1$.

Each of the solution curves presented in Figure~\ref{fig:SCS_ex} can be ``flipped over'' or ``reflected about'' the $x$-axis to get ``symmetric'' solution curves: For example when an $LS$ curve is reflected one gets an $RS$ curve which has the same length, oriented endpoints and maximum curvature. An $RLR$ curve can also be reflected in the same way to get an $LRL$ curve which has the same length, oriented endpoints and maximum curvature.  These reflected curves have also been obtained as computational solutions to Problem~(Ps).

\subsubsection{Example 2}
\label{ex:COC}

Suppose that the endpoints are placed on a unit circle with the end velocities chosen tangent to the circle. Recall from Proposition~\ref{COC} that a curve of minimax curvature between these endpoints can be of type $COC$ (or equivalently type $OC$ or type $CO$, as the loop $O$ can be shifted anywhere on a circular arc $C$) only in the case when the endpoints are separated by at most $\pi/2$ radians, i.e., they are separated by an angle $\beta \le \pi/2$ (see Figure~\ref{fig:COC}), and $t_f$ is chosen as $\beta + 2\,\pi$.  Here we numerically illustrate Proposition~\ref{COC} and explore instances with various other values of the separation angle $\beta$ and with $t_f = \beta + 2\,\pi$---see Figure~\ref{fig:COC_ex}(a)--(d).

We place the pair of oriented endpoints on a unit circle, and impose $t_f = \beta + 2\,\pi$ (so that a loop can be observed at least in a critical curve), in each part of Figure~\ref{fig:COC_ex}, and obtain optimal solutions as follows.
\begin{itemize}
\item[(a)] $\beta = \pi/3 < \pi/2$, i.e., $t_f = 7\,\pi/3$: $(-1/2, -\sqrt{3}/2, -\pi/6)$ and $(1/2, -\sqrt{3}/2, \pi/6)$.  The optimal curve is of type $L_{\xi_4}R_{\xi_5}$ (with $\xi_1 = \xi_2 = \xi_3 = 0$), where $a = 1$, $\xi_4 = \pi/3$, and $\xi_5 = 2\,\pi$.
\item[(b)] $\beta = 2\,\pi/3$, i.e., $t_f = 8\,\pi/3$: $(-\sqrt{3}/2, -1/2, -\pi/3)$ and $(\sqrt{3}/2, -1/2, \pi/3)$.  The optimal curve is of type $R_{\xi_2}L_{\xi_4}R_{\xi_5}$ (with $\xi_1 = \xi_3 = 0$), where $a = 0.8174619979$, $\xi_2 = 1.4538775285$, $\xi_4 = 5.4698253525$, and $\xi_5 = 1.4538775285$, all correct to 10~dp.
\item[(c)] $\beta = 4\,\pi/3$, i.e., $t_f = 10\,\pi/3$: $(-\sqrt{3}/2, 1/2, -2\,\pi/3)$ and $(\sqrt{3}/2, 1/2, 2\,\pi/3)$.  The optimal curve is of type $R_{\xi_2}L_{\xi_4}R_{\xi_5}$, where $a = 0.5245840019$, $\xi_2 = 0.6217500975$, $\xi_4 = 9.2284753170$, and $\xi_5 = 0.6217500975$, all correct to 10~dp.
\item[(d)] $\beta = 5\,\pi/3$, i.e., $t_f = 11\,\pi/3$: $(-1/2, \sqrt{3}/2, -5\,\pi/6)$ and $(1/2, \sqrt{3}/2, 5\,\pi/6)$.  The optimal curve is of type $R_{\xi_2}L_{\xi_4}R_{\xi_5}$, where $a = 0.5026360614$, $\xi_2 = 0.2755293870$, $\xi_4 = 10.9681142892$, and $\xi_5 = 0.2755293870$, all correct to 10~dp.
\end{itemize}

\begin{figure}[t!]
\begin{minipage}{80mm}
\begin{center}
\includegraphics[width=90mm]{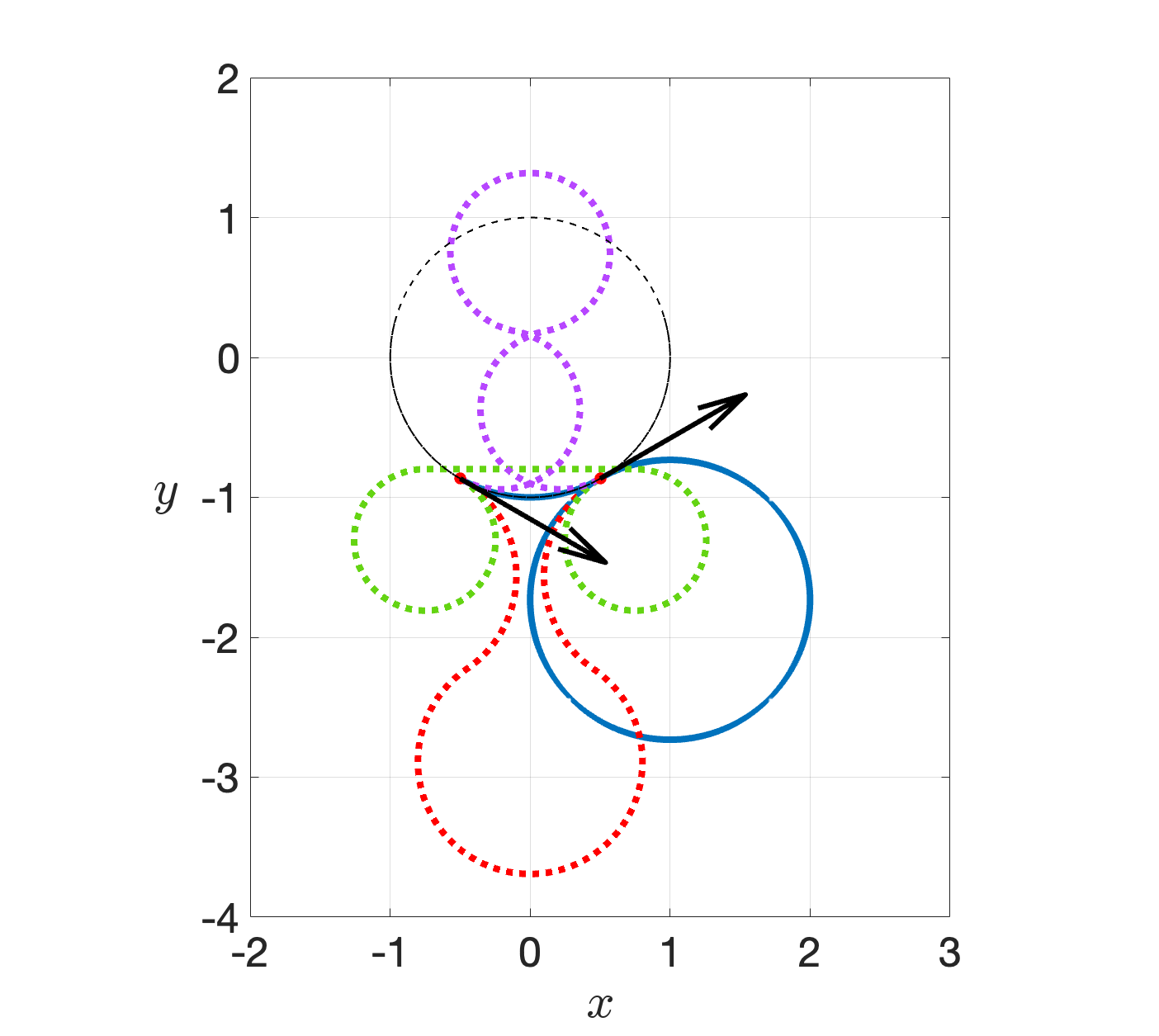} \\[0mm]
{\sf\small\hspace{7mm} (a) $\beta = \pi/3$: Solid curve LR (blue) with $a = 1$; dotted curves, RLR (red) with $a \approx 1.246$, LRL (purple) with $a \approx 1.754$, RSR (green) with $a \approx 1.978$.}
\end{center}
\end{minipage}
\begin{minipage}{80mm}
\begin{center}
\includegraphics[width=90mm]{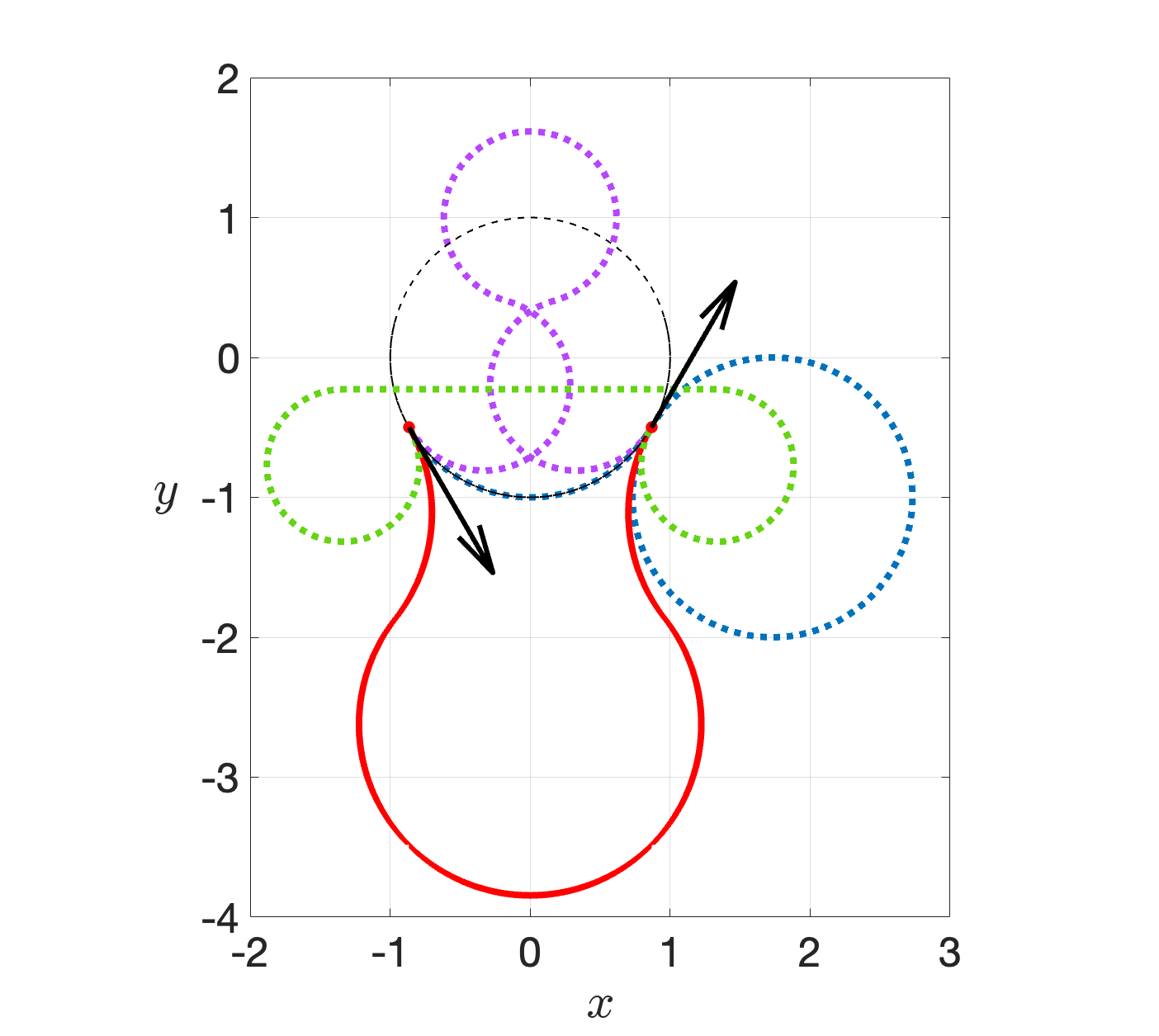} \\[0mm]
{\sf\small\hspace{7mm} (b) $\beta = 2\,\pi/3$: Solid curve RLR (red) with $a \approx 0.817$; dotted curves, LR (blue) with $a = 1$, LRL (purple) with $a \approx 1.620$, RSR (green) with $a \approx 1.836$.}
\end{center}
\end{minipage}
\begin{minipage}{80mm}
\begin{center}
\includegraphics[width=90mm]{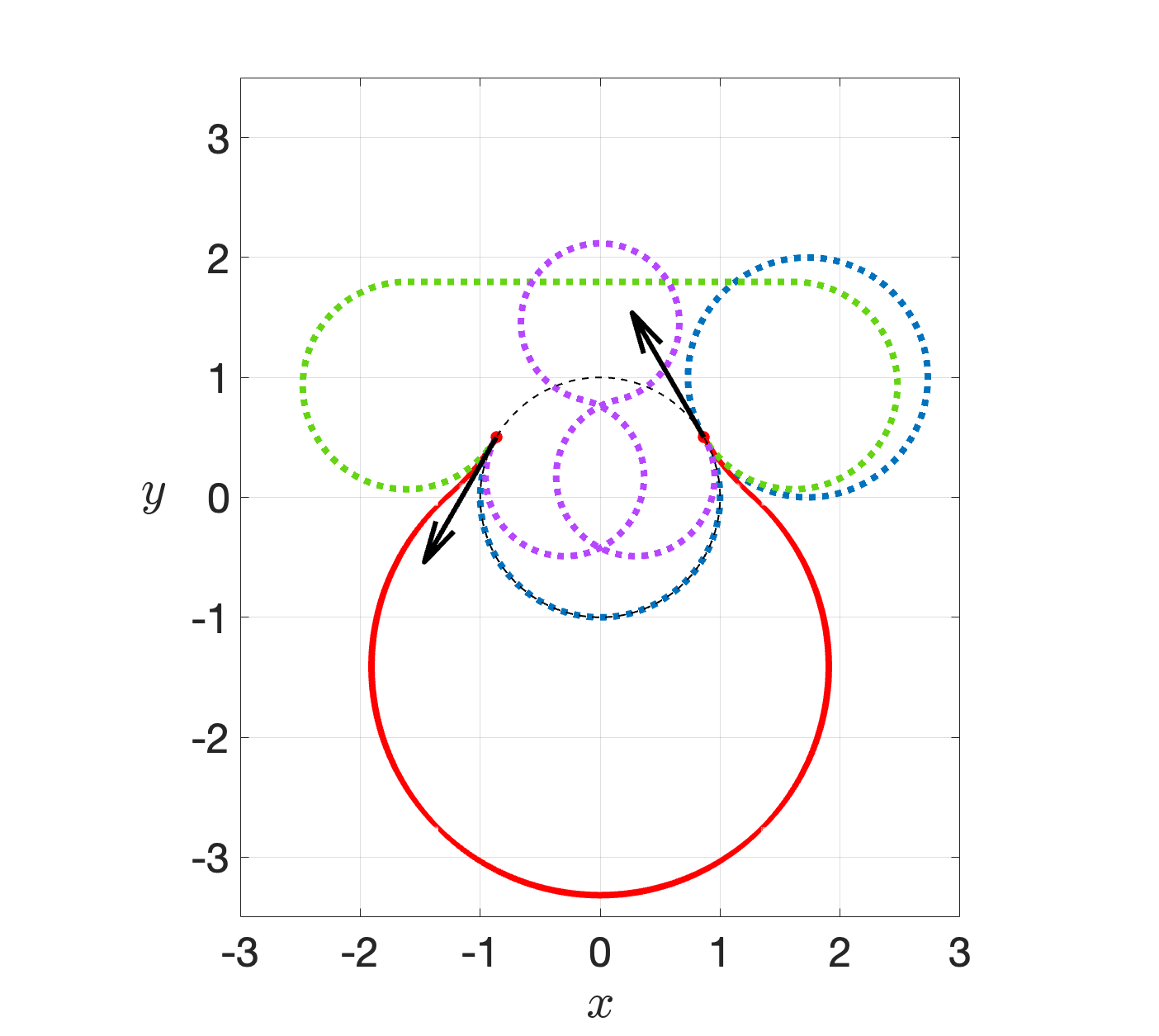} \\[0mm]
{\sf\small\hspace{7mm} (c) $\beta = 4\,\pi/3$: Solid curve RLR (red) with $a \approx 0.525$; dotted curves, LR (blue) with $a = 1$, RSR (green) with $a \approx 1.157$, LRL (purple) with $a \approx 1.514$.}
\end{center}
\end{minipage}
\begin{minipage}{80mm}
\begin{center}
\includegraphics[width=90mm]{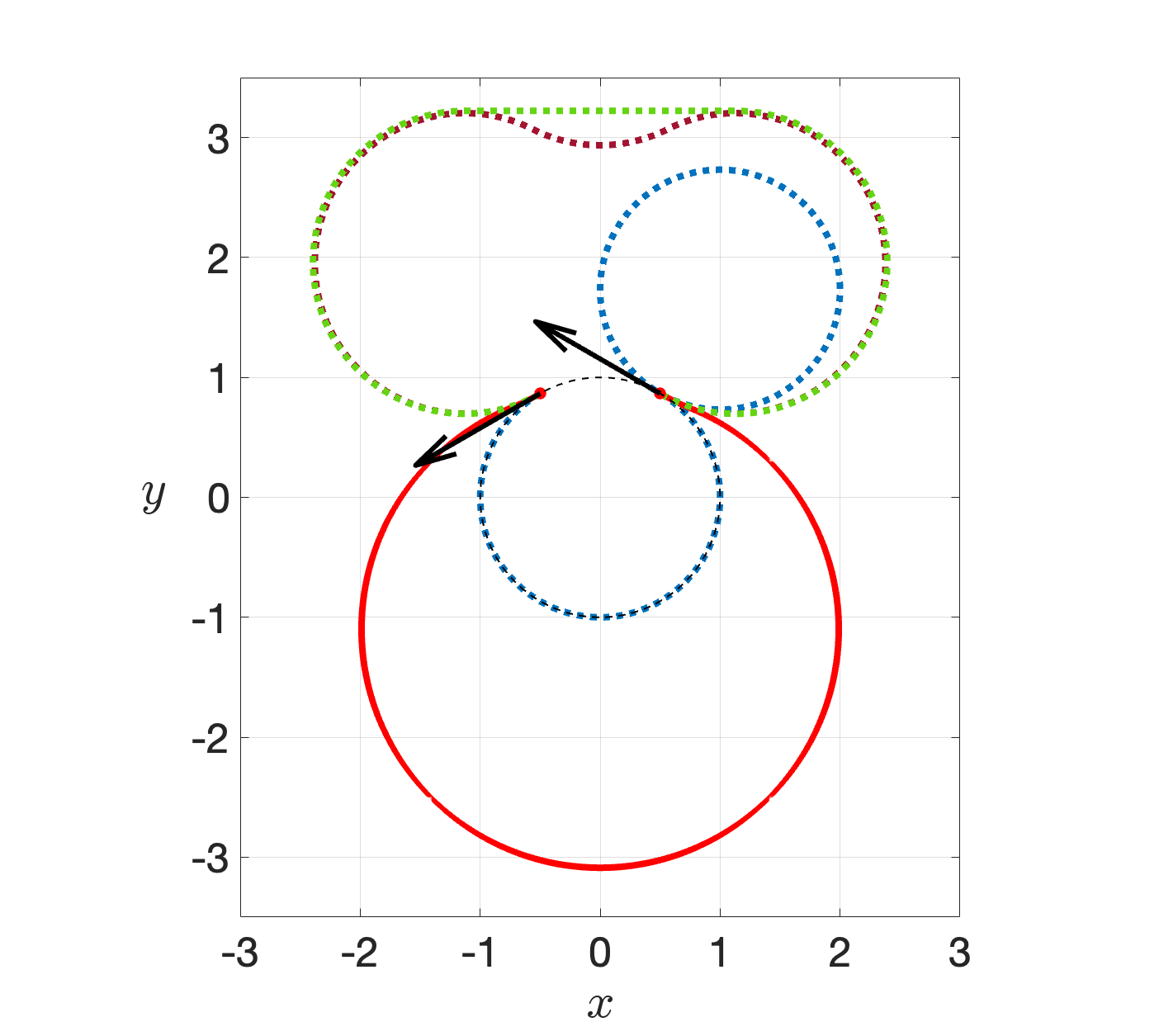} \\[0mm]
{\sf\small\hspace{7mm} (d) $\beta = 5\,\pi/3$: Solid curve RLR (red) with $a \approx 0.503$; dotted curves, RSR (green) with $a \approx 0.792$, RLR (brown) with $a \approx 0.798$, LR (blue) with $a = 1$.}
\end{center}
\end{minipage}
\
\caption{\sf Example~2---Problem instances for Proposition~\ref{COC}: Curves of minimax curvature (and of particular length) between oriented points placed on a unit circle.  The solid curves are optimal (identified by numerical experiments) while the dotted ones are only critical.}
\label{fig:COC_ex}
\end{figure}

In each of the problem instances, one curve of $CO$ type, two curves of $CCC$ type, and one curve of $CSC$ type have been found as solutions of Problem~(Ps) by the AMPL--Knitro suite.  Of these four distinct solutions in each case, the curve with the smallest $a$ was identified as the optimal curve for that case.

Similarly to Example~\ref{ex:SCS}, the loop in each of the $LR$ curves presented in Figure~\ref{fig:COC_ex} can be ``flipped over'' or ``reflected about'' the circle to get ``symmetric'' curves of type $L$ with $\xi_1 = \beta + 2\,\pi$, with the same endpoints and maximum curvatures.  These reflected curves have also been encountered as computational solutions to Problem~(Ps), but are not being shown in Figure~\ref{fig:COC_ex}.

We note that if $t_f \neq \beta + \pi/2$ then one does not encounter a solution curve (optimal or only critical) which contains a loop.

\subsubsection{Example 3}
Consider the problem of finding a curve of minimax curvature between the oriented points $(0,0,-\pi/3)$ and $(0.4,0.4,-\pi/6)$---note that the same oriented endpoints were used in~\cite{Kaya2017} for finding the shortest curve of bounded curvature, i.e., the Markov--Dubins path.  For various lengths, namely for $t_f = 0.8, 1.3, 2$ and $2.5$, here we find curves minimizing the bound $a$ on the curvature.  Figure~\ref{fig:maxcurv}(a)--(d) depict these curves found as a solution of Problem~(Ps) by employing the AMPL--Knitro software suite.

In Figure~\ref{fig:maxcurv}(a)--(d), the optimal curves are shown as solid curves.  These curves have been identified amongst all the critical solutions to Problem~(Ps) (of the allowed types) that the AMPL--Knitro suite could find, as the ones with the smallest curvature, or indeed with the biggest turning radius.  All other critical curves are shown as dotted curves in Figure~\ref{fig:maxcurv}(a)--(d).  A summary of all these solutions in each part of Figure~\ref{fig:maxcurv} can be given as follows, with arc lengths and curvatures correct up to 10 dp.

\begin{figure}[t!]
\begin{minipage}{80mm}
\begin{center}
\includegraphics[width=90mm]{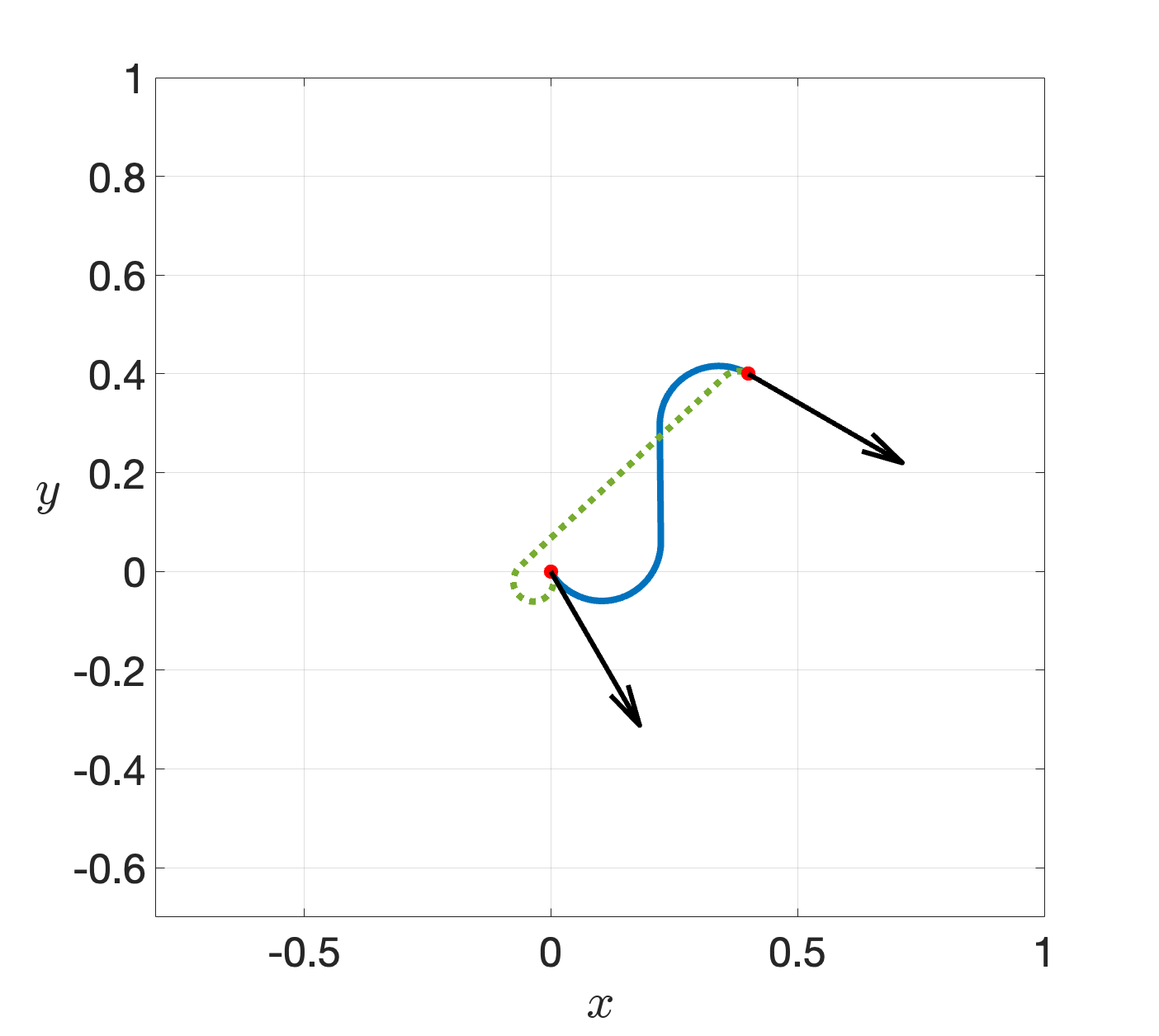} \\[0mm]
{\sf\small\hspace{7mm} (a) $t_f = 0.8$: Solid curve LSR with $a \approx 8.366$.}
\end{center}
\end{minipage}
\begin{minipage}{80mm}
\begin{center}
\includegraphics[width=90mm]{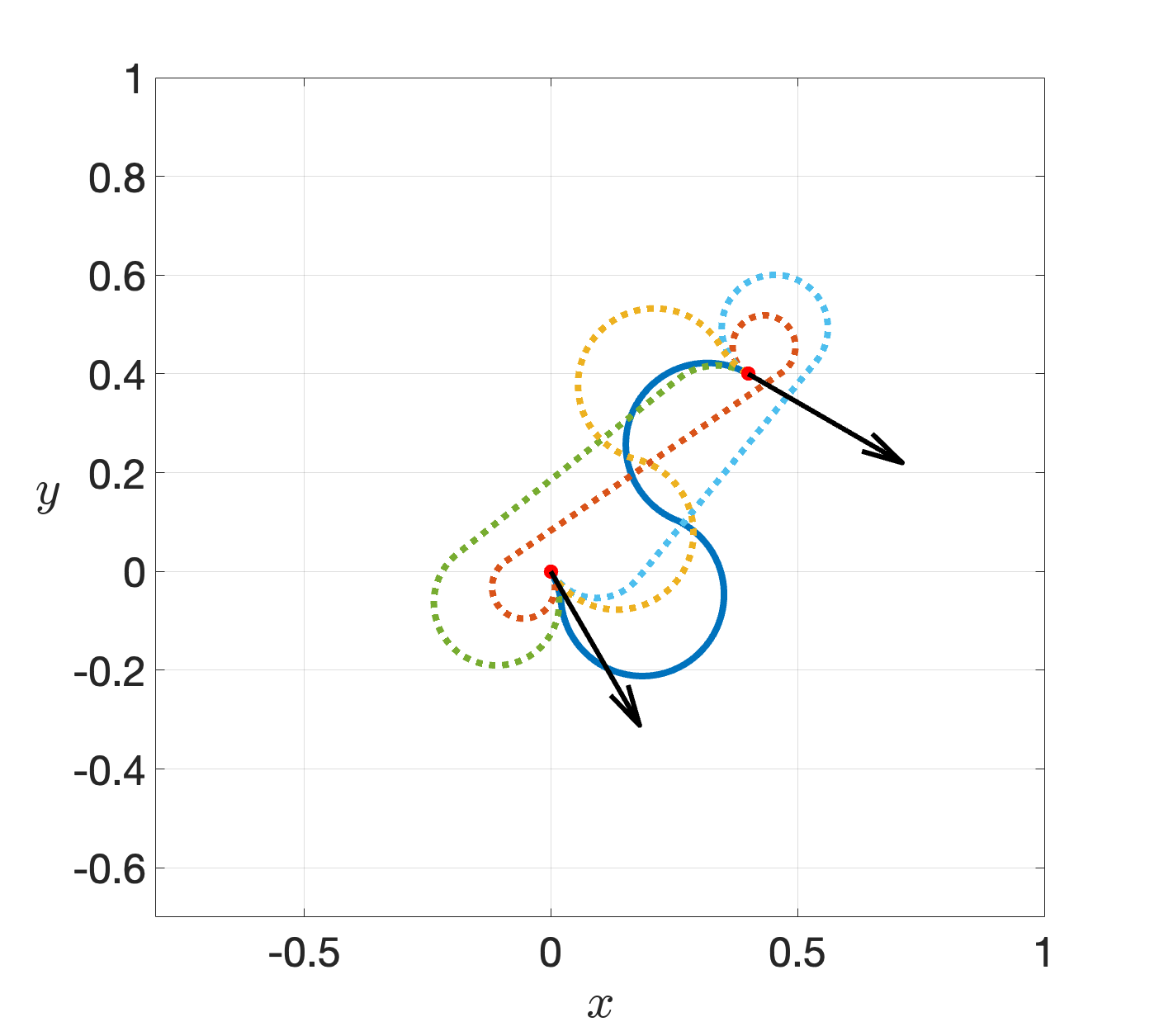} \\[0mm]
{\sf\small\hspace{7mm} (b) $t_f = 1.3$: Solid curve RLR with $a \approx 6.048$.}
\end{center}
\end{minipage}
\\[0mm]
\begin{minipage}{80mm}
\begin{center}
\includegraphics[width=90mm]{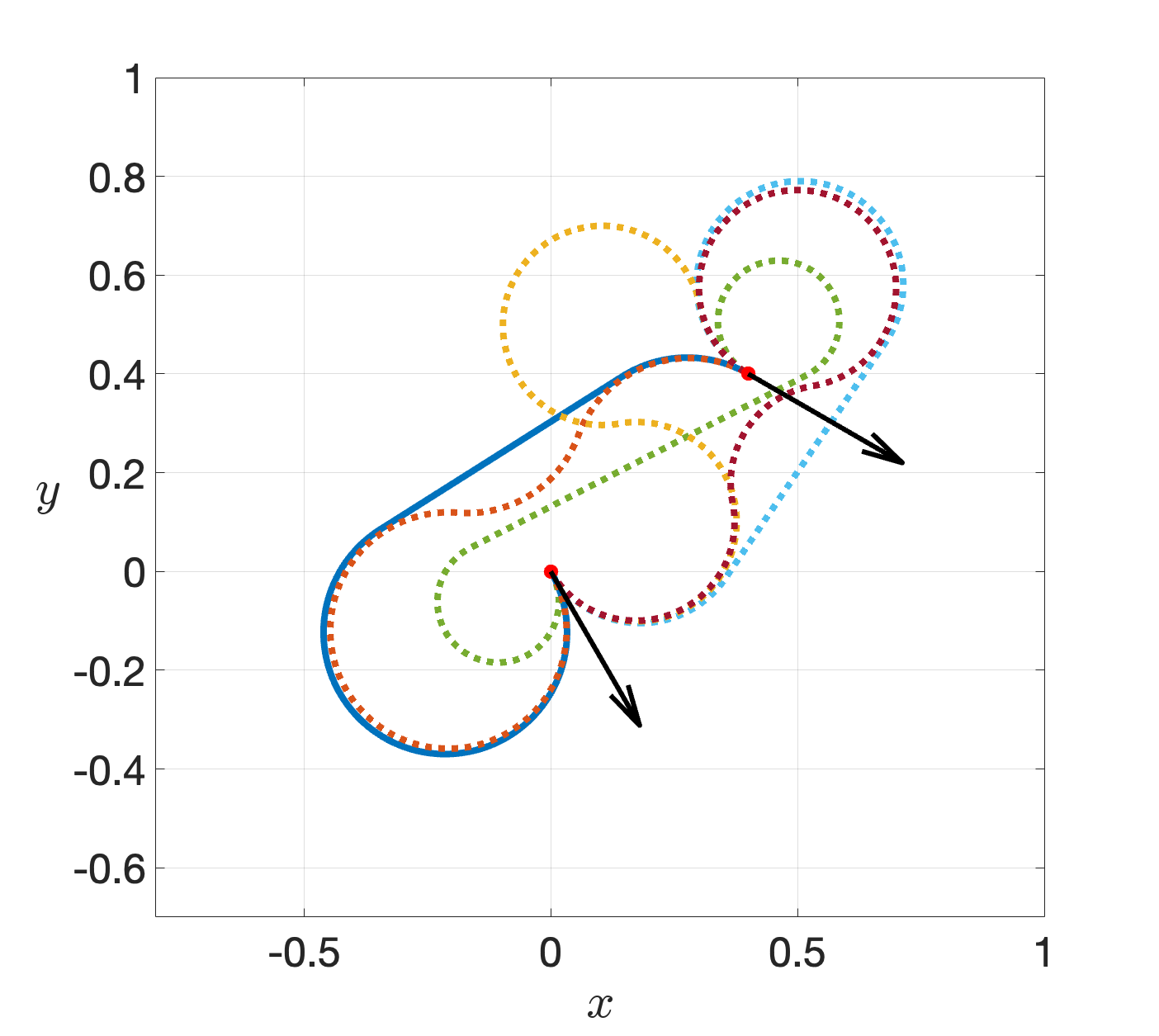} \\[0mm]
{\sf\small\hspace{7mm} (c) $t_f = 2$: Solid curve RSR with $a \approx 4.055$.}
\end{center}
\end{minipage}
\begin{minipage}{80mm}
\begin{center}
\includegraphics[width=90mm]{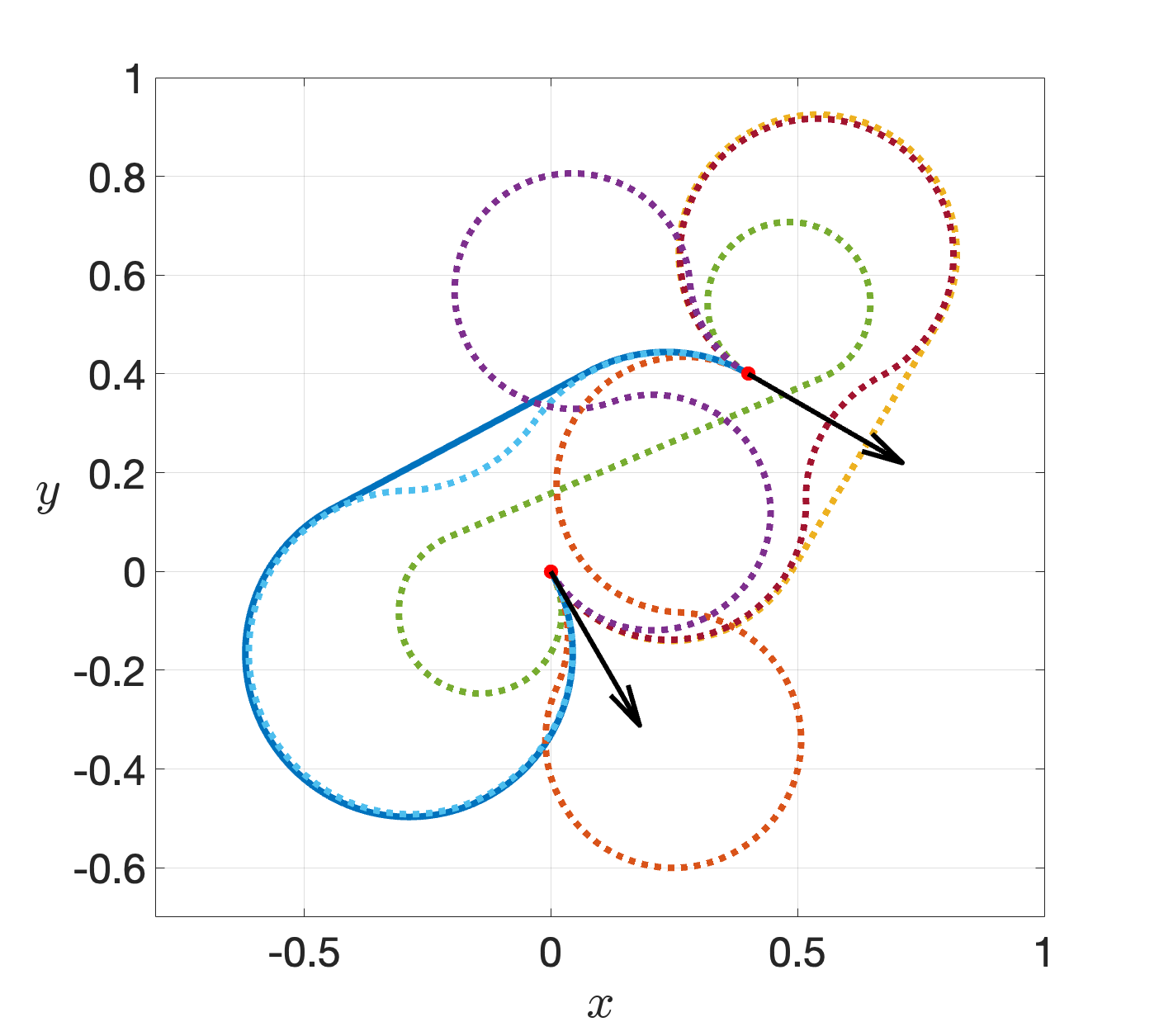} \\[0mm]
{\sf\small\hspace{7mm} (d) $t_f = 2.5$: Solid curve RSR with $a \approx 3.017$.} 
\end{center}
\end{minipage}
\
\caption{\sf Example 3---Curves of minimax curvature between the oriented points $(0,0,-\pi/3)$ and $(0.4,0.4,-\pi/6)$.  The solid (blue) curves are optimal (identified by numerical experiments) while the dotted ones (in various other colours) are only critical.}
\label{fig:maxcurv}
\end{figure}

\begin{itemize}
\item[(a)]  $t_f = 0.8$: The optimal curve has $a = 8.3661513485$ and is of type $L_{\xi_1}S_{\xi_3}R_{\xi_5}$ (with $\xi_2 = \xi_4 = 0$), where $\xi_1 = 0.3141136578$, $\xi_3 = 0.2343580660$, $\xi_5 = 0.2515282761$.  Only one other critical curve is found that has $a = 24.6216106492$ and is of type $RSR$.

\item[(b)]  $t_f = 1.3$: The optimal curve has $a = 6.0477371511$ and is of type $R_{\xi_2}L_{\xi_4}R_{\xi_5}$ with\linebreak $\xi_2 = 0.0684530840$, $\xi_4 = 0.6932888171$, $\xi_5 = 0.5382580988$. Four other critical curves are found which respectively have $a = 7.8842574114$, $a = 15.7179879207$, $a = 9.3041564227$, $a = 6.4570352671$, and are of types $RSR$, $RSL$, $LSL$, $LRL$.

\item[(c)]  $t_f = 2$: The optimal curve has $a = 4.0557744873$ and is of type $R_{\xi_2}S_{\xi_3}R_{\xi_5}$ with\linebreak $\xi_2 = 1.1520596173$, $\xi_3 = 0.5799046398$, $\xi_5 = 0.2680357429$. Five other critical curves are found which respectively have $a = 8.1329787696$, $a = 4.1795061273$, $a = 4.7799043550$, $a = 4.9575871447$, $a = 5.0128303633$ and are of types $RSL$, $RLR$, $LSL$, $LRL$, $LRL$.

\item[(d)]  $t_f = 2.5$: The optimal curve has $a = 3.0172721767$ and is of type $R_{\xi_2}S_{\xi_3}R_{\xi_5}$ with\linebreak $\xi_2 = 1.5726285431$, $\xi_3 = 0.5911279480$, $\xi_5 = 0.3362435090$.  Six other critical curves are found which respectively have $a = 6.0662510507$, $a = 3.0516505811$, $a = 3.8632611845$, $a = 3.5528730863$, $a = 4.1925834507$, $a = 3.6102865904$ and are of types $RSL$, $RLR$, $RLR$, $LSL$, $LRL$, $LRL$.
\end{itemize}

\section{Conclusion}
\label{sec:conclusion}
We have presented a classification of the solutions to the problem of finding planar curves of minimax curvature between two endpoints with the specified endpoint tangents (Problem~(P)).  We devised a numerical procedure to find the critical points of Problem~(P), or the $\infty$-elastica, of the classification types we have presented.  Of these critical solutions, the one with the smallest curvature is the (global) solution to Problem~(P), which we refer to as a curve of minimax curvature.

A natural extension of this paper would be the study of interpolating curves which not only satisfy the endpoint constraints, but must also pass through a number of specified intermediate points.

It would also be interesting to study another extension of Problem~(P), from the curves in $\dR^2$ to the curves in $\dR^n$.  In $\dR^3$ alone, it is highly likely that one would encounter curve segments of different types in addition to circular and straight line types.

A further extension from the Euclidean space to curved spaces would be of great interest, where the curves of minimax curvature would now lie in a Riemannian manifold.

\section*{Acknowledgments}
The authors would like to thank the two anonymous reviewers whose useful suggestions and insightful comments improved the manuscript.


\small
\begin{thebibliography}{30}

\bibitem{adams}
{\sc R.~A.~Adams, J.~J.~F.~Fournier},
{\em Sobolev Spaces},
Second Edition, Elsevier, 2003.
\url{https://shop.elsevier.com/books/sobolev-spaces/adams/978-0-12-044143-3}

\bibitem{Andreani2007}
{\sc R.~Andreani, E.~G.~Birgin, J.~M.~Mart{\'\i}nez, and M.~L.~Schuverdt}, 
{\em On augmented Lagrangian methods with general lower-level constraints}, 
SIAM J.~Optim., 18(4) (2008), pp. 1286--1309.
\url{https://doi.org/10.1137/060654797}

\bibitem{BelJac1975}
{\sc D.~J.~Bell and D.~H.~Jacobson},
{\em Singular Optimal Control Problems},
Academic Press, London, 1975.
\url{https://shop.elsevier.com/books/singular-optimal-control-problems/bell/978-0-12-085060-0}

\bibitem{BirMar2014}
{\sc E.~G.~Birgin and J.~M.~Mart{\'\i}nez},
{\em Practical Augmented Lagrangian Methods for Constrained Optimization},
SIAM Publications, 2014.
\url{https://doi.org/10.1137/1.9781611973365}

\bibitem{BoiCerLeb1991}
{\sc J.-D.~Boissonnat, A.~C{\'e}r{\'e}zo, and J.~Leblond},
{\em Shortest paths of bounded curvature in the plane. Plus courts chemins de courbure born\'ee dans le plan},
INRIA internal report, July 1991.
\url{https://inria.hal.science/inria-00075059/file/RR-1503.pdf}

\bibitem{BoiCerLeb1994}
{\sc J.-D.~Boissonnat, A.~C{\'e}r{\'e}zo, and J.~Leblond},
{\em Shortest paths of bounded curvature in the plane},
J. Intel. Robot. Sys, 11 (1994), 5--20.
\url{https://doi.org/10.1007/BF01258291}

\bibitem{Borisov2000}
{\sc V.~F.~Borisov},
{\em Fuller's phenomenon: review},
J. Math. Sci., 100(4) (2000), 2311--2354. \url{https://doi.org/10.1007/s10958-000-0001-9}

\bibitem{BuiBoiSouLau1994}
{\sc X.-N.~Bui, J.-D. Boissonnat, P. Sou\`eres, and J.-P. Laumond},
{\em Shortest path synthesis for Dubins non-holonomic robot},
Proc.\ 1994 IEEE Int.\ Conf.\ Rob.\ Autom., San Diego, CA, USA, May 1994,
2--7. 
\url{https://doi.org/10.1109/ROBOT.1994.351019}

\bibitem{Knitro}
{\sc R.~H.~Byrd, J.~Nocedal, R.~A.~Waltz}, 
{\em KNITRO: An integrated package for nonlinear optimization}. In: G.~di Pillo and M.~Roma, eds., Large-Scale Nonlinear Optimization, Springer, New York, 35--59, 2006.  \url{https://doi.org/10.1007/0-387-30065-1\_4}

\bibitem{Clarke2013}
{\sc F.~Clarke},
{\em Functional Analysis, Calculus of Variations and Optimal Control},
Springer-Verlag, London, 2013.
\url{https://doi.org/10.1007/978-1-4471-4820-3}

\bibitem{Dubins1957}
{\sc L.~E.~Dubins}, 
{\em On curves of minimal length with a constraint on average curvature and with prescribed initial and terminal positions and tangents},
Amer.~J.~Math., 79 (1957), pp. 497--516.
\url{https://doi.org/10.2307/2372560}

\bibitem{Filippov1962}
{\sc A.~F.~Filippov},
{\em On certain questions in the theory of optimal control},
J. SIAM Control, 1 (1962), 76--84.
\url{https://doi.org/10.1137/0301006}

\bibitem{AMPL}
{\sc R.~Fourer, D.~M.~Gay, and B.~W.~Kernighan},
{\em AMPL: A Modeling Language for Mathematical Programming, Second Edition},
Brooks/Cole Publishing Company / Cengage Learning, 2003.
\url{https://ampl.com/learn/ampl-book/}

\bibitem{gallagher2023elastica}
{\sc E.~Gallagher and R.~Moser},
{\em The $\infty$-elastica problem on a Riemannian manifold},
J. Geom. Anal., 33(7) (2023), 226.
\url{https://doi.org/10.1007/s12220-023-01281-2}

\bibitem{GilMurSau2005}
{\sc P.~E.~Gill, W.~Murray, and M.~A.~Saunders},
{\em SNOPT: an SQP algorithm for large-scale constrained optimization},
SIAM Rev., 47(1) (2005), 99--131.
\url{https://doi.org/10.1137/S1052623499350013}

\bibitem{Huang2004}
{\sc R.~Huang},
{\em A note on the $p$-elastica in a constant sectional curvature manifold},
J.\ Geom.\ Phys., 49(3--4) (2004), 343--349.
\url{https://doi.org/10.1016/S0393-0440(03)00107-4}

\bibitem{Kaya2017}
{\sc C.~Y.~Kaya},
{\em Markov--Dubins path via optimal control theory},
Comput.~Optim.~Appl., 68(3) (2017), 719--747.
\url{https://doi.org/10.1007/s10589-017-9923-8}

\bibitem{Kaya2019}
{\sc C.~Y.~Kaya},
{\em Markov--Dubins interpolating curves},
Comput.~Optim.~Appl., 73(2) (2019), 647--677.
\url{https://doi.org/10.1007/s10589-019-00076-y}

\bibitem{KayLucSim2004}
{\sc C.~Y.~Kaya, S.~K.~Lucas, and S.~T.~Simakov},
{\em Computations for bang–bang constrained optimal control using a mathematical programming formulation},
Opt.~Contr.~Appl.~Meth., 25 (2004), 295--308.
\url{https://doi.org/10.1002/oca.749}

\bibitem{KayNoa1996}
{\sc C.~Y.~Kaya and J.~L.~Noakes},
{\em Computations and time‐optimal controls},
Opt.\ Contr.\ Appl.\ Meth., 17 (1996), 171--185.
\url{https://doi.org/10.1002/(SICI)1099-1514(199607/09)17:3<171::AID-OCA571>3.0.CO;2-9}

\bibitem{KayNoa2003}
{\sc C.~Y.~Kaya and J.~L.~Noakes},
{\em Computational method for time-optimal switching control},
J.~Optim.~Theory App., 117 (2003), 69--92.
\url{https://doi.org/10.1023/A:1023600422807}

\bibitem{KayNoa2013}
{\sc C.~Y.~Kaya and J.~L.~Noakes},
{\em Finding interpolating curves minimizing $L^\infty$ acceleration in the Euclidean space via optimal control theory},
SIAM J.~Control Optim., 51(1) (2013), 442--464.
\url{https://doi.org/10.1137/12087880X}

\bibitem{KayNoaSch2025}
{\sc C.~Y.~Kaya, J.~L.~Noakes, and P.~Schrader},
{\em Curves of minimax curvature},
ESAIM Control Optim. Calc. Var., 32 (2026), Article no.\ 6, 1-24.
\url{https://doi.org/10.1051/cocv/2025095}


\bibitem{KreNud1977}
{\sc M.~G.~Kre{\u i}n and A.~A.~Nudel'man},
{\em The Markov Moment Problem and Extremal Problems},
Translations of Mathematical Monographs, Amer. Math. Soc., 1977.
\url{https://bookstore.ams.org/MMONO/50}

\bibitem{MacStr1982}
{\sc J.~Macki and A.~Strauss},
{\em Introduction to Optimal Control Theory},
Springer-Verlag, New York, 1982.
\url{https://doi.org/10.1007/978-1-4612-5671-7}

\bibitem{Markov1889}
{\sc A.~A.~Markov},
{\em Some examples of the solution of a special kind of problem on greatest and least quantities},
Soobscenija Charkovskogo Matematiceskogo Obscestva, 2-1(5,6),
250--276, 1889 (in Russian).

\bibitem{MauBueKimKay2005}
{\sc H.~Maurer, C.~B\"uskens, J.-H.~R.~Kim, and C.~Y.~Kaya},
{\em Optimization methods for the verification of second order sufficient conditions for bang--bang controls},
Opt.\ Cont.\ Appl.\ Meth., 26 (2005), 129--156.
\url{https://doi.org/10.1002/oca.756}

\bibitem{Melikyan2006}
{\sc A.~A.~Melikyan},
{\em Necessary optimality conditions for diﬀerent phase
portraits in a neighborhood of a singular arc},
Proc. Steklov Inst. Math., 255 (Suppl. 2) (2006), S126--S139. \url{https://doi.org/10.1134/S0081543806060113}

\bibitem{MiuYos2022}
{\sc T.~Miura and K.~Yoshizawa},
{\em Pinned planar p-elasticae}, 
arXiv preprint arXiv:2209.05721 (2022).
\url{https://doi.org/10.48550/arXiv.2209.05721}

\bibitem{moser2022}
{\sc R.~Moser},
{\em Structure and classification results for the $\infty$-elastica problem},
Amer.~J.~Math., 144(5) (2022), 1299--1329.
\url{https://doi.org/10.1353/ajm.2022.0030}

\bibitem{PonBolGamMis1962}
{\sc L.~S.~Pontryagin, V.~G.~Boltyanskii, R.~V.~Gamkrelidze, and E.~F.~Mishchenko}, 
{\em The Mathematical Theory of Optimal Processes} (Russian), 
English translation by K.~N.~Trirogoff, ed. by L.~W.~Neustadt,
Interscience Publishers, New York, 1962.
\url{https://doi.org/10.1201/9780203749319}

\bibitem{ReeShe1990}
{\sc J.~A.~Reeds and L.~A.~Shepp},
{\em Optimal paths for a car that goes both forwards and backwards}, 
Pacific J. Math., 145 (1990), 367--393.
\url{https://doi.org/10.2140/pjm.1990.145.367}

\bibitem{SchLed2012}
{\sc H.~Sch\"attler and U.~Ledzewicz},
{\em Geometric Optimal Control: Theory, Methods and Examples},
Springer, New York, 2012.
\url{https://doi.org/10.1007/978-1-4614-3834-2}

\bibitem{SouLau1996}
{\sc P.~Sou\`eres and J.-P.~Laumond},
{\em Shortest paths synthesis for a car-like robot},
IEEE Trans. Autom. Control, 41 (1996), 672--688.
\url{https://doi.org/10.1109/9.489204}

\bibitem{Sussmann1995}
{\sc H.~J.~Sussmann},
{\em Shortest 3-dimensional paths with a prescribed curvature bound},
Proc.\ 34th IEEE Conf.\ Dec.\ Contr., New Orleans, LA, USA, December 1995, 3306--3312. 
\url{https://doi.org/10.1109/CDC.1995.478997}

\bibitem{Sussmann1997}
{\sc H.~J.~Sussmann},
{\em The Markov-Dubins problem with angular acceleration control},
Proc.\ 36th IEEE Conf.\ Dec.\ Contr., San Diego, CA, USA, December 1997, 2639--2643.
\url{https://doi.org/10.1109/CDC.1997.657778}

\bibitem{SusTan1991}
{\sc H.~J.~Sussmann and G.~Tang},
{\em Shortest paths for the Reeds-Shepp car: a worked out example of the use of geometric techniques in nonlinear optimal control},
Rutgers Center for Systems and Control (Sycon) Report 91-10, September 1991.

\bibitem{Trelat2024}
{\sc E.~Tr{\'e}lat},
{\em Control in Finite and Infinite Dimension},
Springer Nature, Singapore, 2024.
\url{https://doi.org/10.1007/978-981-97-5948-4}

\bibitem{WacBie2006}
{\sc A.~W\"achter and L.~T.~Biegler},
{\em On the implementation of a primal-dual interior point filter line search algorithm for large-scale nonlinear programming},
Math.~Program., 106 (2006), 25--57.
\url{https://doi.org/10.1007/s10107-004-0559-y}

\bibitem{yosida}
{\sc K.~Yosida},
{\em Functional Analysis},
Fifth Edition, Springer-Verlag, 1978.

\end{thebibliography}
\end{document}